\documentclass[a4paper,11pt]{amsart}
\usepackage{mathptmx}
\usepackage{amsmath}
\usepackage{amscd}
\usepackage{amssymb, bm}
\usepackage{amsthm}
\usepackage{xspace}
\usepackage[all,tips]{xy}
\usepackage[dvips]{graphicx}
\usepackage{verbatim}
\usepackage{syntonly}
\usepackage{hyperref}
\usepackage{amsmath, amsthm, graphics, amssymb,fullpage,color, epsfig,url}
\usepackage{indentfirst}
\usepackage{esint}
\usepackage{enumitem}
\usepackage[dvipsnames]{xcolor}
\usepackage{soul}

\providecommand{\MR}{\relax\ifhmode\unskip\space\fi MR }

\providecommand{\href}[2]{#2}

\theoremstyle{plain}
\newtheorem{thm}{Theorem}[section]
\newtheorem{lem}[thm]{Lemma}
\newtheorem{prop}[thm]{Proposition}
\newtheorem{defn}[thm]{Definition}
\newtheorem{cor}[thm]{Corollary}

\theoremstyle{remark}
\newtheorem{rem}[thm]{Remark}

\newenvironment{pf}
{\begin{proof}} {\end{proof}}

\newcommand{\disp}{\displaystyle}

\DeclareMathOperator{\dist}{dist}

\DeclareMathOperator{\osc}{osc}

\DeclareMathOperator{\supp}{supp}
\DeclareMathOperator{\di}{div}

\newcommand{\eps}{\varepsilon}
\newcommand{\vp}{\varphi}

\newcommand{\al}{\alpha}
\newcommand{\be}{\beta}
\newcommand{\ga}{\gamma}
\newcommand{\de}{\delta}

\newcommand{\Ga}{\Gamma}
\newcommand{\te}{\theta}
\newcommand{\la}{\lambda}
\newcommand{\La}{\Lambda}
\newcommand{\om}{\omega}
\newcommand{\Om}{\Omega}
\newcommand{\si}{\sigma}

\newcommand{\iny}{\infty}
\newcommand{\del}{ \partial}
\newcommand{\su}{\subset}

\newcommand{\bb}{{\bf b}}
\newcommand{\bd}{{\bf d}}
\newcommand{\bu}{{\bf u}}
\newcommand{\bv}{{\bf v}}
\newcommand{\bw}{{\bf w}}

\newcommand{\bV}{{\bf V}}
\newcommand{\bA}{{\bf A}}
\newcommand{\bG}{\pmb{ \Ga}}
\newcommand{\bff}{{\bf f}}
\newcommand{\bg}{{\bf g}}
\newcommand{\bF}{{\bf F}}

\newcommand{\bz}{{\bf 0}}

\newcommand{\norm}[1]{\left\vert \left\vert #1\right\vert\right\vert}
\newcommand{\innp}[1]{\left< #1 \right>}
\newcommand{\abs}[1]{\left\vert#1\right\vert}

\newcommand{\set}[1]{\left\{#1\right\}}
\newcommand{\brac}[1]{\left[#1\right]}
\newcommand{\pr}[1]{\left( #1 \right) }
\newcommand{\pb}[1]{\left( #1 \right] }
\newcommand{\brp}[1]{\left[ #1 \right) }

\newcommand{\N}{\ensuremath{\mathbb{N}}}

\newcommand{\R}{\ensuremath{\mathbb{R}}}

\newcommand{\bGr}{{\bf G}}
\newcommand{\cd}{\cdot}

\newcommand{\iom}{\int_{\Omega}}
\newcommand{\eq}[1]{\begin{equation}#1\end{equation}}
\newcommand{\eqs}[1]{\begin{equation*}#1\end{equation*}}
\newcommand{\aln}[1]{\begin{align}#1\end{align}}
\newcommand{\alns}[1]{\begin{align*}#1\end{align*}}

\newcommand{\rn}{\mathbb{R}^n}

\definecolor{LightGreen}{RGB}{150,200,30}

\usepackage{pgf}
\usepackage{color}

\calclayout
\allowdisplaybreaks

\newcommand{\loc}{\operatorname{loc}}

\newcommand{\RR}{{\mathbb{R}}}
\newcommand{\NN}{{\mathbb{N}}}

\newcommand{\po}{{\partial\Omega}}

\setstcolor{violet}

\numberwithin{equation}{section}
\date{}
\begin{document}

\title{Fundamental Matrices and Green Matrices for non-homogeneous elliptic systems}

\author[Davey]{Blair Davey}
\address{Department of Mathematics, City College of New York CUNY, New York, NY 10031, USA}
\email{bdavey@ccny.cuny.edu}
\author[Hill]{Jonathan Hill}
\address{School of Mathematics, University of Minnesota, Minneapolis, MN 55455, USA}
\email{hill0631@umn.edu}
\author[Mayboroda]{Svitlana Mayboroda}
\address{School of Mathematics, University of Minnesota, Minneapolis, MN 55455, USA}
\email{svitlana@math.umn.edu}
\thanks{Davey is supported in part by the Simons Foundation Award 430198. \\
Mayboroda is supported in part by the Alfred P. Sloan Fellowship, the NSF INSPIRE Award DMS 1344235, NSF CAREER Award DMS 1220089 and  NSF UMN MRSEC Seed grant DMR 0212302.}

\begin{abstract} 
In this paper, we establish existence, uniqueness, and scale-invariant estimates for fundamental solutions of non-homogeneous second order elliptic systems with bounded measurable coefficients in $\RR^n$ and for the corresponding Green functions in arbitrary open sets. 
We impose certain non-homogeneous versions of de Giorgi-Nash-Moser bounds on the weak solutions and investigate in detail the assumptions on the lower order terms sufficient to guarantee such conditions. 
\end{abstract}

\maketitle

\tableofcontents

\section{Introduction}

In this paper, we consider non-homogeneous second order uniformly elliptic systems, formally given by ${\mathcal L} \bu =-D_\al\pr{\bA^{\al \be} D_\be \bu + \bb^\al \bu} + \bd^\be D_\be \bu + \bV \bu $. 
The principal term, $L := - D_\al \bA^{\al \be} D_\be$, satisfies the following uniform ellipticity and boundedness conditions
\begin{align*}
&A^{\al \be}_{ij}\pr{x} \xi_{\be}^j \xi_{\al}^i \ge \la \abs{\bf \xi}^2 
:= \la \sum_{i = 1}^N \sum_{\al = 1}^n \abs{\xi_{\al}^i}^2 \\
& \sum_{i, j = 1}^N \sum_{\al, \be = 1}^n \abs{A_{ij}^{\al \be}\pr{x}}^2 \le \La^2,
\end{align*}
for some $0 < \la, \La < \iny$ and for all $x$ in the domain.
Note that, in particular, equations with complex bounded measurable coefficients fit into this scheme.
We establish existence, uniqueness, as well as global scale-invariant estimates for the fundamental solution in $\RR^n$ and for the Dirichlet Green function in any connected, open set $\Omega\subset \RR^n$, where $n\geq 3$. 
The key difficulty in our work is the lack of homogeneity of the system since this typically results in a lack of scale-invariant bounds.
Here, the existence of solutions relies on a coercivity assumption, which controls the lower-order terms, and the validity of the Caccioppoli inequality. 
Furthermore, following many predecessors (see, e.g., \cite{HK07}, \cite{KK10}), we require certain quantitative versions of the local boundedness of solutions. 
This turns out to be a delicate game, however, to impose local conditions which are sufficient for the construction of fundamental solutions and necessary for most prominent examples. 
Indeed, they have not been completely well-understood even in the case of real {\it equations}, due to the same type of difficulties: 
Solutions to non-homogeneous equations can grow exponentially with the growth of the domain in the absence of a suitable control on the potential $\bV$, even if $\bb=\bd= \bf 0$.
This affects the construction of the fundamental solution. 
Let us discuss the details. 

The fundamental solutions and Green functions for {\it homogeneous} second order elliptic systems are fairly well-understood by now. 
We do not aim to review the vast literature addressing various situations with additional smoothness assumptions on the coefficients of the operator and/or the domain, and will rather comment on those works that are most closely related to ours.  
The analysis of Green functions for operators with bounded measurable coefficients goes back to the early 80's, \cite{GW82} (see also \cite{LSW63} for symmetric operators), in the case of  homogeneous equations with real  coefficients ($N=1$). The case of homogeneous {\it systems}, and, respectively, equations with complex coefficients, has been treated much more recently in \cite{HK07} and \cite{KK10}  under the assumptions of local boundedness and H\"older continuity of solutions, the so-called de Giorgi-Nash-Moser estimates. 
Later on, in \cite{Ros13}, the fundamental solution in $\RR^n$ was constructed using only the assumption of local boundedness, that is, without the requirement of H\"older continuity.
In \cite{Bar14}, Barton constructed fundamental solutions, also in $\RR^n$ only, in the full generality of homogeneous elliptic systems without assuming any de Giorgi-Nash-Moser estimates.
The techniques in \cite{Bar14} are based on descent from the higher order case. 

The present paper can be split into two big portions. 
In the first part, we prove that one can define the fundamental solution and the Green function, and establish global estimates on par with the aforementioned works for homogeneous equations, roughly speaking, if:
\begin{enumerate}
\item The bilinear form associated to ${\mathcal L}$ is coercive and bounded in a suitable Hilbert space.
\item The Caccioppoli inequality holds:

If $\bu$ is a weak solution to $\mathcal{L} \bu = \bf 0$ in $U \su \Om$ and $\zeta$ is a smooth cutoff  function, then 
$$\int \abs{D \bu}^2 \zeta^2 \le C \int \abs{\bu}^2 \abs{D \zeta}^2,$$
where $C$ is independent of the subdomain $U$.

\item The interior scale-invariant Moser bounds hold:

If $\bu$ is a weak solution to $\mathcal{L} \bu = \bff$ in $B_R\subset \Omega$, for some $R>0$, where $\bff \in L^\ell\pr{B_R}^N$ for some $\ell \in \pb{\frac{n}{2}, \iny}$, then for any $q > 0$,
$$ \sup_{B_{R/2}} \abs{\bu} \le C \brac{ \pr{\fint_{B_R} \abs{\bu}^q }^{1/q} + R^{2 - \frac{n}\ell} \norm{\bff}_{L^\ell\pr{B_R}}},$$
where $C$ is independent of $R$.

\item The solutions are H\"older continuous: 

If $\bu$ is a weak solution to $\mathcal{L} \bu = \bf 0$ in $B_{R_0} \subset \Omega$, for some $R_0>0$, then there exists $\eta \in \pr{0, 1}$, depending on $R_0$, and $C_{R_0}>0$ so that whenever $0 < R \le R_0$,
$$\sup_{x, y \in B_{R/2}, x \ne y} \frac{\abs{\bu\pr{x} - \bu\pr{y}}}{\abs{x - y}^\eta}
\le C_{R_0} R^{-\eta} \pr{\fint_{B_{R}} \abs{\bu}^{2^*} }^{1/{2^*}} .$$
\end{enumerate} 

If, in addition, the boundary scale-invariant Moser bounds hold (that is, the Moser estimate holds for solutions with trace zero on balls possibly intersecting the boundary), then the Green functions exhibit respectively stronger boundary estimates. 
This part of the paper is modeled upon the work in \cite{HK07} and \cite{KK10}.
However, the scaling issues and identifying the exact form of necessary conditions that are compatible with the principal non-homogeneous examples make our arguments considerably more delicate. 
Note, in particular, the local nature of H\"older estimates versus the global nature of Moser-type bounds. 
The Moser-type bounds are independent of the domain, whereas the H\"older estimates may depend on the size of the ball. 


In the second portion of the paper, we motivate the assumptions from above by showing that conditions (1)--(4) above are valid in the following three situations. 
To be precise, we show that in each case listed below, (1)--(2) from above hold for the general systems, while (3)--(4) holds for {\it equations} and, hence, the resulting estimates on fundamental solutions and Green functions are valid for the {\it equations} with real coefficients in each of the three cases below.

\begin{enumerate}
\item[Case 1.] {\em Homogeneous operators}: $\bb, \bd, \bV \equiv \bf 0$ and the function space for solutions is $\bF\pr{\Om} = Y^{1,2}\pr{\Om}^N$. 
Here, $Y^{1,2}\pr{\Om}$ is the family of all weakly differentiable functions $u \in L^{2^*}\pr{\Om}$, with $2^* = \frac{2n}{n-2}$, whose weak derivatives are functions in $L^2\pr{\Om}$. 
\item[Case 2.] {\em  Lower order coefficients in $L^p$}: 
There exist $p \in \pb{ \frac n 2, \iny}$, $s, t \in \pb{ n, \iny}$ so that $\bV  \in L^p\pr{\Om}^{N \times N}$, $\bb \in L^s\pr{\Om}^{n \times N \times N}$, $\bd \in L^t\pr{\Om}^{n \times N \times N}$ and we take the function space for solutions to be $\bF\pr{\Om} = W^{1,2}\pr{\Om}^N$. As usual, 
$W^{1,2}\pr{\Om}$ is the family of all weakly differentiable functions $u \in L^{2}\pr{\Om}$ whose weak derivatives are functions in $L^2\pr{\Om}$.
The lower-order terms are chosen so that the bilinear form associated to $\mathcal{L}$ is coercive.
For conditions (3)-(4), we assume further that  $\bV - \di \bb \ge 0$ and $\bV - \di \bd \ge 0$ in the sense of distributions.
\item[Case 3:] {\em Reverse H\"older potentials}: $\bV \in B_p$, the reverse H\"older class, for some $p \in \brp{ \frac n 2, \iny}$, $\bb, \bd \equiv \bf 0$, and $\bF\pr{\Om} = W^{1,2}_{ V}\pr{\Om}^N$, a weighted Sobolev space (with the weight given by a certain maximal function associated to $\bV$ -- see \eqref{eq1.1} and definitions in the body of the paper).
\end{enumerate}

We would like to point out that in Theorem 18 from \cite{AT98}, P.\, Auscher and Ph.\, Tchamitchian state the global Gaussian bounds on the heat semigroup under the assumption of $W^{1,2}(\RR^n)$ coercivity of the corresponding form, for $b,d,V\in L^\infty(\RR^n)$, without any additional non-degeneracy condition. 
This is a version of our Case 2. 
Such estimates should, in principle, imply a global pointwise estimate on the fundamental solution in $\RR^n$ of the form $|\Gamma(x, y)|\leq C |x-y|^{2-n}$, for all $x, y\in \RR^n$, $x \ne y$. 
A similar result could be obtained in Case 3 by the maximum principle. 
It is not immediately clear, however, if in either case one can obtain a complete package of results that we have targeted (see Theorems \ref{t3.6} and \ref{t3.10}), particularly for the Green functions on domains. 
For those reasons, we did not pursue this route in the present work. 
More generally, one can sometimes establish bounds on the fundamental solutions and Green functions for elliptic boundary problems by an integration of the estimates of the corresponding heat kernels. 
However, the latter requires a suitable form of uniform exponential decay of the heat kernel in $t>0$, while the non-homogeneous equations typically give rise to bounds for a finite time, $0<t<T$, with a constant depending on $T$ (cf., e.g., \cite{Dav95}, \cite{Aro68}, \cite{AQ00}, \cite{Ouh05}). 
There are notable exceptions to this rule, including \cite{AT98}, but they do not provide a basis for a unified theory, particularly on general domains. 

The verification of local bounds and H\"older continuity in our arguments follows a traditional route (see \cite{GT01}, \cite{HL11}, \cite{Sta65}).
However, we have to carefully adjust the arguments so that the dependence on constants coincides with our constructions of fundamental solutions.

Going further, let us say a few words about Case 3. 
This is the version of the Schr\"odinger equation that initially interested us. 
With pointwise bounds on the fundamental solution and the Green function (Theorems \ref{t3.6} and \ref{t3.10}, respectively), as well as basic Moser, H\"older, Harnack estimates established in our present work, one can now move on to derive the sharp exponential decay of the fundamental solutions in terms of the Agmon distance associated to the maximal function
\begin{equation}\label{eq1.1} 
m\pr{x, V} = \left(\sup_{r > 0} \set{ r : \psi\pr{x,r; V} \le 1}\right)^{-1}.
\end{equation}
For instance, it is natural to expect that in Case 3
$$\Gamma(x,y)\leq C\,\frac{e^{-\eps\, d(x,y,V)}}{|x-y|^{n-2}}$$
for some $C,\eps>0$, with the distance function
$$
d(x,y,V)=\inf_{\gamma}\int\limits_0^1m(\gamma(t),V)|\gamma'(t)|\,dt,
$$
where $\gamma:[0,1]\to\RR^n$ is absolutely continuous, $\gamma(0)=x, \gamma(1)=y$, and $m$ is the Fefferman-Phong maximal function. 
This question will be addressed in the upcoming work \cite{MP16}, along with the corresponding estimates from below.
See \cite{She99} for the case of $-\Delta+V$. 

\section{Basic assumptions and notation}
\label{s2}

Throughout this article, the summation convention will be used.
Let $n \ge 3$ denote the dimension of the space, and let $N \ge 1$ denote the number of components in each vector function.
Let $\Om \su \R^n$ be an open, connected set.
We use the notation $B_r\pr{x}$ to denote a ball of radius $r > 0$ centered at $x \in \R^n$, and the abbreviated notation $B_r$ when $x$ is clear from the context.
For any $x \in \Om$, $r > 0$, we define $\Om_r(x):=\Om \cap B_r(x)$.
Let $C^\iny_c\pr{\Om}$ denote the set of all infinitely differentiable functions with compact support in $\Om$.
We set $2^*=\frac{2n}{n-2}$.

For any open set $\Om \su \R^n$, define the space $Y^{1,2}\pr{\Om}$ as the family of all weakly differentiable functions $u \in L^{2^*}\pr{\Om}$ whose weak derivatives are functions in $L^2\pr{\Om}$.
The space $Y^{1,2}\pr{\Om}$ is endowed with the norm
\begin{align}
\norm{u}^2_{Y^{1,2}\pr{\Om}} :=  \norm{u}^2_{L^{2^*}\pr{\Om}} + \norm{D u}^2_{L^2\pr{\Om}}.
\label{eq2.1}
\end{align}
Define $Y^{1,2}_0\pr{\Om}$ as the closure of $C^\iny_c\pr{\Om}$ in $Y^{1,2}\pr{\Om}$. 
When $\Om = \R^n$, $Y^{1,2}\pr{\R^n} = Y^{1,2}_0\pr{\R^n}$ (see, e.g.,  Appendix \ref{AppA}).
By the Sobolev inequality,
\begin{equation}\label{eq2.2}
\norm{u}_{L^{2^*}\pr{\Om}} 
\le c_n \norm{D u}_{L^2\pr{\Om}} \quad \text{for all $u \in Y^{1,2}_0\pr{\Om}$.} 
\end{equation}
It follows that $W^{1,2}_0\pr{\Om} \su Y^{1,2}_0\pr{\Om}$ with set equality when $\Om$ has finite measure. 
Here, $W^{1,2}\pr{\Om}$ is the family of all weakly differentiable functions $u \in L^{2}\pr{\Om}$ whose weak derivatives are functions in $L^2\pr{\Om}$.
The norm on $W^{1,2}\pr{\Om}$ is given by
\eqs{
\norm{u}_{W^{1,2}\pr{\Om}} = \norm{u}_{L^{2}\pr{\Om}} + \norm{D u}_{L^2\pr{\Om}},
}
and $W^{1,2}_0\pr{\Om}$ is the closure of $C^\iny_c\pr{\Om}$ in $W^{1,2}\pr{\Om}$. 
We shall mostly be talking about the spaces of vector-valued functions in $Y_0^{1,2}\pr{\Om}^N$. 
The bilinear form
\begin{align}
\innp{\bu, \bv}_{Y_0^{1,2}\pr{\Om}^N} := \int_{\Om} D_\al u^i D_\al v^i
\label{eq2.3}
\end{align}
defines an inner product on $Y_0^{1,2}\pr{\Om}^N$.
With this inner product, $Y_0^{1,2}\pr{\Om}^N$ is a Hilbert space with norm 
\begin{align*}
\norm{\bu}_{Y_0^{1,2}\pr{\Om}^N} 
:= \innp{\bu, \bu}_{Y_0^{1,2}\pr{\Om}^N}^{1/2} 
= \norm{D \bu}_{L^2\pr{\Om}^N}.
\end{align*}
For the sake of brevity, we sometimes drop the superscript of dimension from the norm notation when it is understood from the context.
For further properties of $Y^{1,2}\pr{\Om}$, and some relationships between $Y^{1,2}\pr{\Om}$ and $W^{1,2}\pr{\Om}$, we refer the reader to Appendix \ref{AppA}.

Hofmann and Kim used the space $Y^{1,2}\pr{\Om}^N$ in their constructions of fundamental matrices and Green matrices for homogeneous operators \cite{HK07}.  
Since we are concerned with non-homogeneous operators, this function space will not always be appropriate, but we intend to mimic some of its properties. 
To this end, we will define the pair consisting of a non-homogeneous elliptic operator and a suitably accompanying Banach space, and then show that standard cases of interest fit in this framework. 

We assume that for any $\Omega\subset \RR^n$ open and connected, there exists a Banach space $\bF\pr{\Om}$ consisting of weakly differentiable, vector-valued  $L^1_{\loc}\pr{\Om}$ functions that satisfy the following properties:

\begin{enumerate}[label=A\arabic*)]
\item\label{A1}
Whenever $U\subset \Omega$, 
\eq{\label{eq2.4}
\bu \in \bF\pr{\Om} \,\to \,  \bu|_U \in \bF(U), \quad \text{with $\left\|\bu|_U\right\|_{\bF(U)} \le \left\|\bu\right\|_{\bF\pr{\Om}}.$}
}

\item\label{A2}  $C_c^\infty\pr{\Om}^N$ functions belong to $\bF\pr{\Om}$.
The space $\bF_0\pr{\Om}$, defined as the closure of $C_c^\infty\pr{\Om}^N$ with respect to the $\bF\pr{\Om}$-norm, is a Hilbert space with respect to some $\|\cdot\|_{\bF_0(\Omega)}$ such that 
$$\|u\|_{\bF_0(\Omega)} \approx \|u\|_{\bF(\Omega)} \quad \mbox{ for all $u\in \bF_0(\Omega)$}.$$

\item\label{A3}  The space ${\bF}_0\pr{\Om}$ is continuously embedded into 
$Y^{1,2}_0\pr{\Om}^N$ and respectively, there exists $c_0>0$ such that for any $\bu \in \bF_0\pr{\Om}$ 
\begin{align}
\norm{\bu}_{Y^{1,2}_0{\pr{\Om}^N}} \le c_0 \norm{\bu}_{\bF\pr{\Om}}.
\label{eq2.5}
\end{align}
Note that this embedding and \eqref{eq2.2} imply a homogeneous Sobolev inequality in $\bF_0\pr{\Om}$:
\begin{equation}
\norm{\bu}_{L^{2^*}\pr{U}} \lesssim \norm{D\bu}_{L^2\pr{U}} \mbox{ for any }\bu \in \bF_0\pr{\Omega},
\label{eq2.6}
\end{equation}
which will be used repeatedly throughout.

\item\label{A4}
For any $U\subset \rn$ open and connected
\begin{equation}
\label{eq2.7}
\begin{array}{c}
\mbox{ $\bu\in \bF\pr{\Om}$ and $\xi\in C_c^\infty(U) \quad \Longrightarrow \quad \bu \xi\in \bF(\Om \cap U)$,} \\
\mbox{ $\bu\in \bF\pr{\Om}$ and $\xi\in C_c^\infty(\Om \cap U) \quad \Longrightarrow \quad \bu \xi\in \bF_0(\Om \cap U)$,}
\end{array}
\end{equation}
with $\|\bu \xi\|_{\bF(\Om\cap U)}\leq C_\xi \, \|\bu\|_{\bF(\Om)}.$

It follows, in particular, that
\eq{\label{eq2.8}
\bF\pr{\Omega} \su Y^{1,2}_{\loc}\pr{\Omega}^N.}
Indeed, for any $x\in \Omega$ there exists a ball $B_r(x)\subset \Omega$. If $\bu\in \bF(\Omega)$ and $\xi\in C_c^\infty (B_r)$, we have $\bu \xi\in {\bF}_0\pr{\Om} \hookrightarrow Y^{1,2}_0{\pr{\Om}^N}$. 
Hence, taking $\xi \equiv 1$ on $B_{r/2}(x),$ we conclude that $\bu\in Y^{1,2}\pr{B_{r/2}\pr{x}}^N$.

Another consequence of \eqref{eq2.7} is that for any $U\subset \rn$ open and connected
\begin{equation}
\label{eq2.9}
\mbox{ if $\bu\in \bF_0\pr{\Om}$ and $\xi\in C_c^\infty(U) \quad \Longrightarrow \quad \bu \xi\in \bF_0(\Om \cap U)$.}
 \end{equation}
Indeed, if $\bu\in \bF_0\pr{\Om}$ then there exists a sequence  $\set{\bu_n}_{n \in \N} \su C_c^\infty(\Omega)$ which converges to $\bu$ in $\bF(\Omega)$. 
But then $\set{\xi \bu_n}_{n\in\N} \su C_c^\infty(\Omega\cap U)$ is Cauchy in $\bF(\Omega\cap U)$ and in $Y^{1,2}(\Omega\cap U)^N$ by \eqref{eq2.7} and \eqref{eq2.5}. 
Therefore, it converges in $\bF(\Omega\cap U)$  and in $Y^{1,2}(\Omega\cap U)^N$ to some element of $\bF_0(\Omega\cap U)\hookrightarrow Y^{1,2}_0(\Omega\cap U)^N$, call it $\bv$. 
And it follows that $\bv=\bu\xi$ as elements of $Y^{1,2}_0(\Omega\cap U)^N$.

\end{enumerate}

For future reference, we mention that for $\Omega, U\subset \rn$ open and connected, the assumption 
\begin{equation}\label{eq2.10}
\bu \in \bF(\Omega), \quad \bu={\bf 0} \mbox{ on } U\cap \po, 
\end{equation}
is always meant in the weak sense of 
\begin{equation}\label{eq2.11}
\bu \in \bF(\Omega) \mbox{ and } \bu \xi \in \bF_0(\Omega) \mbox{ for any } \xi \in C_c^\infty(U).
\end{equation}
This definition of (weakly) vanishing on the boundary is independent of the choice of $U$.  Indeed, suppose $V$ is another open and connected subset of $\R^n$ such that $V \cap \partial \Om = U \cap \partial \Om$ and let $\xi \in C_c^\infty\pr{V}$.
Choose $\psi \in C_c^\infty\pr{U \cap V}$ such that $0\le \psi \le 1$ and $\psi \equiv 1$ on the support of $\xi$ in some neighborhood of the boundary.
Then $\xi \pr{1-\psi}|_{\Om}\in C_c^\infty\pr{\Om}$, so by \eqref{eq2.7} we have $\bu \xi \pr{1-\psi} \in \bF_0\pr{\Om}$.
Additionally, $\xi \psi \in C_c^\infty\pr{U}$, so by \eqref{eq2.11}, $\bu \xi \psi \in \bF_0 \pr{\Om}$.
Therefore, $\bu \xi = \bu \xi \psi + \bu \xi \pr{1-\psi} \in \bF_0\pr{\Om}$, as desired.

Before stating the remaining properties of $\bF\pr{\Om}$, we define the elliptic operator. 
Let $\bA^{\al \be} = \bA^{\al \be}\pr{x}$, $\al, \be \in \set{ 1, \dots, n}$, be an $N \times N$ matrix with bounded measurable coefficients defined on $\Om$.
We assume that $\bA^{\al \be}$ satisfies uniform ellipticity and boundedness conditions: 
\begin{align}
&A^{\al \be}_{ij}\pr{x} \xi_{\be}^j \xi_{\al}^i \ge \la \abs{\bf \xi}^2 
:= \la \sum_{i = 1}^N \sum_{\al = 1}^n \abs{\xi_{\al}^i}^2
\label{eq2.12}\\
& \sum_{i, j = 1}^N \sum_{\al, \be = 1}^n \abs{A_{ij}^{\al \be}\pr{x}}^2 \le \La^2,
\label{eq2.13}
\end{align}
for some $0 < \la, \La < \iny$ and for all $x \in \Om$. 
Let $\bV$ denote the zeroth order term, an $N \times N$ matrix defined on $\Om$.
The first order terms, denoted by $\bb^\al$ and $\bd^\be$, for each $\al, \be \in \set{1, \ldots, n}$, are $N \times N$ matrices defined on $\Om$.
We assume that there exist $p \in \pb{\frac{n}{2}, \iny}$ and $s, \, t \in \pb{n, \iny}$ such that
\eq{\label{eq2.14} \bV\in L^p_{\loc}\pr{\Om}^{N \times N}, \,\bb\in L^s_{\loc}\pr{\Om}^{n \times N \times N}, \, \bd\in L^t_{\loc}\pr{\Om}^{n \times N \times N}.}

We now formally fix the notation and then we will discuss the proper meaning of the operators at hand. For every $\bu = \pr{u^1,\ldots, u^N }^T$ in $\bF_{\loc}\pr{\Om}$ (and hence, in $Y^{1,2}_{\loc}\pr{\Om}^N$) we define\begin{equation}
L \bu = - D_\al \pr{\bA^{\al \be} D_\be \bu}.
\label{eq2.15}
\end{equation}
If we write out \eqref{eq2.15} component-wise, we have
\begin{align*}
\pr{L \bu}^i = - D_\al \pr{A_{ij}^{\al \be} D_\be u^j}, \;\;  \text{ for each } i = 1, \ldots, N.
\end{align*}

The non-homogeneous second-order operator is written as
\begin{align}
\mathcal{L} \bu &:= L \bu - D_\al \pr{\bb^\al \bu} + \bd^\be D_\be \bu + \bV \bu \nonumber \\
&= -D_\al\pr{\bA^{\al \be} D_\be \bu + \bb^\al \bu} + \bd^\be D_\be \bu + \bV \bu,
\label{eq2.16}
\end{align}
or, component-wise, 
\begin{align*}
\pr{ \mathcal{L} \bu}^i
&= -D_\al\pr{A_{ij}^{\al \be} D_\be u^j + b_{ij}^\al u^j} + d_{ij}^\be D_\be u^j + V_{ij} u^j, \;\; \text{ for each } i = 1, \ldots, N.
\end{align*}

The transpose operator of $L$, denoted by $L^*$, is defined by
\begin{align*}
L^* \bu = - D_\al \brac{\pr{\bA^{\al\be}}^* D_\be \bu},
\end{align*}
where $\pr{\bA^{\al \be}}^* = \pr{\bA^{\be\al}}^T$, or rather $\pr{A_{ij}^{\al\be}}^* = A_{ji}^{\be\al}$.
Note that the adjoint coefficients, $\pr{A_{ij}^{\al\be}}^*$ satisfy the same ellipticity assumptions as $A_{ij}^{\al\be}$ given by \eqref{eq2.12} and \eqref{eq2.13}.
Take $\pr{\bb^{\al}}^* = \pr{\bd^{\al}}^T$, $\pr{\bd^{\be}}^* = \pr{\bb^{\be}}^T$, and $\bV^* = \bV^T$.
The adjoint operator to $\mathcal{L}$ is given by
\begin{align}
\mathcal{L}^* \bu 
&:=\,  L^* \bu - D_\al \brac{\pr{\bb^\al}^* \bu} + \pr{\bd^\be}^* D_\be \bu + \bV^* \bu \nonumber \\
&= -D_\al\brac{\pr{\bA^{\be\al}}^T D_\be \bu + \pr{\bd^\al}^T \bu} + \pr{\bb^\be}^T D_\be \bu + \bV^T \bu,
\label{eq2.17}
\end{align}
or
\begin{align*}
\pr{\mathcal{L}^* \bu}^i 
&= -D_\al\pr{A_{ji}^{\be \al} D_\be u^j + d_{ji}^\al u^j} + b_{ji}^\be D_\be u^j + V_{ji} u^j, \;\; \text{ for each } i = 1, \ldots, N.
\end{align*}

All operators, $L, L^*, \mathcal{L}, \mathcal{L}^*$ are understood in the sense of distributions on $\Omega$. 
Specifically, for every $\bu \in Y^{1,2}_{\loc}\pr{\Om}^N$ and $\bv\in C_c^\infty\pr{\Om}^N$, we use the naturally associated bilinear form and write the action of the functional ${\mathcal L}\bu$ on $\bv$ as
\begin{align}
({\mathcal L}\bu, \bv)=\mathcal{B}\brac{\bu, \bv} 
&= \int_\Om \bA^{\al \be} D_\be \bu \cdot D_\al \bv + \bb^\al \, \bu \cdot D_\al \bv + \bd^\be D_\be \bu \cdot \bv + \bV \, \bu \cdot \bv
\nonumber \\
&= \int_\Om A_{ij}^{\al \be} D_\be u^j D_\al v^i + b_{ij}^\al u^j D_\al v^i + d_{ij}^\be D_\be u^j v^i + V_{ij} u^j v^i.
\label{eq2.18}
\end{align}
It is not hard to check that for such $\bv, \bu$ and for the coefficients satisfying \eqref{eq2.13}, \eqref{eq2.14}, the bilinear form above is well-defined and finite.  
We often drop the $\Om$ from the subscript on the integral when it is understood.
Similarly, $\mathcal{B}^*\brac{\cd, \cd}$ denotes the bilinear operator associated to $\mathcal{L}^*$, given by
\begin{align}
({\mathcal L}^*\bu, \bv)=\mathcal{B}^*\brac{\bu, \bv} 
&= \int  \pr{\bA^{\be \al}}^T D_\be \bu \cdot D_\al \bv + \pr{\bd^\al}^T \bu \cdot D_\al \bv + \pr{\bb^\be}^T D_\be \bu \cdot \bv + \bV^T \, \bu \cdot \bv 
\nonumber \\
&= \int A_{ji}^{\be \al} D_\be u^j D_\al v^i + d_{ji}^\al u^j D_\al v^i + b_{ji}^\be D_\be u^j v^i + V_{ji} u^j v^i .
\label{eq2.19}
\end{align}
Clearly,
\begin{equation}\label{eq2.20}
\mathcal{B}\brac{\bv,\bu}=\, \mathcal{B}^*\brac{\bu,\bv}.
\end{equation}
For any vector distribution $\bff$ on $\Omega$ and $\bu$ as above we always understand ${\mathcal L}\bu= \bff $ on $\Omega$ in the weak sense, that is, as $\mathcal{B}\brac{\bu,\bv}= \bff(\bv)$ for all $\bv\in C_c^\infty\pr{\Om}^N$. 
Typically $\bff$ will be an element of some $L^\ell(\Omega)^N$ space and so the action of $\bff$ on $\bv$ is then simply $\disp \int \bff\cdot \bv.$ 
The identity ${\mathcal L}^*\bu= \bff $ is interpreted similarly.

Returning to the properties of the Banach space $\bF\pr{\Om}$ and the associated Hilbert space $\bF_0\pr{\Om}$, we require that $\mathcal{B}$ and $\mathcal{B}^*$ can be extended to bounded and accretive bilinear forms on $\bF_0\pr{\Om} \times \bF_0\pr{\Om}$ so that the Lax-Milgram theorem may be applied in $\bF_0\pr{\Om}$.
\begin{enumerate}[label=A\arabic*), resume]
\item\label{A5} {\it Boundedness hypotheses:}\\
There exists a constant $\Ga > 0$ so that for any $\bu, \bv \in \bF_0\pr{\Om}$,
\begin{align}
\mathcal{B}\brac{\bu, \bv} \le \Ga \norm{\bu}_{\bF} \norm{\bv}_{\bF}. 
\label{eq2.21}
\end{align}
\item\label{A6} {\it Coercivity hypotheses:}\\
There exists a constant $\ga > 0$ so that for any $\bu \in \bF_0\pr{\Om}$,
\begin{align}
 \ga \norm{\bu}_{\bF}^2 \le \mathcal{B}\brac{\bu, \bu} 
\label{eq2.22}
\end{align}
\end{enumerate}
Finally, we assume
\begin{enumerate}[label=A\arabic*), resume]
\item\label{A7} {\it The Caccioppoli inequality:} 
If $\bu \in \bF\pr{\Om}$ is a weak solution to $\mathcal{L} \bu = \bf 0$ in $\Omega$ and $\zeta \in C^\infty(\rn)$ is such that $D\zeta \in C_c^\infty (\Omega)$ and $\zeta \bu  \in \bF_0\pr{ \Omega}$, $ \partial^i\zeta \,\bu\in L^2(\Omega)^N$, $i=1, ..., n$, then
\begin{align}
\int \abs{D \bu}^2 \zeta^2 \le C \int \abs{\bu}^2 \abs{D \zeta}^2,
\label{eq2.23}
\end{align}
where $C$ is a constant that depends on $n, s, t, \ga, \Ga, \norm{\bb}_{L^s\pr{\Om}}$, and $\norm{\bd}_{L^t\pr{\Om}}$.  
However, $C$ is independent of the set on which $\zeta$ and $D\zeta$ are supported.

We remark that the assumption $D\zeta \in C_c^\infty (\Omega)$ implies that $\zeta$ is a constant in the exterior of some large ball and, in particular, one can show that under the assumptions of {\rm\ref{A7}} we have also $\zeta^2 \bu  \in \bF_0\pr{ \Omega}$ (using {\rm\ref{A4}}). 
This will be useful later on. 
We also remark that the right-hand side of \eqref{eq2.23} is finite by our assumptions. 

Finally, let us point out that normally the Caccioppoli inequality will be used either in a ball or in the complement of the ball, that is, $\zeta=\eta$ or $\zeta=1-\eta$ for $\eta\in C_c^\infty (B_{2R})$ with $\eta=1$ on $B_R$, where $B_R$ is some ball in $\rn$ possibly intersecting $\po$. 
It is, in fact,  only the second case (the complement of the ball) which is needed for construction of the fundamental solution. 
\end{enumerate}

Throughout the paper, whenever we assume that \rm{\ref{A1} -- \ref{A7}} hold, we mean that the assumptions described by \rm{\ref{A1} -- \ref{A7}} hold for the collections of spaces $\bF\pr{\Om}$ and $\bF_0\pr{\Om}$ and  the elliptic operators $\mathcal{L}$ and $\mathcal{L}^*$ with bilinear forms $\mathcal{B}$ and $\mathcal{B}^*$, respectively.

We shall discuss extensively in Section~\ref{s7} and below how the common examples (notably, homogeneous elliptic systems and non-homogeneous elliptic systems with lower order terms in suitable $L^p$ or $B_p$ classes) fit into this framework.

To avoid confusion, we finally point out that  $\bF\pr{\Om}$ is of course a {\it collection} of Banach spaces, indexed by the domain $\Omega$, and the connection between $\bF(\Omega_1)$ and $\bF(\Omega_2)$ for $\Omega_1\cap \Omega_2\neq \emptyset$ is seen through the property \rm{\ref{A1}}.
That is, $\bF(U)$ contains all restrictions of elements of $\bF\pr{\Om}$, when $U\subset\Omega$. 
We do not assume that any element of  $\bF(U)$ can be extended to $\bF\pr{\Om}$. 
This is typical, e.g., for Sobolev spaces $W^{1,2}\pr{\Om}$, because the extension property might fail on bad domains. 

\section{Fundamental matrices and Green matrices}

This section resembles the work done in \cite{HK07}, but we deal here with operators that have lower order terms. 
In addition to the assumptions regarding $\bF\pr{\Om}$, $\bF_0\pr{\Om}$, $\mathcal{L}$ and $\mathcal{B}$ that are described in the previous section, we assume that all solutions satisfy certain de Giorgi-Nash-Moser estimates. 
Recall that in \cite{HK07} the authors imposed that all solutions to $L\bu = \bf 0$ satisfy bounds on Dirichlet integrals (their results applied only to homogeneous operators). 
Here, instead, we assume that weak solutions to non-homogeneous equations, $\mathcal{L} \bu = \bff$, for suitable $\bff$, satisfy certain scale-invariant Moser-type estimates and that solutions to homogeneous equations, $\mathcal{L} \bu = \bf 0$, are H\"older continuous. 
We shall make it precise below. 
To start though, let us introduce a slightly weaker hypothesis (a Moser-type local bound):
\begin{itemize}
\item For any $y\in \Om$, there exists an $R_y\in (0,\infty]$ such that whenever $0 < 2 r < R_y$, $\bff\in L^\ell\pr{\Om_r\pr{y}}^N$ for some $\ell \in \pb{\frac n 2, \iny}$, $\bu\in \bF\pr{\Om_{2r}\pr{y}}$ satisfies $\bu = \bz$ on $\partial \Om \cap B_{2r}(y)$ in the weak sense of \eqref{eq2.11}, and either $\mathcal{L} \bu = \bff$ or $\mathcal{L}^* \bu = \bff$ in $\Om_{r}\pr{y}$ in the weak sense, then for any $q>0$ there is a $C > 0$ so that
\eq{\label{eq3.1}
\norm{\bu}_{L^\infty\pr{\Om_{r/2}\pr{y}}} 
\le C \brac{ r^{- \frac n {q}} \norm{\bu}_{L^{q}\pr{\Om_r\pr{y}}} 
+ r^{2-\frac{n}{\ell}}\norm{\bff}_{L^\ell\pr{\Om_r\pr{y}}} }.
}
\end{itemize}

Without loss of generality, we can assume that the righthand side of \eqref{eq3.1} is finite. 
Indeed, if we take $\zeta \in C^\iny_c\pr{B_{2r}\pr{y}}$ such that $\zeta \equiv 1$ on $B_{r}\pr{y}$ then $\bu\zeta \in \bF_0(\Omega_{2r})$ assures that $\bu \in L^{2^*}\pr{\Omega_r}^N$ by the homogenous Sobolev inequality, \eqref{eq2.6}. 
Then \eqref{eq3.1} shows that $\bu \in L^{q}(\Omega_{r/2})^N$ for any $q<\infty$. 
Strictly speaking, it would be more coherent then to write \eqref{eq3.1} for $r<R_y/4$ but we ignore this minor inconsistency as clearly in practice one can always adjust the constants when proving \eqref{eq3.1}.
If $\ell = \iny$, then we interpret $\frac 1 \ell$ to equal $0$.
This convention will be used throughout.

Note that the constant $C$ in the estimate above is allowed to depend on the choice of $\mathcal{L}$, but it should be independent of $r$ and $R_y$.
In other words, we assume that all solutions satisfy a local scale-invariant Moser boundedness condition.

In this respect, we would like to make the following remark. 
All boundedness and H\"older continuity conditions on solutions that we impose are local in nature. 
However, slightly abusing the terminology, we refer to a given condition as local if it only holds for balls of the radius smaller than $R_0$, for some fixed $R_0>0$, depending or not depending on the center of the ball. 
As such, \eqref{eq3.1} is local. 
Later on, we will also talk about interior estimates which hold for balls inside $\Omega$ and boundary estimates in which balls are allowed to intersect the boundary.  
Either can be local or global depending on whether the size of the balls is restricted, and the interior estimates are of course always local if $\Omega\neq \RR^n$. 
In any case, we are always careful to specify the exact condition. 

\begin{rem} If $R_y=\dist\pr{y, \partial \Om}$ then $\partial \Om \cap B_r=\emptyset$, hence, in that case, \eqref{eq3.1} is merely an interior (rather than a boundary) condition.
\end{rem}

\subsection{A general construction method}

First, we establish a supporting lemma that will make the proofs in the following sections more concise.  We follow closely the argument in \cite{HK07}.

\begin{lem}\label{l3.2}
Let $\Om$ be an open connected subset of $\R^n$.  
Assume that {\rm\ref{A1} -- \ref{A7}} hold.
Then for all $y\in \Om$, $0 < \rho <d_y := \dist\pr{y, \del \Om}$, $k\in \set{1,\dots, N}$, there exists $\bv_\rho=\bv_{\rho; y, k}\in \bF_0\pr{\Om}$ such that
\eq{\label{eq3.2}
\mathcal{B}[\bv_\rho, \bu] = \fint_{B_\rho(y)} u^k = \frac{1}{\abs{B_\rho(y)}} \int_{B_\rho(y)} u^k, \quad \forall \, \bu \in \bF_0\pr{\Om}.
}
If, in addition, \eqref{eq3.1} holds, then there exists a function $\bv=\bv_{y, k}$ and a 
subsequence $\{\rho_\mu\}_{\mu=1}^\infty$, $\rho_{\mu}\to 0$, such that 
\aln{
&\bv_{\rho_\mu} \rightharpoonup \bv \quad \text{in } W^{1,q}\pr{\Om_r\pr{y}}^N
\quad \forall r< \tfrac 1 2 R_y, \quad \forall  q\in \pr{1,\frac{n}{n-1}}, 
\label{eq3.3}\\
&\bv_{\rho_\mu} \rightharpoonup \bv \quad \text{in } L^q\pr{\Om_r\pr{y}}^N
\quad \forall r< \tfrac 1 2 R_y, \quad \forall q \in \pr{1,\frac{n}{n-2}},
\label{eq3.4}\\
&\bv_{\rho_\mu} \rightharpoonup \bv \quad \text{in } Y^{1,2}\pr{\Om \setminus 
\Om_r\pr{y}}^N, \quad \forall r>0\label{eq3.5}.
}
For any $\phi \in C_c^\infty\pr{\Om}^N$, 
\eq{\label{eq3.6}
\mathcal{B}\brac{\bv, \phi}= \phi^k\pr{y}.
}
If $\bff \in L_c^\infty\pr{\Om}^N$ and $\bu\in \bF_0\pr{\Om}$ is the unique weak solution to $\mathcal{L}^* \bu = \bff$, then for a.e. $y\in \Om$,
\eq{\label{eq3.7}
u^k(y)=\int_{\Om} \bv \cdot \bff.
}
Furthermore, $\bv$ satisfies the following estimates:
\begin{align}
& \norm{\bv}_{L^{2^*}\pr{\Om \setminus \Om_r\pr{y}}} 
\le  C r^{1-\frac{n}{2}}, \quad \forall r< \tfrac 1 2 R_y,
\label{eq3.8} \\
& \norm{D\bv}_{L^2\pr{\Om \setminus \Om_r(y)}} 
\le C r^{1-\frac{n}{2}}, \quad \forall r<  \tfrac 1 2 R_y, 
\label{eq3.9} \\
& \norm{\bv}_{L^{q}\pr{\Om_r\pr{y}}} 
\le C_{q} r^{2-n+\frac{n}{q}}, \quad \forall r<  \tfrac 1 2 R_y, \quad \forall q\in \left[1,\tfrac{n}{n-2}\right),
\label{eq3.10} \\
& \norm{D \bv}_{L^{q}\pr{\Om_r\pr{y}}} 
\le C_{q} r^{1-n+\frac{n}{q}}, \quad \forall r< \tfrac 1 2 R_y, \quad \forall q\in \left[1,\tfrac{n}{n-1}\right),
\label{eq3.11} \\
& \abs{\set{x\in \Om : \abs{\bv(x)}> \tau }}
\le C\tau^{-\frac{n}{n-2}}, \quad \forall \tau >\pr{\tfrac 1 2 R_y}^{2-n}, 
\label{eq3.12} \\
& \abs{\set{x\in \Om : \abs{D\bv(x)}> \tau }}
\le C \tau^{-\frac{n}{n-1}}, \quad \forall \tau >\pr{\tfrac 1 2 R_y}^{1-n},
\label{eq3.13} \\  
& \abs{\bv(x)}
\le C R_{x,y}^{2-n} \quad \text{for a.e. $x\in \Om$, where } R_{x,y}:=\min \set{R_x, R_y, \abs{x-y}}, 
\label{eq3.14}
\end{align}
where each constant depends on $n$, $N$, $c_0$, $\Gamma$, $\gamma$, and the 
constants from \eqref{eq2.23} and \eqref{eq3.1}, and each $C_q$ depends additionally on $q$.
\end{lem}

\begin{pf}[Proof of Lemma~\ref{l3.2}] 
Let $\bu\in \bF_0\pr{\Om}$. 
Fix $y\in \Om$, $0 < \rho <d_y$, and $k \in \set{1,\dots, N }$, and consider the linear functional
\eqs{
\bu \mapsto \fint_{B_\rho(y)} u^k.
}
By  the H\"older inequality, \eqref{eq2.6}, and \eqref{eq2.5},
\begin{align}
\abs{ \fint_{B_\rho(y)} u^k}
&\le \frac{1}{\abs{B_\rho\pr{y}}} \int_{B_\rho\pr{y}} \abs{\bu}
\le \abs{B_\rho\pr{y}}^{\frac{2-n}{2n}} \pr{\int_{\Om} \abs{\bu}^{\frac{2n}{n-2}} }^{\frac{n-2}{2n}} 
\le c_n \abs{B_\rho\pr{y}}^{\frac{2-n}{2n}} \pr{\int_{\Om} \abs{D\bu}^{2} }^{\frac{1}{2}}  \nonumber \\
&\le c_0 c_n \rho^{\frac{2-n}{2}} \norm{\bu}_{\bF\pr{\Om}}.
\label{eq3.15}
\end{align}
Therefore, the functional is bounded on $\bF_0\pr{\Om}$, and by the Lax-Milgram theorem there exists a unique $\bv_\rho \in \bF_0\pr{\Om}$ satisfying \eqref{eq3.2}.
By coercivity of $\mathcal{B}$ given by \eqref{eq2.22} along with \eqref{eq3.15}, we obtain,
\eqs{
\ga \norm{\bv_\rho}^2_{\bF\pr{\Om}}
\le \mathcal{B}\brac{\bv_\rho,\bv_\rho}
= \abs{\fint_{B_\rho(y)} v_\rho^k}
\le c_0 c_n \rho^{\frac{2-n}{2}} \norm{\bv_\rho}_{\bF\pr{\Om}}
}
so that 
\eq{\label{eq3.16}
\norm{D\bv_\rho}_{L^2\pr{\Om}} 
\le c_0 \norm{\bv_\rho}_{\bF\pr{\Om}}
\le C \rho^{\frac{2-n}{2}},
}
where the first inequality is by \eqref{eq2.5}.

For $\bff \in L^\infty_c\pr{\Om}^N$, consider the linear functional
\eqs{
\bF_0\pr{\Om} \ni \bw \mapsto \iom \bff \cdot \bw.
}
This functional is bounded on $\bF_0\pr{\Om}$ since for every $\bw \in \bF_0\pr{\Om}$, and any $\ell \in \pb{ \frac n 2, \iny}$,
\eq{\label{eq3.17}
\abs{\iom \bff \cdot \bw }
\le \norm{\bff}_{L^{\ell}\pr{\Om}} \norm{\bw}_{L^{\frac{2n}{n-2}}\pr{\Om}} \abs{\supp \bff}^{\frac{n+2}{2n} - \frac 1 \ell}
\le C \norm{\bff}_{L^{\ell}\pr{\Om}} \abs{\supp \bff}^{\frac{n+2}{2n} - \frac 1 \ell} \norm{\bw}_{\bF\pr{\Om}},
%
}
where we have again used \eqref{eq2.6} and \eqref{eq2.5}.
Then, once again by Lax-Milgram, we obtain $\bu \in \bF_0\pr{\Om}$ such that
\eq{\label{eq3.18}
\mathcal{B}^*\brac{\bu, \bw}=\iom \bff \cdot \bw, \quad \forall \, \bw \in \bF_0\pr{\Om}.
}
Set $\bw=\bu$ in \eqref{eq3.18} and use the coercivity assumption, \eqref{eq2.22}, for $\mathcal{B}^*$ and \eqref{eq3.17} to get
\eq{\label{eq3.19}
\norm{\bu}_{\bF\pr{\Om}} \le C \norm{\bff}_{L^{\ell}\pr{\Om}}  \abs{\supp \bff}^{\frac{n+2}{2n} - \frac 1 \ell}.
}
Also, if we take $\bw = \bv_\rho$ in \eqref{eq3.18}, we get
\eq{\label{eq3.20}
\iom \bff \cdot \bv_\rho 
= \mathcal{B}^*[\bu, \bv_\rho]
 = \mathcal{B}[\bv_\rho, \bu] 
 = \fint_{B_\rho(y)} u^k.
}

Let $\bff \in L_c^\infty\pr{\Om}^N$ be supported in $\Om_r(y)$, where $0< 2r<R_y$, and let $\bu$ be as in \eqref{eq3.18}.  
Since $\bu \in \bF_0\pr{\Om}$, then \rm{\ref{A1}} implies that $\bu \in \bF\pr{\Om_{2r}}$ and \rm{\ref{A4}} gives $\bu = \bz$ on $\del \Om \cap \Om_{2r}$ so that \eqref{eq3.1} is applicable. 
Then, by \eqref{eq3.1} with 
$q=\frac{2n}{n-2}$ and $\ell \in \pb{ \frac n 2, \iny}$,
\eqs{
\norm{\bu}^2_{L^\infty\pr{\Om_{r/2}\pr{y}}} 
\le C \bigg( r^{2-n} \norm{\bu}^2_{L^{\frac{2n}{n-2}}\pr{\Om_r\pr{y}}} + r^{4 - \frac{2n}{\ell}} \norm{\bff}^2_{L^{\ell}\pr{\Om_r\pr{y}}}\bigg).
}
By \eqref{eq2.6}, \eqref{eq2.5}, and \eqref{eq3.19} with $\supp \bff \subset \Om_r\pr{y}$,
\eqs{
\norm{\bu}^2_{L^{2^*}\pr{\Om_r\pr{y}}} 
\le \norm{\bu}^2_{L^{2^*}\pr{\Om}} \le C \norm{\bu}^2_{\bF\pr{\Om}}
\le C \abs{\Om_r\pr{y}}^{1 + \frac{2}{n} - \frac {2} \ell} \norm{\bff}^2_{L^{\ell}\pr{\Om}}
\le C \abs{B_r\pr{y}}^{1 + \frac{2}{n} - \frac {2} \ell} \norm{\bff}^2_{L^{\ell}\pr{\Om}},
}
where, as before, $2^*=\frac{2n}{n-2}$.
Combining the previous two inequalities, we get
\eqs{
\norm{\bu}^2_{L^\infty\pr{\Om_{r/2}\pr{y}}} 
\le C r^{4-\frac{2n}{\ell}} \norm{\bff}^2_{L^{\ell}\pr{\Om}}.
}
Therefore,
\eq{
\norm{\bu}_{L^\infty \pr{\Om_{r/2}\pr{y}}} 
\le C r^{2-\frac{n}{\ell}} \norm{\bff}_{L^{\ell}\pr{\Om}} 
= C r^{2-\frac{n}{\ell}} \norm{\bff}_{L^{\ell}\pr{\Om_r\pr{y}}}, \quad \forall \ell \in \pb{\frac{n}{2}, \iny}.\label{eq3.21}
}
By \eqref{eq3.20} and \eqref{eq3.21}, if $\rho \le r/2$, $\rho < d_y$, we have
\eqs{
\abs{\int_{\Om_r\pr{y}} \bff \cdot \bv_\rho}
= \abs{\iom \bff \cdot \bv_\rho}
\le \fint_{B_\rho(y)} \abs{\bu}
\le \norm{\bu}_{L^\infty(B_\rho(y))}
\le \norm{\bu}_{L^\infty(\Om_{r/2}(y))} 
\leq Cr^{2-\frac{n}{\ell}} \norm{\bff}_{L^{\ell}\pr{\Om_r\pr{y}}}, \quad \forall \ell \in \pb{\frac{n}{2}, \iny}. 
}
By duality, this implies that for $r<\frac 1 2 R_y$,
\eq{\label{eq3.22}
\norm{\bv_\rho}_{L^q\pr{\Om_r\pr{y}}}
\le C r^{2-n+\frac{n}{q}}, \quad \mbox{ for all } \rho \le \frac{r}{2}, \rho < d_y, \quad \forall q \in \left[1,\tfrac{n}{n-2}\right).
}

Fix $x\ne y$ such that $r:= \frac{4}{3} |x-y| < \frac{1}{2}R_y$.  
For $\rho \le r/2, \rho < d_y$, $\bv_\rho$ is a weak solution to $\mathcal{L} \bv_\rho=\bz$ in $\Om_{r/4}(x)$.  
Moreover, since $\bv_\rho \in \bF_0\pr{\Om}$, then \rm{\ref{A1}} implies that $\bv_\rho \in \bF\pr{\Om_{r/2}\pr{x}}$ and \rm{\ref{A4}} implies that $\bv_\rho = \bz$ on $\del \Om \cap \Om_{r/2}\pr{x}$ so we may use \eqref{eq3.1}.
Thus, applying \eqref{eq3.1} with $q=1$ and \eqref{eq3.22}, we get for a.e. $x \in \Om$ as above, 
\eq{\label{eq3.23}
\abs{\bv_\rho(x)}
\le C r^{-n} \norm{\bv_\rho}_{L^1\pr{\Om_{r/4}\pr{x}}} 
\le C r^{-n} \norm{\bv_\rho}_{L^1\pr{\Om_{r}\pr{y}}} 
\le Cr^{2-n} 
\approx  |x-y|^{2-n}.
}
 
Now, for any  $r< \frac 1 2 R_y$ and $\rho\le r/{2}$, $\rho<d_y$,   
let $\zeta$ be a cut-off function such that
\eq{\label{eq3.24}
\zeta \in C^{\iny}(\R^n), \quad
0\le \zeta \le 1, \quad 
\zeta \equiv 1 \text{ outside $B_{r}(y)$}, \quad 
\zeta \equiv 0 \text{ in $B_{r/2}(y)$}, \quad 
\text{and} \;  \abs{D \zeta} \le C/r.}
Then $\bv_\rho \zeta, \bv_\rho \partial^i\zeta \in \bF_0\pr{\Om\setminus \overline{\Om_{r/2}(y)}}$, for all $i=1, ..., n$. 
For the functions $\bv_\rho \partial^i\zeta$, this fact follows from \eqref{eq2.9}.
 The function $\bv_\rho \zeta$ is a little more delicate since $\zeta$ is not compactly supported. 
 However, since $\zeta$ equals 1 in the complement of $B_r(y)$, then $1-\zeta$ is compactly supported.
Thus, if $\set{\bv_n} \su C_c^\infty (\Omega)^N$ converges to $\bv_\rho$ in the $\bF(\Omega)$-norm, then, by \eqref{eq2.9}, $\set{\bv_n(1-\zeta)} \su C_c^\infty (\Omega)^N$ converges to $\bv_\rho(1-\zeta)$ in the $\bF(\Omega)$-norm . 
 Adding up these statements, we conclude that $\set{\bv_n \zeta} \su C_c^\infty (\Omega)^N$ approximates $\bv_\rho \zeta$ in $\bF(\Omega)$, as required.
 
 Now, since $\mathcal{L} \bv_\rho = \bz$ in $\Om\setminus \overline{\Om_{r/2}(y)}$, the Caccioppoli inequality, \eqref{eq2.23}, implies that
\begin{equation}\label{eq3.25}
\int_{\Om} \zeta^2 \abs{D \bv_\rho}^2 
\le C \int_{\Om} \abs{D\zeta}^2 \abs{\bv_\rho}^2 
\le C r^{-2} \int_{\Om_{r}(y)\setminus \Om_{r/2}(y)} \abs{\bv_\rho}^2, \quad \forall \rho \le \frac{r}{2}, \,\rho<d_y.
\end{equation}
Combining \eqref{eq3.25} and \eqref{eq3.23}, we have for all $r< \frac 1 2 R_y$ and $\zeta$ as above,
\aln{\label{eq3.26}
\begin{split}
\int_{\Om} \abs{D(\zeta \bv_\rho)}^2 
&\le 2 \int_{\Om} \zeta^2 \abs{D \bv_\rho}^2 + 2 \int_{\Om} \abs{D\zeta}^2 \abs{\bv_\rho}^2 \\
&\le C r^{-2} \int_{\Om_{r}(y)\setminus \Om_{r/2}(y)} \abs{\bv_\rho}^2 
\le C r^{2-n}
, \quad \forall \rho\le \frac{r}{2}, \,\rho<d_y.
\end{split}
}
It follows from \eqref{eq2.6} and \eqref{eq3.26} that for $r< \frac 1 2R_y$,
\begin{equation}
\int_{\Omega \setminus \Om_r(y)} \abs{\bv_\rho}^{\frac{2n}{n-2}} 
\le \int_{\Omega} \abs{\zeta \bv_\rho}^{\frac{2n}{n-2}} 
\le \pr{\int_{\Omega} \abs{D\pr{\zeta \bv_\rho}}^2}^{\frac{n}{n-2}} 
\le C r^{-n}
, \quad \forall \rho \le \frac{r}{2}, \,\rho<d_y.
\end{equation}
On the other hand, if $\frac r 2 < \rho < d_y$, then \eqref{eq2.6} and 
\eqref{eq3.16} imply
\begin{equation}
\int_{\Om \setminus \Om_r(y)} \abs{\bv_\rho}^{\frac{2n}{n-2}} 
\le \int_{\Om} \abs{\bv_\rho}^{\frac{2n}{n-2}} 
\le C \pr{ \int_{\Om} \abs{D \bv_\rho}^2 }^{\frac{n}{n-2}} 
\le C r^{-n}.
\end{equation}
Therefore, combining the previous two results, we have 
\begin{equation}\label{eq3.29}
\int_{\Om\setminus \Om_r(y)} |\bv_\rho|^{\frac{2n}{n-2}} 
\le C r^{-n}, \quad \forall r< \tfrac{1}{2} R_y, \quad \forall \, 0 < \rho < d_y.
\end{equation}

Fix $\tau > \pr{R_y/2}^{2-n}$.  
If $R_y=\infty$, then fix $\tau> 0$.  
Let $A_\tau=\set{x\in \Om: \abs{\bv_\rho}>\tau}$ and set $r=\tau^{\frac{1}{2-n}}$.  
Note that $r< \frac 1 2 R_y$.  
Then, using \eqref{eq3.29}, we see that if $0 < \rho < d_y$,
\begin{equation*}
\abs{A_\tau\setminus \Om_r(y)}
\le \tau^{-\frac{2n}{n-2}}\int_{A_\tau\setminus \Om_r(y)} \abs{\bv_\rho}^{\frac{2n}{n-2}} 
= C\tau^{-\frac{n}{n-2}}.
\end{equation*}
Since $\abs{A_\tau \cap \Om_r(y)} \le \abs{\Om_r(y)} \le Cr^n = C\tau^{-\frac{n}{n-2}}$, we have
\begin{equation}\label{eq3.30}
\abs{\set{ x\in \Om: \abs{\bv_\rho(x)} > \tau }}
\le C \tau^{-\frac{n}{n-2}} \quad \text{if }\tau> \pr{\frac{R_y}{2}}^{2-n}, \quad \forall \, 0 < \rho < d_y.
\end{equation}

Fix $r<\frac 1 2 R_y$ and let $\zeta$ be as in \eqref{eq3.24}.  
Then \eqref{eq3.26} gives
\begin{equation*}
 \int_{\Om \setminus \Om_{r}(y)} \abs{D \bv_\rho}^2 
 \le C r^{2-n}
 , \quad \forall r< \tfrac{1}{2} R_y, \quad \forall \rho\le \frac{r}{2}.
\end{equation*}
Now, if $\frac r 2 < \rho < d_y$, we have from \eqref{eq3.16} that
\begin{equation*}
 \int_{\Om\setminus \Om_r(y)} \abs{D \bv_\rho}^2 
 \le \int_{\Om} \abs{D \bv_\rho}^2 
 \le C \rho^{2-n}\le C r^{2-n}.
\end{equation*}
Combining the previous two results yields
\begin{equation}\label{eq3.31}
 \int_{\Om\setminus \Om_r(y)} \abs{D \bv_\rho}^2 
 \le C r^{2-n}, \quad \forall r< \tfrac{1}{2} R_y, \quad \forall \, 0 < \rho < d_y.
\end{equation}

Fix $\tau > \pr{R_y/2}^{1-n}$.  
If $R_y=\infty$, let $\tau >0$.  
Let $A_\tau = \set{x \in \Om : \abs{D \bv_\rho} > \tau}$ and set $r=\tau^{\frac{1}{1-n}}$.  
Note that $r<\frac 1 2 R_y.$  
Then, using \eqref{eq3.31}, we see that if $0 < \rho < d_y$,
\eqs{
\abs{A_\tau \setminus \Om_r(y)} \le \tau^{-2} \int_{A_\tau\setminus \Om_r(y)} \abs{D \bv_\rho}^2 
\le C \tau^{-\frac{n}{n-1}}.
}
Since $\abs{A_\tau \cap \Om_r(y)} \le C r^n = C\tau^{-\frac{n}{n-1}}$, then
\begin{equation}\label{eq3.32}
 \abs{\set{x\in \Om: \abs{D \bv_\rho(x)} > \tau }} \le C \tau^{-\frac{n}{n-1}} 
 \quad \text{if } \tau > 
\pr{\tfrac{1}{2} R_y }^{1-n}, \quad \forall \, 0 < \rho < d_y.
\end{equation}

For any $\si> \pr{R_y/2}^{1-n}$ and $q > 0$, we have
\begin{equation*}
\int_{\Om_r(y)} \abs{D\bv_\rho}^q 
\le \si^q \abs{\Om_r(y)} + \int_{\set{\abs{D\bv_\rho}>\si}} \abs{D \bv_\rho}^q.
\end{equation*}
By \eqref{eq3.32}, for $q\in \pr{0,\frac{n}{n-1}}$ and $\rho \in \pr{0, d_y}$, 
\begin{align*}
\int_{\set{|D\bv_\rho|>\si}} \abs{D\bv_\rho}^q 
&= \int_0^\infty q \tau^{q-1} \abs{\set{\abs{D\bv_\rho}>\max\set{\tau,\si}}} d\tau \\
&\le C \si^{-\frac{n}{n-1}} \int_0^\si q \tau^{q-1} \, d \tau
+ C\int_\si^\infty q \tau^{q-1 -\frac{n}{n-1}} \, d\tau
= C \pr{1-\frac{q}{q-\frac{n}{n-1}}} \si^{q-\frac{n}{n-1}}.
\end{align*}
Therefore, taking $\si=r^{1-n}$, we conclude that
\begin{equation}\label{eq3.33}
 \int_{\Om_r(y)} \abs{D\bv_\rho}^q 
 \le C_q r^{q(1-n)+n}, \quad \forall r< \tfrac 1 2 R_y, \quad \forall \, 0 < \rho < d_y, \quad \forall q\in \pr{0, \tfrac{n}{n-1}}.
\end{equation}
By the same process with \eqref{eq3.30} in place of \eqref{eq3.32} and $\si= r^{2-n}$, we have
\begin{equation}\label{eq3.34}
\int_{\Om_r(y)} \abs{\bv_\rho}^q 
\le C_q r^{q(2-n)+n}, \quad \forall r< \tfrac 1 2 R_y, \quad \forall \, 0 <  \rho < d_y, \quad \forall q \in \pr{0, \tfrac{n}{n-2}}.
\end{equation}

Fix $q\in \pr{1, \frac{n}{n-1}}$ and $\tilde q\in \pr{1, \frac{n}{n-2}}$.  
From \eqref{eq3.33} and \eqref{eq3.34}, it follows that for any $r < \frac 1 2 R_y$
\begin{equation}\label{eq3.35}
\norm{\bv_\rho}_{W^{1,q}\pr{\Omega_{r}\pr{y}}} \le C\pr{r} \mbox{ and } \norm{\bv_\rho}_{L^{\tilde q}\pr{\Omega_{r}\pr{y}}} \le C\pr{r}\quad \text{uniformly in } 
\rho.
\end{equation}
Therefore, (using diagonalization) we can show that there exists a sequence $\set{\rho_\mu}_{\mu=1}^\infty$ tending to $0$ and a function $\bv=\bv_{y,k}$ such that
\begin{equation}\label{eq3.36}
\bv_{\rho_\mu} \rightharpoonup \bv \quad \text{in } 
W^{1,q}\pr{\Om_{r}\pr{y}}^{N} \mbox{ and in } L^{\tilde q} \pr{\Om_{r}\pr{y}}^{N}, \mbox{ for all } r < \frac 1 2 R_y.
\end{equation}

Furthermore, for fixed $r_0 < r$, \eqref{eq3.29} and \eqref{eq3.31} imply uniform bounds on $\bv_{\rho_\mu}$ in $Y^{1,2}\pr{\Om \setminus \Om_{r_0}\pr{y}}^N$ for small $\rho_\mu$.
Thus, there exists a subsequence of $\set{\rho_\mu}$ (which we will not rename)  and a function $\widetilde{\bv}=\widetilde{\bv}_{y,k}$ such that
\begin{equation}\label{eq3.37}
\bv_{\rho_\mu} \rightharpoonup \widetilde{\bv} \quad \text{in } Y^{1,2}\pr{\Om 
\setminus \Om_{r_0}\pr{y}}^N.
\end{equation}
Since $\bv \equiv \widetilde{\bv}$ on $\Om_{r}\pr{y}\setminus \Om_{r_0}\pr{y}$, we can extend $\bv$ to the entire $\Om$ by setting $\bv=\widetilde{\bv}$ on $\Om\setminus \Om_{r}\pr{y}$.  
For ease of notation, we call the extended function $\bv$.  
Applying the diagonalization process again, we conclude that there exists a sequence $\rho_\mu \to 0$ and a function $\bv$ on $\Omega$ such that 
\begin{equation}\label{eq3.38}
\bv_{\rho_\mu} \rightharpoonup \bv \quad \text{in } 
W^{1,q}\pr{\Om_{r}\pr{y}}^{N} \mbox{ and in }L^{\tilde q}\pr{\Om_{r}\pr{y}}^{N},
\end{equation}
and
\begin{equation}\label{eq3.39}
\bv_{\rho_{\mu}} \rightharpoonup \bv \quad \text{in } Y^{1,2}\pr{\Om \setminus 
\Om_{r_0}\pr{y}}^N,
\end{equation}
for all $r_0 < r< \frac 1 2 R_y$.

Let $\phi \in C_c^\infty\pr{\Om}^N$ and $r<\frac 12 R_y$ such that $r<d_y$.
Choose $\eta \in C^\iny_c\pr{B_r\pr{y}}$ to be a cutoff function so that $\eta \equiv 1$ in $B_{r/2}\pr{y}$.
We write $\phi = \eta \phi + \pr{1 - \eta} \phi$.
By \eqref{eq3.2} and the definition of $\mathcal{B}$,
\begin{align*}
\lim_{\mu \to \iny} \fint_{B_{\rho_\mu}\pr{y}} &\eta \phi^k  
= \lim_{\mu\to \infty} \mathcal{B}[\bv_{\rho_\mu; y,k},\eta \phi] \\
&= \lim_{\mu\to \infty} \int_\Om A_{ij}^{\al \be} D_\be {\bv_{\rho_\mu; y,k}^j} D_\al \pr{\eta \phi^i} + b_{ij}^\al \bv_{\rho_\mu; y,k}^j D_\al \pr{\eta \phi^i} + d_{ij}^\be D_\be {\bv_{\rho_\mu; y,k}^j} \eta \phi^i + V_{ij} \bv_{\rho_\mu; y,k}^j \eta \phi^i .
\end{align*}
Note that $\eta \phi^i$ and $D\pr{\eta \phi^i}$ belong to $C^\iny_c\pr{\Om_r\pr{y}}$.  
From this, the boundedness of $\bA$ given by \eqref{eq2.13}, and the assumptions on $\bV$, $\bb$, and $\bd$ given by \eqref{eq2.14}, it follows that there exists a $q^\prime > n$ such that each of $A_{ij}^{\al \be} D_\al \pr{\eta \phi^i}$, $b_{ij}^\al D_\al \pr{\eta \phi^i}$, $d_{ij}^\be \eta \phi^i$, and $V_{ij}\eta \phi^i$ belong to $L^{q^\prime}\pr{\Om_r\pr{y}}^N$.
Therefore, by \eqref{eq3.38},
\begin{align}
\lim_{\mu \to \iny} \fint_{B_{\rho_\mu}\pr{y}} \eta \phi^k  
&= \int_\Om A_{ij}^{\al \be} D_\be {\bv_{y,k}^j} D_\al \pr{\eta \phi^i} + b_{ij}^\al \bv_{y,k}^j D_\al \pr{\eta \phi^i} + d_{ij}^\be D_\be {\bv_{y,k}^j} \eta \phi^i + V_{ij} \bv_{y,k}^j \eta \phi^i \nonumber \\
&= \mathcal{B}[\bv_{y,k},\eta \phi] .
\label{eq3.40}
\end{align}
Another application of \eqref{eq3.2} shows that
\begin{align*}
\lim_{\mu \to \iny} \fint_{B_{\rho_\mu}\pr{y}} \pr{1 - \eta} \phi^k 
&= \lim_{\mu\to \infty} \int_\Om A_{ij}^{\al \be} D_\be {\bv_{\rho_\mu; y,k}^j} D_\al \brac{\pr{1-\eta} \phi^i} 
+ b_{ij}^\al \bv_{\rho_\mu; y,k}^j D_\al \brac{\pr{1-\eta} \phi^i} \\
&+ \lim_{\mu\to \infty} \int_\Om d_{ij}^\be D_\be {\bv_{\rho_\mu; y,k}^j} \pr{1-\eta} \phi^i 
+ V_{ij} \bv_{\rho_\mu; y,k}^j \pr{1-\eta} \phi^i.
\end{align*}
Since $\phi \in C_c^\infty\pr{\Om}^N$ and $\eta \in C^\iny_c\pr{B_r\pr{y}}$, then $\pr{1 - \eta}\phi$ and $D\brac{\pr{1 - \eta}\phi}$ belong to $C_c^\infty(\Om \setminus B_{r/2}\pr{y})^N$.
In combination with \eqref{eq2.13}, this implies that each $A_{ij}^{\al \be} D_\al \brac{\pr{1 - \eta} \phi^i}$ belongs to $L^{2}\pr{\Om \setminus B_{r/2}\pr{y}}^N$.
The assumption on $\bd$ given in \eqref{eq2.14} implies that each $d_{ij}^\be \pr{1-\eta} \phi^i$ belongs to $L^{2}\pr{\Om \setminus B_{r/2}\pr{y}}^N$ as well.
And the assumption on $\bb$ and $\bV$ given in \eqref{eq2.14} imply that every $b_{ij}^\al D_\al \brac{\pr{1-\eta} \phi^i}$ and $V_{ij} \pr{1-\eta} \phi^i$ belong to $L^{\frac{2n}{n+2}}\pr{\Om \setminus B_{r/2}\pr{y}}^N$.
Therefore, it follows from \eqref{eq3.39} that
\begin{align}
\lim_{\mu \to \iny} \fint_{B_{\rho_\mu}\pr{y}} \pr{1 - \eta} \phi^k 
&= \int_\Om A_{ij}^{\al \be} D_\be {\bv_{y,k}^j} D_\al \brac{\pr{1-\eta} \phi^i} 
+ b_{ij}^\al \bv_{y,k}^j D_\al \brac{\pr{1-\eta} \phi^i} \nonumber \\
&+ \int_\Om d_{ij}^\be D_\be {\bv_{y,k}^j} \pr{1-\eta} \phi^i 
+ V_{ij} \bv_{y,k}^j \pr{1-\eta} \phi^i
= \mathcal{B}[\bv_{y,k},\pr{1-\eta} \phi] .
\label{eq3.41}
\end{align}
It follows from combining \eqref{eq3.40} and \eqref{eq3.41} that for any $\phi \in C^\iny_c\pr{\Om}^N$,
\begin{align}
\phi^k\pr{y}
&= \lim_{\mu \to \iny} \fint_{B_{\rho_\mu}\pr{y}} \phi^k  
= \lim_{\mu \to \iny} \fint_{B_{\rho_\mu}\pr{y}} \eta \phi^k  
+ \lim_{\mu \to \iny} \fint_{B_{\rho_\mu}\pr{y}} \pr{1 - \eta} \phi^k \nonumber \\
&= \mathcal{B}[\bv_{y,k},\eta \phi] 
+  \mathcal{B}[\bv_{y,k},\pr{1-\eta} \phi] 
=  \mathcal{B}[\bv_{y,k}, \phi],
\label{eq3.42}
\end{align}
so that \eqref{eq3.6} holds.

As before, for any $\bff\in L_c^\infty\pr{\Om}^N$, let $\bu\in \bF_0\pr{\Om}$ be the unique weak solution to $\mathcal{L}^* \bu = \bff$, i.e, assume that $\bu\in \bF_0\pr{\Om}$ satisfies \eqref{eq3.18}.
Then for a.e. $y\in \Om$, 
\begin{align*}
u^k(y)
=\lim_{\mu\to \infty} \fint_{B_{\rho_\mu}\pr{y}} u^k 
= \lim_{\mu\to \infty} \mathcal{B}\brac{\bv_{\rho_\mu; y, k}, \bu}
&= \lim_{\mu\to \infty} \mathcal{B}^*\brac{\bu, \bv_{\rho_\mu; y, k}}
= \lim_{\mu\to \infty} \int_{\Om} \bv_{\rho_\mu} \cdot \bff ,
\end{align*}
where we have used \eqref{eq3.2}.
For $\eta \in C^\iny_c\pr{B_r\pr{y}}$ as defined in the previous paragraph, since $\eta \bff \in L^{q^\prime}\pr{B_r\pr{y}}^N$ and $\pr{1 - \eta} \bff \in L^{\frac{2n}{n+2}}\pr{\Om \setminus B_{r/2}\pr{y}}^N$, then it follows from \eqref{eq3.38} and \eqref{eq3.39} that
\begin{align*}
\lim_{\mu\to \infty} \int_{\Om} \bv_{\rho_\mu} \cdot \bff
&= \lim_{\mu\to \infty}  \int_{B_r\pr{y}} \bv_{\rho_\mu} \cdot \eta \bff 
+ \lim_{\mu\to \infty}  \int_{\Om \setminus B_{r/2}\pr{y}} \bv_{\rho_\mu} \cdot \pr{1 - \eta} \bff \\
&= \int_{B_r\pr{y}} \bv \cdot \eta \bff 
+  \int_{\Om \setminus B_{r/2}\pr{y}} \bv \cdot \pr{1 - \eta} \bff
= \int_{\Om} \bv \cdot \bff.
\end{align*}
Combining the last two equations gives \eqref{eq3.7}.

The estimates \eqref{eq3.8}--\eqref{eq3.13} follow almost directly by passage to the limit. 
Indeed, for any $r < \frac 1 2 R_y$ and any $\bg\in L_c^\infty\pr{\Om_r\pr{y}}^N$, \eqref{eq3.34} implies that
\eqs{
\abs{\int_{\Om} \bv \cdot \bg}
= \lim_{\mu\to \infty} \abs{\int_{\Om} \bv_{\rho_\mu} \cdot \bg}
\le C_q r^{2-n+\frac{n}{q}} \norm{\bg}_{L^{q'}\pr{\Om_r\pr{y}}},
}
where $q'$ is the H\"older conjugate exponent of $q \in [1, \frac{n}{n-2}).$  
By duality, we obtain that for every $q \in [1,\frac{n}{n-2})$,
\eq{
\norm{\bv}_{L^q\pr{\Om_r\pr{y}}} 
\le C_q r^{2-n+\frac{n}{q}}, \quad \forall r<  \tfrac 1 2 R_y,
\label{eq3.43}
}
that is, \eqref{eq3.10} holds. 
A similar argument using \eqref{eq3.33}, \eqref{eq3.29} and \eqref{eq3.31},
yields \eqref{eq3.11}, \eqref{eq3.8}, and \eqref{eq3.9}, respectively. Now, as in the proofs of  \eqref{eq3.30} and \eqref{eq3.32}, \eqref{eq3.8} and \eqref{eq3.9} give \eqref{eq3.12} and \eqref{eq3.13}.

Passing to the proof of \eqref{eq3.14}, fix $x \ne y$. For a.e. $x \in \Om$, the Lebesgue differentiation theorem implies that
\begin{align*}
\bv\pr{x} 
&= \lim_{\de \to 0^+} \fint_{\Om_\de\pr{x}} \bv
= \lim_{\de \to 0^+} \frac{1}{\abs{\Om_\de}} \int \bv \, \chi_{\Om_\de\pr{x}},
\end{align*}
where $\chi$ denotes an indicator function.
Assuming as we may that $2\de \le \min\set{d_x, \abs{x - y}}$, it follows that $\chi_{\Om_\de\pr{x}} = \chi_{B_\de\pr{x}} \in L^{\frac{2n}{n+2}}\pr{\Om\setminus \Om_{\de}\pr{y}}$. 
Therefore, \eqref{eq3.39} implies that
\begin{align*}
\frac{1}{\abs{B_\de}}  \int \bv \, \chi_{B_\de\pr{x}}
&= \lim_{\mu \to \iny} \frac{1}{\abs{B_\de}} \int \bv_{\rho_\mu} \, \chi_{B_\de\pr{x}}
= \lim_{\mu \to \iny} \fint_{B_\de\pr{x}} \bv_{\rho_\mu}.
\end{align*}
If $\abs{x - y} \le \frac 1 4 R_y$ and $\rho_\mu \le \frac 1 3 \abs{x - y}$, $\rho_\mu<d_y$,  then \eqref{eq3.23} implies that for a.e. $z \in B_\de\pr{x}$
\begin{align*}
\abs{\bv_{\rho_\mu}\pr{z}}
\le C \abs{z-y}^{2-n},
\end{align*}
where $C$ is independent of $\rho_\mu$.
Since $\abs{z - y} >\frac{1}{2} \abs{x - y}$ for every $z \in B_\de\pr{x} \su B_{\abs{x-y}/2}\pr{x}$, then
\begin{align}
\norm{\bv_{\rho_\mu}}_{L^\iny\pr{B_\de\pr{x}}} \le C \abs{x - y}^{2-n}.
\label{eq3.44}
\end{align}
On the other hand, if $\abs{x - y} > \frac 1 4 R_y$, then for $r := \frac 1 8 \min\set{R_x, R_y}$, the restriction property, \rm{\ref{A1}}, implies that $\bv_{\rho_{\mu}} \in \bF\pr{\Om_{2 r}\pr{x}}$ and it follows from \rm{\ref{A4}} that $\bv_{\rho_{\mu}}$ vanishes along $\Om_{2 r}\pr{x} \cap \del \Om$.
As long as $\rho_\mu \le r$, $\rho_\mu < d_y$, $\mathcal{L} \bv_{\rho_{\mu}} = \bz$ in $\Om_{r}\pr{x}$, so we may apply \eqref{eq3.1} with $q = 2^*$.
We have
\begin{align}
\norm{\bv_{\rho_\mu}}_{L^\iny\pr{\Om_{r/2}\pr{x}}}
\le C r^{- \frac{n-2}{2}} \pr{\int_{\Om_r\pr{x}} \abs{\bv_{\rho_\mu}}^{2^*}}^{\frac 1 {2^*}}
\le C r^{- \frac{n-2}{2}} \pr{\int_{\Om \setminus \Om_r\pr{y}} \abs{\bv_{\rho_\mu}}^{2^*}}^{\frac 1 {2^*}}
\le C r^{2-n},
\label{eq3.45}
\end{align}
where the last inequality follows from \eqref{eq3.29}.
If we define $R_{x,y} = \min\set{R_x, R_y, \abs{x-y}}$, then \eqref{eq3.44} and \eqref{eq3.45} imply that for $\delta$ and $\rho_\mu$ sufficiently small (independently of each other), 
\begin{align}
\norm{\bv_{\rho_\mu}}_{L^\iny\pr{B_{\de}\pr{x}}}
\le C R_{x,y}^{2-n}.
\end{align}
By combining with the observations above, we see that for a.e. $x \in \Om$,
\begin{align*}
\bv\pr{x} 
&= \lim_{\de \to 0^+} \frac{1}{\abs{\Om_\de}} \int \bv \, \chi_{\Om_\de\pr{x}}
= \lim_{\de \to 0^+}  \lim_{\mu \to \iny} \fint_{B_\de\pr{x}} \bv_{\rho_\mu}
\le \lim_{\de \to 0^+}  \lim_{\mu \to \iny} C R_{x,y}^{2-n}
= C R_{x, y}^{2-n}.
\end{align*}
\end{pf}

\subsection{Fundamental matrix}\label{s3.2}
In this section, we construct the fundamental matrix associated to $\mathcal{L}$ on $\Omega=\mathbb{R}^n$ with $n\ge 3$.  
We maintain the assumptions \rm{\ref{A1}--\ref{A7}} with $\Omega=\mathbb{R}^n$ and replace \eqref{eq3.1} with the following global (interior) scale-invariant Moser-type bound. 
For the sake of future reference, within these definitions we maintain a general set $\Omega$ and emphasize their interior nature.
\begin{itemize}
\item[(IB)] Let $\Omega$ be a connected open set in $\RR^n$.   
We say that \rm{(IB)} holds in $\Omega$ if whenever $\bu \in \bF(B_{2R})$ is a weak solution to $\mathcal{L} \bu = \bff$ or $\mathcal{L}^* \bu = \bff$ in $B_R$, for some $B_R\subset \Omega$, $R>0$, where $\bff \in L^\ell\pr{B_R}^N$ for some $\ell \in \pb{ \frac{n}{2}, \iny}$, then for any $q > 0$,
\eq{
\norm{\bu}_{L^\iny\pr{B_{R/2}}} 
\le C \brac{ R^{- \frac n q}\norm{\bu}_{L^q\pr{B_R}} + R^{2 - \frac{n}{\ell}} \norm{\bff}_{L^\ell\pr{B_R}}},
\label{eq3.47} }
where the constant $C>0$ is independent of $R>0$.
\end{itemize}

We also assume a local H\"older continuity condition for solutions:
\begin{itemize}
\item[(H)]  Let $\Omega$ be a connected open set in $\RR^n$. 
 We say that \rm{(H)} holds in $\Om$ if whenever $\bu \in \bF(B_{2R_0})$ is a weak solution to $\mathcal{L} \bu = \bz$ or $\mathcal{L}^* \bu = \bz$ in $B_{R_0}$ for some $B_{2R_0}\subset \Omega$, $R_0>0$, then there exists $\eta \in \pr{0, 1}$, depending on $R_0$, and $C_{R_0}>0$ so that whenever $0 < R \le R_0$,
\begin{align}
\sup_{x, y \in B_{R/2}, x \ne y} \frac{\abs{\bu\pr{x} - \bu\pr{y}}}{\abs{x - y}^\eta}
&\le C_{R_0} R^{-\eta} \pr{\fint_{B_R} \abs{\bu}^{2^*}}^{\frac 1 {2^*}}
\label{eq3.48}
\end{align}
\end{itemize}

Notice that (IB) is \eqref{eq3.1} with $R_y=d_y$.
Note also that the solutions to $\mathcal{L} \bu = \bff$ and $\mathcal{L} \bu = \bz$ above are well-defined in the weak sense for the same reason as those in \eqref{eq3.1}. 

Existence of the fundamental solution may be obtained even when properties (IB) and (H) are replaced by the weaker assumption \eqref{eq3.1} (see Proposition~\ref{p3.5}).  
What is gained by property (IB) over \eqref{eq3.1} is a quantification of the constraint given by $R_y$.  The property (H) assures H\"older continuity and, in addition, helps to 
show that $\bG(x,y)=\bG^*(y,x)^T$, which leads to analogous estimates for $\bG(x, \cdot)$ as for $\bG(\cdot, y)$.  

\begin{defn}
We say that the matrix function $\bG\pr{x,y}= \pr{\Ga_{ij}\pr{x,y}}_{i,j=1}^N$ defined on $\set{\pr{x,y} \in \R^n \times \R^n : x \ne y}$ is the {\bf fundamental matrix} of $\mathcal{L}$ if it satisfies the following properties:
\begin{enumerate}
\item[1)] $\bG\pr{\cdot, y}$ is locally integrable and $\mathcal{L} \bG\pr{\cdot, y} = \de_y I$ for all $y \in \R^n$ in the sense that for every $\phi = \pr{\phi^1, \ldots, \phi^N}^T \in C^\iny_c\pr{\R^n}^{N}$,
\begin{align*}
&\int_{\R^n} A_{ij}^{\al \be} D_\be \Ga_{jk}\pr{\cdot, y} D_\al \phi^i + b_{ij}^\al \Ga_{jk}\pr{\cdot, y} D_\al \phi^i + d_{ij}^\be D_\be \Ga_{jk}\pr{\cdot, y} \phi^i + V_{ij} \Ga_{jk}\pr{\cdot, y} \phi^i 
= \phi^k\pr{y}.
\end{align*}
\item[2)] For all $y \in \R^n$ and $r > 0$, $\bG\pr{\cdot, y} \in Y^{1,2}\pr{\R^n \setminus B_r\pr{y}}^{N \times N}$. \\
\item[3)] For any $\bff = \pr{f^1, \ldots, f^N}^T \in L^\iny_c\pr{\R^n}^N$, the 
function $\bu = \pr{u^1, \ldots, u^N}^T$ given by
$$u^k\pr{y} = \int_{\R^n} \Ga_{jk}\pr{x,y} f^j\pr{x} \,dx$$
belongs to $\bF_0(\R^n)$ and satisfies $\mathcal{L}^* \bu = \bff$ in the sense that for 
every $\phi = \pr{\phi^1, \ldots, \phi^N}^T \in C^\iny_c\pr{\R^n}^{N}$,
\begin{align*}
&\int_{\R^n} A_{ij}^{\al \be} D_\al u^i D_\be \phi^j + b_{ij}^\al D_\al u^i \phi^j + d_{ij}^\be  u^i D_\be \phi^j + V_{ij}  u^i\phi^j 
= \int_{\R^n} f^j \phi^j.
\end{align*}
\end{enumerate}
We say that the matrix function $\bG\pr{x,y}$ is the {\bf continuous fundamental matrix} if it satisfies the conditions above and is also continuous.
\label{d3.3}
\end{defn}

\begin{rem}
As we will see below, we first establish the existence of a fundamental matrix using an application of Lemma~\ref{l3.2}.
With the additional assumption of H\"older continuity of solutions, we then show that our fundamental matrix is in fact a continuous fundamental matrix.
\end{rem}

We show here that there is at most one fundamental matrix.  
In general, we mean uniqueness in the sense of Lebesgue, i.e. almost everywhere uniqueness.
However, when we refer to the continuous fundamental matrix, we mean true pointwise equivalence.

Assume that $\bG$ and $\widetilde \bG$ are fundamental matrices satisfying Definition \ref{d3.3}. 
Then, for all $\bff\in L_c^\infty\pr{\Om}^N$, the functions $\bu$ and $\widetilde{\bu}$ given by
\begin{equation*}
u^k\pr{y}
=\int_{\R^n} \Gamma_{jk}\pr{x,y} f^j\pr{x} dx, \quad 
\widetilde{u}^k\pr{y}
=\int_{\R^n} \widetilde{\Gamma}_{jk}\pr{x,y} f^j\pr{x} dx
\end{equation*}
satisfy
\begin{equation*}
\mathcal{L}^* \pr{\bu - \widetilde{\bu}}=\bz \quad \text{in } \R^n
\end{equation*}
and $\bu-\widetilde{\bu}\in \bF_0(\R^n)$.  
By uniqueness of solutions ensured by the Lax-Milgram lemma, $\bu-\widetilde{\bu}\equiv \bf 0$.  
Thus, for a.e. $x\in \R^n$,
\begin{equation*}
\int_{\R^n} \brac{\Gamma_{jk}(x,y)-\widetilde{\Gamma}_{jk}(x,y)} f^j(x) dx=0, \quad 
\forall f\in L_c^\infty(\R^n)^N.
\end{equation*}
Therefore, $\bG = \widetilde \bG$ a.e. in $\set{x \ne y}$.
If we further assume that $\bG$ and $\widetilde \bG$ are continuous fundamental matrices, then we conclude that $\bG \equiv \widetilde \bG$ in $\set{x \ne y}$.

\begin{prop}\label{p3.5} 
Assume that {\rm\ref{A1}--\ref{A7}} and \eqref{eq3.1} hold.
Then there exists a fundamental matrix, $\bG(x,y)=(\Gamma_{ij}(x,y))_{i,j=1}^N, \, \set{x\ne y}$, unique in the Lebesgue sense, that satisfies Definition \ref{d3.3}.  
Furthermore, $\bG(x,y)$ satisfies the following estimates:
\begin{align}
&\norm{\bG(\cdot, y)}_{Y^{1,2}\pr{\R^n\setminus B_r(y)}} 
\le C r^{1-\frac{n}{2}}, \quad \forall r< \tfrac 1 2 R_y,
\label{eq3.49}\\
& \norm{\bG(\cdot, y)}_{L^q\pr{B_r(y)}} 
\le C_q r^{2-n+\frac{n}{q}}, \quad \forall q \in \left[1, \tfrac{n}{n-2}\right),  \quad \forall  r< \tfrac 1 2 R_y,
\label{eq3.50}\\
& \norm{D \bG\pr{\cdot, y}}_{L^{q}\pr{B_r\pr{y}}} 
\le C_q r^{1-n +\frac{n}{q}}, \qquad \forall q \in \left[ 1, \tfrac{n}{n-1}\right), \quad \forall r < \tfrac 1 2 R_y,
\label{eq3.51}\\
& \abs{\set{x \in \R^n : \abs{\bG\pr{x,y}} > \tau }} 
\le C \tau^{- \frac{n}{n-2}}, \quad \forall \tau >(\tfrac 1 2 R_y)^{2-n},
\label{eq3.52} \\
& \abs{\set{x \in \R^n : \abs{D_x \bG\pr{x,y}} > \tau }} 
\le C \tau^{- \frac{n}{n-1}}, \quad \forall \tau >(\tfrac 1 2 R_y)^{1-n},
\label{eq3.53} \\
& \abs{\bG\pr{x,y}} 
\le C R_{x,y}^{2 - n}, \qquad \text{where } R_{x,y}:=\min(R_x,R_y,|x-y|) ,
\label{eq3.54}
\end{align}
where each constant depends on $n$, $N$, $c_0$, $\Gamma$, $\gamma$, and the 
constants from \eqref{eq2.23} and \eqref{eq3.1}, and each $C_q$ depends additionally on $q$.
\end{prop}

\begin{pf}
By assumption, the hypotheses of Lemma~\ref{l3.2} are satisfied, and for each $y \in \R^n$, $0 < \rho < d_y$, and $k= 1,\dots, N$, we obtain $\{\bv_{\rho; y,k}\} \subset \bF_0\pr{\R^n}$ and $\bv_{y, k}$ satisfying properties \eqref{eq3.2}-\eqref{eq3.7} and the estimates \eqref{eq3.8}-\eqref{eq3.14}.

For each $y\in \R^n$, define $\bG^\rho\pr{\cdot, y}$ and $\bG\pr{\cdot, y}$ to be the $N\times N$ matrix functions whose k$^{th}$ columns are given by $\bv^T_{\rho; y, k}$ and $\bv^T_{y, k}$, respectively.  That $\bG$ is the fundamental matrix of $\mathcal{L}$ follows immediately from the conclusions of Lemma~\ref{l3.2}. 
One can also deduce from Lemma~\ref{l3.2} that $\bG(\cdot, y)$ satisfies \eqref{eq3.49}--\eqref{eq3.54} as a function of $x$.
\end{pf}

\begin{thm}\label{t3.6}
Assume that {\rm\ref{A1}--\ref{A7}} as well as properties {\rm{(IB)}} and {\rm{(H)}} hold.
Then there exists a unique continuous fundamental matrix, $\bG(x,y)=(\Gamma_{ij}(x,y))_{i,j=1}^N, \, 
\set{x\ne y}$, that satisfies Definition \ref{d3.3}.  
We have $\bG(x,y)= \bG^*(y,x)^T$, where $\bG^*$ is the unique continuous fundamental matrix 
associated to $\mathcal{L}^*$.  
Furthermore, $\bG(x,y)$ satisfies the following estimates:
\begin{align}
&\norm{\bG(\cdot, y)}_{Y^{1,2}\pr{\R^n\setminus B_r(y)}} 
+ \norm{\bG(x, \cdot)}_{Y^{1,2}\pr{\R^n\setminus B_r(x)}} 
\le C r^{1-\frac{n}{2}}, \quad \forall r>0,
\label{eq3.55} \\
&\norm{\bG(\cdot, y)}_{L^q\pr{B_r(y)}} 
+ \norm{\bG(x, \cdot)}_{L^q\pr{B_r(x)}} 
\le C_q r^{2-n+\frac{n}{q}}, \quad \forall q\in \left[1, \tfrac{n}{n-2}\right), \quad \forall r>0,
\label{eq3.56} \\
& \norm{D \bG\pr{\cdot, y}}_{L^{q}\pr{B_r\pr{y}}} 
+ \norm{D \bG\pr{x, \cdot}}_{L^{q}\pr{B_r\pr{x}}}
\le C_q r^{1-n +\frac{n}{q}}, \qquad \forall q \in \left[ 1, \tfrac{n}{n-1}\right), \quad \forall r>0,
\label{eq3.57} \\
& \abs{\set{x \in \R^n : \abs{\bG\pr{x,y}} > \tau}} 
+ \abs{\set{y \in \R^n : \abs{\bG\pr{x,y}} > \tau}}
\le C \tau^{- \frac{n}{n-2}}, \quad \forall \tau > 0,
\label{eq3.58} \\
& \abs{\set{x \in \R^n : \abs{D_x \bG\pr{x,y}} > \tau}} 
+ \abs{\set{y \in \R^n : \abs{D_y \bG\pr{x,y}} > \tau}}
\le C \tau^{- \frac{n}{n-1}},  \quad \forall \tau >0,
\label{eq3.59} \\
& \abs{\bG\pr{x,y}} \le C \abs{x - y}^{2 - n}, \qquad \forall x \ne y ,
\label{eq3.60}
\end{align}
where each constant depends on $n$, $N$, $c_0$, $\Gamma$, $\gamma$,  and the 
constants from \eqref{eq2.23} and {\rm{(IB)}}, and each $C_q$ depends additionally on $q$.  
Moreover, for any $0<R\le R_0<|x-y|$,
\aln{
&\abs{\bG\pr{x,y} - \bG\pr{z,y}} 
\le C_{R_0} C \pr{\frac{|x-z|}{R}}^\eta R^{2-n} 
\label{eq3.61}
}
whenever $|x-z|<\frac{R}{2}$ and
\aln{
&\abs{\bG\pr{x,y} - \bG\pr{x,z}} 
\le C_{R_0} C \pr{\frac{|y-z|}{R}}^\eta R^{2-n}
\label{eq3.62}
}
whenever $|y-z|<\frac{R}{2}$, where $C_{R_0}$ and $\eta=\eta(R_0)$ are the same as in assumption {\rm{(H)}}.
\end{thm}

\begin{pf} 
By our assumptions, Proposition~\ref{p3.5} holds with $R_y=\infty$ for all $y\in \R^n$.  
Let $\bG^\rho\pr{\cdot, y}$ and $\bG\pr{\cdot, y}$ be as in Proposition~\ref{p3.5}.

Fix $x, \, y \in \R^n$ and $0<R\le R_0< |x-y|$.  
Then $\mathcal{L} \bG\pr{\cdot, y}=\bz$ on $B_{R_0}\pr{x}$.  
Therefore, by assumption (H) and the pointwise bound \eqref{eq3.54}, whenever $\abs{x-z} < \frac{R}{2}$ we have
\alns{
\abs{\bG\pr{x,y} - \bG\pr{z,y}} 
&\le C_{R_0} \pr{\frac{|x-z|}{R}}^\eta C \norm{\bG(\cdot, y)}_{L^\infty\pr{B_{R}(x)}}
\le C_{R_0} C \pr{\frac{|x-z|}{R}}^\eta R^{2-n}.
}
This is the H\"older continuity of $\bG(\cdot, y)$ described by \eqref{eq3.61}.

Using the pointwise bound on $\bv_\rho$ in place of those for $\bv$, a similar statement holds for $\bG^\rho$ with $\rho\leq \frac 38 |x-y|$, and it follows that for any compact set $K \Subset \R^n\setminus\set{y}$, the sequence $\set{\bG^{\rho_\mu}\pr{\cdot, y}}_{\mu = 1}^\iny$ is equicontinuous on $K$.
Furthermore, for any such $K \Subset \R^n\setminus\set{y}$, there are constants 
$C_K < \iny$ and $\rho_K > 0$ such that for all $\rho < \rho_K$,
\begin{align*}
\norm{\bG^{\rho}\pr{\cdot, y}}_{L^\infty\pr{K}} \le C_K.
\end{align*}
Passing to a subsequence if necessary, we have that for any such compact $K 
\Subset \R^n\setminus\set{y}$,
\begin{equation}
\bG^{\rho_\mu}(\cdot, y) \to \bG\pr{\cdot,  y}
\label{eq3.63}
\end{equation}
uniformly on $K$.

We now aim to show 
\begin{equation*}
\bG\pr{x,y} = \bG^*\pr{y,x}^T,
\end{equation*}
where $\bG^*$ is the fundamental matrix associated to $\mathcal{L}^*$.
Let $\widehat\bv_{\si}=\widehat\bv_{\si; x,k}$ denote the averaged fundamental vector from Lemma~\ref{l3.2} associated to $\mathcal{L}^*$.  
By the same arguments used for $\bv_\rho$, we obtain a sequence $\set{\si_{\nu}}_{\nu =1}^\iny$, $\si_\nu \to 0$, such that $\widehat\bG^{\si_\nu}\pr{\cdot, x}$, a matrix whose $k$-th column is $\widehat\bv^T_{\si_\nu; x,k}$, converges to $\bG^*\pr{\cdot, x}$ uniformly on compact subsets of $\R^n \setminus \set{x}$, where $\bG^*\pr{\cdot, x}$  is a fundamental matrix for $\mathcal{L}^*$ that satisfies the properties analogous to those for $\bG\pr{\cdot, y}$.  
In particular, $\bG^*\pr{\cdot, x}$ is H\"older continuous.

By \eqref{eq3.2}, for $\rho_\mu$ and $\sigma_\nu$ sufficiently small,
\eq{
 \fint_{B_\rho\pr{y}} \widehat\Gamma^{\sigma}_{kl} \pr{\cdot, x} 
 = \mathcal{B}\brac{\bv_{\rho; y,l}, \widehat\bv_{\si; x,k}}
 = \mathcal{B}^*\brac{\widehat\bv_{\si; x,k}, \bv_{\rho; y,l}}
 = \fint_{B_\si\pr{x}} \Gamma^\rho_{lk}(\cdot, y).
 \label{eq3.64}
}
Define
$$g_{\mu \nu}^{kl} 
:= \fint_{B_{\rho_\mu}\pr{y}} \widehat\Ga_{kl}^{\si_\nu}\pr{\cdot, x} 
=  \fint_{B_{\si_\nu}\pr{x}} \Ga_{lk}^{\rho_\mu}\pr{\cdot, y} 
.$$
By continuity of $\Ga_{lk}^{\rho_\mu}\pr{\cdot, y}$, it follows that for any $x 
\ne y \in \R^n$,
$$\lim_{\nu \to \iny} g_{\mu \nu}^{kl} 
= \lim_{\nu \to \iny}  \fint_{B_{\rho_\mu}\pr{y}} \widehat\Ga_{kl}^{\si_\nu}\pr{\cdot, x} 
=  \Ga_{lk}^{\rho_\mu}\pr{x, y},$$
so that by \eqref{eq3.63},
$$\lim_{\mu \to \iny}\lim_{\nu \to \iny} g_{\mu \nu}^{kl} 
= \lim_{\mu \to \iny} \Ga_{lk}^{\rho_\mu}\pr{x, y} 
= \Ga_{lk}\pr{x, y}.$$
But by weak convergence in $W^{1,q}\pr{B_r\pr{y}}$, i.e., \eqref{eq3.3} with $R_y=\infty$,
$$\lim_{\nu \to \iny} g_{\mu \nu}^{kl} 
= \lim_{\nu \to \iny} \fint_{B_{\rho_\mu}\pr{y}} \widehat\Ga_{kl}^{\si_{\nu}}\pr{\cdot, x} 
= \fint_{B_{\rho_\mu}\pr{y}} \Ga_{kl}^* \pr{\cdot, x},$$
and it follows then by continuity of $\Ga_{kl}^*\pr{\cdot, x}$ that
$$\lim_{\mu \to \iny}\lim_{\nu \to \iny} g_{\mu \nu}^{kl} 
= \lim_{\mu \to \iny} \fint_{B_{\rho_\mu}\pr{y}} \Ga_{kl}^* \pr{\cdot, x} 
= \Ga_{kl}^* \pr{y, x}.$$
Therefore, for all $k , l \in \set{1, \ldots, N}$, ${x \ne y}$,
$$\Ga_{lk }\pr{x,y} = \Ga_{kl}^*\pr{y,x},$$
or equivalently, for all ${x \ne y}$,
\begin{equation}
\bG\pr{x,y} = \bG^*\pr{y,x}^T.
\label{eq3.65}
\end{equation}
Consequently, all the estimates which hold for $\bG\pr{\cdot, y}$ hold analogously for $\bG\pr{x, \cdot}$.
\end{pf}

\begin{rem}
We have seen that there is a subsequence $\set{\rho_{\mu}}_{\mu = 1}^ \iny$, $\rho_\mu \to 0$, such that $\bG^{\rho_\mu}\pr{x, y} \to \bG\pr{x, y}$ for all $x\in \R^n \setminus \set{y}$. 
In fact, a stronger fact can be proved.  By \eqref{eq3.64},
\begin{align*}
\Ga^{\rho}_{lk}\pr{x,y}
= \lim_{\nu \to \iny} \fint_{B_{\si_\nu}\pr{x}} \Ga^\rho_{lk}\pr{\cdot, y} 
= \lim_{\nu \to \iny} \fint_{B_\rho\pr{y}} \widehat\Ga_{kl}^{\si_{\nu}}\pr{\cdot, x} 
= \fint_{B_\rho\pr{y}} \Ga_{kl}^{ \, *}\pr{\cdot, x}. 
\end{align*}
By \eqref{eq3.65}, this gives
$$\Ga^{\rho}_{lk}\pr{x,y}  = \fint_{B_\rho\pr{y}} \Ga_{lk}\pr{x, z}dz.$$
By continuity, for all $x \ne y$,
\begin{equation}\label{eq3.66}
\lim_{\rho \to 0} \bG^\rho\pr{x,y} = \bG\pr{x,y}.
\end{equation}
\end{rem}

\begin{thm}\label{t3.8} 
Assume that {\rm\ref{A1}--\ref{A7}} as well as properties {\rm{(IB)}} and {\rm{(H)}} hold.
If $\bff\in \pr{L^{\frac{2n}{n+2}}\pr{\R^n} \cap L^\ell_{\loc}\pr{\R^n}}^N$ for some $\ell \in \pb{\frac n 2, \iny}$, then there exists a unique $\bu \in \bF_0(\R^n)$ that is a weak solution to $\mathcal{L} \bu = \bff$.  
Furthermore, we have
\eq{\label{eq3.67}
u^k\pr{x} = \int_{\R^n} \Gamma_{ki}\pr{x,y} f^i\pr{y}\, dy, \quad k=1,\dots,N.
}
for a.e. $x\in \R^n$.
\end{thm}

\begin{pf} 
We see from \eqref{eq3.17} that
\eqs{
\bF_0\pr{\R^n} \ni \bw \mapsto \int_{\R^n} \bff \cdot \bw
}
defines a bounded linear functional on $\bF_0\pr{\R^n}$.  
Therefore, the existence of a unique $\bu \in \bF_0\pr{\R^n}$ that is a weak solution to $\mathcal{L} \bu = \bff$ follows from the Lax-Milgram theorem.

By definition of a weak solution, we have
\eq{\label{eq3.68}
\int_{\R^n} \bff \cdot \widehat\bv_{\sigma} 
= \mathcal{B}\brac{\bu, \widehat\bv_{\sigma}} 
= \mathcal{B}^*\brac{\widehat\bv_{\sigma}, \bu} 
= \fint_{\Om_\sigma\pr{x}} u^k,
}
where $\widehat\bv_{\sigma} =\widehat\bv_{\sigma; \, y, k}$ is the averaged fundamental vector from Lemma~\ref{l3.2} associated to $\mathcal{L}^*$.  
Taking the limit in $\sigma$ of the left-hand side, we get
\eq{\label{eq3.69}
\lim_{\sigma \to 0} \int_{\R^n} \bff \cdot \widehat\bv_{\sigma} 
= \lim_{\si\to 0} \pr{\int_{B_1\pr{x}} \bff \cdot \widehat\bv_{\sigma } 
+ \int_{\R^n \setminus B_1 \pr{x}} \bff \cdot \widehat\bv_{\sigma }}
=\int_{\R^n} \bff \cdot \widehat\bv,
}
where $\widehat\bv $ is the $k$-th column of $\bG^*\pr{\cdot,x}$.  
Here, we have used \eqref{eq3.22} and $\bff \in L^\ell_{\loc}\pr{\R^n}^N$ for $\ell\in \pb{\frac n 2, \iny}$ to establish convergence of the first integral, and we have used \eqref{eq3.5} and $\bff\in L^{\frac{2n}{n+2}}\pr{\R^n}^N$ to establish convergence of the second integral.  
Combining \eqref{eq3.68} and \eqref{eq3.69}, we get
\eqs{
u^k\pr{x} = \int_{\R^n} \Gamma^*_{ik}\pr{y,x} f^i\pr{y} \, dy, \quad \text{for a.e. } x\in \R^n.
}
The conclusion \eqref{eq3.67} now follows from \eqref{eq3.65}.
\end{pf}

\subsection{Green matrix}

Here we show existence of the Green matrix of $\mathcal{L}$ on any connected open set $\Omega \subset \mathbb{R}^n$ with $n\ge 3$.

\begin{defn} 
Let $\Omega$ be an open, connected subset of $\RR^n$.
 We say that the matrix function $\bGr\pr{x,y} = \pr{G_{ij}\pr{x,y}}_{i,j=1}^N$ defined on the set $\set{(x,y)\in \Omega \times \Omega: x\ne y}$ is the \textbf{Green matrix} of $\mathcal{L}$ if it satisfies the following properties:
\begin{itemize}
\item[1)] $\bGr\pr{\cdot, y}$ is locally integrable and $\mathcal{L} \bGr\pr{\cdot, y} = \de_y I$ for all $y \in \Omega$ in the sense that for every $\phi = \pr{\phi^1, \ldots, \phi^N}^T \in C^\iny_c\pr{\Omega}^N$,
\begin{equation*}
\int_{\Om} A_{ij}^{\al \be} D_\be G_{jk}\pr{\cdot, y} D_\al \phi^i + b_{ij}^\al 
G_{jk}\pr{\cdot, y} D_\al \phi^i + d_{ij}^\be D_\be G_{jk}\pr{\cdot, y} \phi^i + 
V_{ij} G_{jk}\pr{\cdot, y} \phi^i 
= \phi^k\pr{y}
\end{equation*}
\item[2)] For all $y\in \Om$ and $r > 0$, $\bGr\pr{\cdot, y}\in Y^{1,2}\pr{\Om \setminus  \Omega_r\pr{y}}^{N \times N}$.  
In addition, $\bGr(\cdot,  y)$ vanishes on $\partial \Omega$ in the sense that for every $\zeta \in C_c^\infty\pr{\Om}$ satisfying $\zeta \equiv 1$ on $B_r(y)$ for some $r>0$, we have
\begin{equation*}
(1-\zeta)\bGr(\cdot, y) \in Y_0^{1,2}\pr{\Om \setminus \Omega_r(y)}^{N \times N}.
\end{equation*}
\item[3)] For any $\bff = \pr{f^1, \ldots, f^N}^T \in L^\iny_c\pr{\Omega}^N$, the function $\bu = \pr{u^1, \ldots, u^N}^T$ given by
$$u^k\pr{y} = \int_{\Om} G_{jk}\pr{x,y} f^j\pr{x} dx$$
belongs to $\bF_0\pr{\Om}$ and satisfies $\mathcal{L}^* \bu = \bff$ in the sense that for every $\phi = \pr{\phi^1, \ldots, \phi^N}^T \in C^\iny_c\pr{\Om}^N$,
\begin{align*}
&\int_{\Om} A_{ij}^{\al \be} D_\al u^i D_\be \phi^j + b_{ij}^\al D_\al u^i 
\phi^j + d_{ij}^\be  u^i D_\be \phi^j + V_{ij}  u^i\phi^j 
= \int_{\Om} f^j \phi^j.
\end{align*}
\end{itemize}
We say that the matrix function $\bGr\pr{x,y}$ is the \textbf{continuous Green matrix} if it satisfies the conditions above and is also continuous.
\label{d3.9}
\end{defn}

As in the case of the (continuous) fundamental matrix, and by the same argument, there exists at most one (continuous) Green matrix, where the sense of uniqueness is also as before.

\begin{thm}\label{t3.10} 
Let $\Om$ be an open, connected, proper subset of $\R^n$.  
Denote $d_x:=\dist(x, \del \Om)$ for $x\in \Om$.  
Assume that {\rm\ref{A1}--\ref{A7}} as well as properties {\rm(IB)} and {\rm (H)} hold.  
Then there exists a unique continuous Green matrix $\bGr(x,y)=(G_{ij}(x,y))_{i,j=1}^N$, defined in $\set{x,y\in \Om, x\neq y}$, that satisfies Definition \ref{d3.9}.  
We have $\bGr(x,y)=\bGr^*(y,x)^T$, where $\bGr^*$ is the unique continuous Green matrix associated to $\mathcal{L}^*$.  
Furthermore, $\bGr(x,y)$ satisfies the following estimates:
\begin{align}
& \norm{\bGr\pr{\cdot, y}}_{Y^{1,2}\pr{\Om \setminus B_r\pr{y}}}
\le C r^{1-n/2}, \qquad \forall r< \tfrac 1 2 d_y,
\label{eq3.70}\\
& \norm{\bGr\pr{\cdot, y}}_{L^{q}\pr{B_r\pr{y}}}
\le C_q r^{2 -n+ \frac{n}{q}}, \qquad \forall r < \tfrac 1 2 d_y, \quad \forall q\in [1,\tfrac{n}{n-2}), \\
& \norm{D \bGr\pr{\cdot, y}}_{L^{q}\pr{B_r\pr{y}}}
\le C_q r^{1-n +\frac{n}{q}}, \qquad \forall r< \tfrac 1 2 d_y, \quad \forall q \in [ 1, \tfrac{n}{n-1}), \\
& \abs{\set{x \in \Om : \abs{\bGr\pr{x,y}} > \tau }} 
\le C \tau^{- \frac{n}{n-2}}, \qquad \forall \tau> \pr{\tfrac 1 2 d_y}^{2-n},\\
& \abs{\set{x \in \Om : \abs{D_x \bGr\pr{x,y}} > \tau}} 
\le C \tau^{- \frac{n}{n-1}}, \qquad \forall \tau > \pr{ \tfrac 1 2 d_y}^{1-n},
\label{eq3.74}\\
&\norm{\bGr\pr{x,\cdot}}_{Y^{1,2}\pr{\Om \setminus B_r\pr{x}}} 
\le C r^{1-n/2}, \qquad \forall r< \tfrac 1 2 d_x,
\label{eq3.75} \\
&\norm{\bGr\pr{x,\cdot}}_{L^{q}\pr{B_r\pr{x}}} 
\le C_q r^{2 -n+ \frac{n}{q}}, \qquad \forall r < \tfrac 1 2 d_x, \quad \forall q \in [1,\tfrac{n}{n-2}), \\
& \norm{D \bGr\pr{x, \cdot}}_{L^{q}\pr{B_r\pr{x}}} 
\le C_q r^{1-n +\frac{n}{q}}, \qquad \forall r< \tfrac 1 2 d_x, \quad \forall q \in [ 1, \tfrac{n}{n-1}), \\
& \abs{\set{y \in \Om : \abs{\bGr\pr{x,y}} > \tau}} 
\le C \tau^{- \frac{n}{n-2}}, \qquad \forall \tau > \pr{\tfrac 1 2 d_x}^{2-n},\\
& \abs{\set{y \in \Om : \abs{D_y \bGr\pr{x,y}} > \tau}} 
\le C \tau^{- \frac{n}{n-1}}, \qquad \forall \tau > \pr{ \tfrac 1 2 d_x}^{1-n},
\label{eq3.79} \\
& \abs{\bGr\pr{x,y}} 
\le C d_{x,y}^{2 - n} \qquad \forall x \ne y, \text{ where } d_{x,y}:=\text{min}(d_x,d_y,|x-y|),
\label{eq3.80}
\end{align}
where the constants depend on $n, \, N, \,c_0, \, \Gamma, \, \gamma$, and the constants from \eqref{eq2.23} and {\rm{(IB)}}, and each $C_q$ depends additionally on $q$.  
Moreover, for any $0<R\le R_0< \frac 1 2 d_{x,y}$,
\aln{
& \abs{\bGr\pr{x,y} - \bGr\pr{z,y}} 
\le C_{R_0} C \pr{\frac{|x-z|}{R}}^\eta R^{2-n},
\label{eq3.81}
}
whenever $|x-z|<\frac{R}{2}$ and
\aln{
& \abs{\bGr\pr{x,y} - \bGr\pr{x,z}} 
\le C_{R_0} C \pr{\frac{|y-z|}{R}}^\eta R^{2-n},
\label{eq3.82}
}
whenever $|y-z|<\frac{R}{2}$, where $C_{R_0}$ and $\eta=\eta(R_0)$ are the same as in assumption {\rm{(H)}}.
\end{thm}

\begin{pf}
The hypotheses of Lemma~\ref{l3.2} are satisfied with $R_y=d_y$, for all $y\in \Omega$.
For each $y\in \Omega$, $0 < \rho < d_y$, and $k= 1,\dots, N$, we obtain $\{\bv_{\rho; y, k}\} \subset \bF_0\pr{\Om}$ and $\bv=\bv_{y,k}$ satisfying 
\eqref{eq3.2}-\eqref{eq3.7} and the estimates \eqref{eq3.8}-\eqref{eq3.14}, where $R_y=d_y$ and $R_{x,y}=\min\{d_x,d_y,|x-y|\}$.

We define $\bGr(\cdot, y)$ to be the matrix whose columns are given by $\bv^T_{y, k}$ for $k=1,\dots,N$, and we define similarly the averaged Green matrix $\bGr^\rho(\cdot, y)$.  
Then estimates \eqref{eq3.70}--\eqref{eq3.74} and \eqref{eq3.80} are inherited directly from Lemma~\ref{l3.2}.

We now prove that $\bGr(x,y)$ satisfies Definition \ref{d3.9}.  
This definition largely resembles that of the fundamental matrix, and the proof can be executed analogously, except for an additional requirement to prove that $\bGr(\cdot, y)= \bz$ on $\del \Om$ in the sense that for all $\zeta \in C_c^\iny\pr{\Om}$ satisfying $\zeta\equiv 1$ on $B_r(y)$ for some $r>0$, we have
\begin{equation}\label{eq3.83}
(1-\zeta)\bGr(\cdot, y) \in Y^{1,2}_0\pr{\Om}^{N\times N}.
\end{equation}
By Mazur's lemma, $Y^{1,2}_0\pr{\Om}^N$ is weakly closed in $Y^{1,2}\pr{\Om}^N$.  
Therefore, since $(1-\zeta)\bv_{\rho_{\mu}}=\bv_{\rho_{\mu}}- \zeta\bv_{\rho_{\mu}}  \in Y^{1,2}_0\pr{\Om}^N$ for all $\rho_\mu<d_y$, it suffices for \eqref{eq3.83} to show that
\begin{equation}\label{eq3.84}
(1-\zeta)\bv_{\rho_\mu} \rightharpoonup (1-\zeta) \bv \quad \text{in } Y^{1,2}\pr{\Om}^N.
\end{equation}
Since $(1-\zeta)\equiv 0$ on $B_r(y)$, the result \eqref{eq3.84} follows from 
\eqref{eq3.5}.  
Indeed,
\begin{align*}
\int_{\Om} (1-\zeta) G_{kl}(\cdot, y)\phi 
&=\int_{\Om} G_{kl} (\cdot, y)(1-\zeta) \phi 
= \lim_{\mu\to \iny} \int_{\Om} G^{\rho_\mu}_{kl}(\cdot, y) (1-\zeta) \phi \\
&=\lim_{\mu\to \iny} \int_{\Om} (1-\zeta) G^{\rho_\mu}_{kl} (\cdot, y) \phi, \quad 
\forall \phi \in L^{\frac{2n}{n+2}}\pr{\Om}, \quad \text{and} \\
\int_{\Om} D\brac{\pr{1-\zeta}G_{kl}(\cdot, y)}\cdot \psi 
&= -\int_{\Om} G_{kl}(\cdot, y) D\zeta \cdot \psi + \int_{\Om} DG_{kl}(\cdot, y) \cdot (1-\zeta)\psi \\
&=-\lim_{\mu\to \iny} \int_{\Om} G^{\rho_\mu}_{kl} (\cdot, y) D\zeta \cdot \psi + \lim_{\mu \to \iny} \int_{\Om} DG^{\rho_\mu}_{kl} (\cdot, y) \cdot (1-\zeta) \psi \\
&=\lim_{\mu \to \iny} \int_{\Om} D\brac{\pr{1-\zeta}G^{\rho_\mu}_{kl} (\cdot, y)}\cdot \psi, \quad 
\forall \psi \in L^2\pr{\Om}^N.
\end{align*}
Therefore, $\bGr(x,y)$ is the unique Green matrix associated to $\mathcal{L}$.

It follows from \eqref{eq3.70} and property (H) that for any $0<R \le R_0 \le \frac 1 2 d_{x,y}$, there exists $\eta=\eta(R_0)$ and $C_{R_0}>0$ such that, whenever $|x-z|\le \frac{R}{2}$,
\aln{
& \abs{\bGr\pr{x,y} - \bGr\pr{z,y}} 
\le C_{R_0} C \pr{\frac{|x-z|}{R}}^\eta R^{2-n}. 
\label{eq3.85}
}
By the same argument that lead to \eqref{eq3.63}, this implies that, passing to a subsequence if necessary, for any compact $K\Subset \Om \setminus \{y\}$,
\begin{equation}\label{eq3.86}
\bGr^{\rho_\mu}(\cdot, y)\to \bGr(\cdot, y)
\end{equation}
uniformly on $K$, and from here the same argument as the one for \eqref{eq3.65} proves 
that
\begin{equation}
\bGr(x,y)=\bGr^*(y,x)^T, \quad \forall x,y\in \Om, \quad x \ne y.
\label{eq3.87}
\end{equation}
The remaining properties, \eqref{eq3.75}--\eqref{eq3.79}, follow from Lemma~\ref{l3.2} applied to $\bGr^*\pr{\cdot, x}$ in combination with \eqref{eq3.87}.
\end{pf}

\begin{rem} As with the fundamental matrix, we obtain
\begin{equation}
\bGr^{\rho}(x,y)=\fint_{\Om_\rho(y)} \bGr(x,z) dz,
\end{equation}
and, by continuity, 
\begin{equation}
\lim_{\rho\to 0} \bGr^{\rho}(x,y)=\bGr(x,y), \quad \forall x,y\in \Om, 
\quad x\ne y.
\end{equation}
\end{rem}

\subsection{Global estimates for the Green matrix}

It was observed in \cite{KK10} that if the interior boundedness assumption (IB) is altered as below (to being valid on balls possibly intersecting the boundary), then the pointwise and local $L^q$ estimates of $\bGr$ can be freed of their dependence on the distances to the boundary for the homogeneous elliptic operators. 
Similarly, assuming local boundedness on boundary balls gives enhanced Green function estimates in our setting. 
\begin{itemize}
\item[(BB)] Let $\Om$ be a connected open set in $\R^n$.  
We say that \rm{(BB)} holds in $\Om$ if whenever $\bu\in \bF\pr{\Om_{2R}}$ is a weak solution to $\mathcal{L} \bu = \bff$ or $\mathcal{L}^*\bu=\bff$ in $\Om_R$, for some $R>0$, where $\bff\in L^\ell\pr{\Om_R}^N$ for some $\ell \in \pb{\tfrac{n}{2}, \iny}$, and $\bu \equiv \bf 0$ on $\partial \Om \cap B_R$, then $\bu$ is a bounded function and for any $q>0$,
\begin{equation}\label{eq3.90}
\norm{\bu}_{L^\iny(\Om_{R/2})}\le C \brac{ R^{-\frac n q}\norm{\bu}_{L^q\pr{\Om_R}}
+ R^{2-\frac{n}{\ell}} \norm{\bff}_{L^\ell(\Om_R)}},
\end{equation}
where the constant $C$ is independent of $R$.
\end{itemize}

We note that condition (BB) holds, for example, whenever (IB) holds for an extended operator $\mathcal{L}^\#$ defined on $\R^n$ with $\mathcal{L}=\mathcal{L}^\#$ on $\Om$.  
This fact can often be established by a reflection argument.

\begin{cor}\label{c3.12} 
Let $\Om$ be an open, connected, proper subset of $\R^n$.  
Assume that {\rm{\ref{A1}--\ref{A7}}} as well as properties {\rm{(BB)}} and {\rm{(H)}} hold.
Then the continuous Green matrix satisfies the following global estimates:
\begin{align}
& \norm{\bGr\pr{\cdot, y}}_{Y^{1,2}\pr{\Om \setminus B_r\pr{y}}}
+\norm{\bGr\pr{x,\cdot}}_{Y^{1,2}\pr{\Om \setminus B_r\pr{x}}} 
\le C r^{1-n/2}, \qquad \forall r>0,\\
& \norm{\bGr\pr{\cdot, y}}_{L^{q}\pr{B_r\pr{y}}}
+\norm{\bGr\pr{x,\cdot}}_{L^{q}\pr{B_r\pr{x}}} 
\le C_q r^{2 -n+ \frac{n}{q}}, \qquad \forall r>0, \quad \forall q \in [1,\tfrac{n}{n-2}), \\
& \norm{D \bGr\pr{\cdot, y}}_{L^{q}\pr{B_r\pr{y}}}
+\norm{D \bGr\pr{x,\cdot}}_{L^{q}\pr{B_r\pr{x}}} 
\le C_q r^{1-n +\frac{n}{q}}, \qquad \forall r>0, \quad \forall q \in [ 1, \tfrac{n}{n-1}), \\
& \abs{\set{x \in \Om : \abs{\bGr\pr{x,y}} > \tau }}
+\abs{\set{y \in \Om : \abs{\bGr\pr{x,y}} > \tau }} 
\le C \tau^{- \frac{n}{n-2}}, \qquad \forall \tau >0,\\
& \abs{\set{x \in \Om : \abs{D_x \bGr\pr{x,y}} > \tau }}
+\abs{\set{y \in \Om : \abs{D_y \bGr\pr{x,y}} > \tau }} 
\le C \tau^{- \frac{n}{n-1}}, \qquad \forall \tau >0,\\
& \abs{\bGr\pr{x,y}} \le C |x-y|^{2 - n} \qquad \forall x \ne y,
\end{align}
where the constants depend on $n, \, N, \, c_0 \, \Gamma, \, \gamma$ and the constants from \eqref{eq2.23} and {\rm{(BB)}}, and each $C_q$ depends additionally on $q$.  
The H\"older continuity estimates of Theorem~\ref{t3.10} remain unchanged.
\end{cor}

\begin{pf}
As in the proof of Theorem~\ref{t3.10}, the global estimates are inherited directly from Lemma~\ref{l3.2} with $R_x, R_y=\infty$ for all $x, y\in \Om$.
\end{pf}

\begin{rem}
In conclusion, $\bG\pr{x,y}$ exists and satisfies the estimates of Theorem~\ref{t3.6} whenever (IB) and (H) hold for solutions.
The conclusion of Theorem~\ref{t3.6} also states that
 \eqs{\bG\pr{\cdot,y}\in Y^{1,2}\pr{\R^n\setminus B_r\pr{y}}^{N\times N} \quad \text{for any $r>0$.}}
However, it does not follow from Theorem~\ref{t3.6} that $\bG\pr{\cdot, y} \in \bF\pr{\R^n \setminus B_r\pr{y}}$ for the general space $\bF$.
In Section~\ref{s7}, we examine a number of examples and show that in each case, a version of this statement holds for $\bG\pr{\cdot,y}$ as well as $\bG\pr{x, \cdot}$, $\bGr\pr{\cdot, y}$, and $\bGr\pr{x, \cdot}$.
Details may be found in Section~\ref{s7.4}
\end{rem}

\section{A Caccioppoli inequality}

The remainder of the paper will essentially be a discussion of the major examples that fit our theory. 
In this section we prove a version of the Caccioppoli inequality. In the next two sections we demonstrate local boundedness and H\"older continuity of solutions (for equations, rather than systems, only). 
And finally, in Section~\ref{s7}, we tie it all together by presenting the most common examples. 

\begin{lem} \label{l4.1}
Let $\Om \su \R^n$ be open and connected.
Assume that $\bF\pr{\Om}$, $\bF_0\pr{\Om}$, $\mathcal{L}$, and $\mathcal{B}$ satisfy \rm{\ref{A1} -- \ref{A5}}.
Suppose $\bb \in L^s\pr{\Omega}^{n \times N \times N}$, $\bd \in L^t\pr{\Omega}^{n \times N \times N}$ for some $s, t \in \brac{ n, \iny}$, and (instead of assuming \rm{\ref{A6}}),  assume that either
\begin{itemize}
\item $s, t = n$ and $\mathcal{B}\brac{\bv, \bv} \ge \ga \norm{D \bv}_{L^2\pr{\Om}^N}^2$ for every $\bv \in \bF_0\pr{\Om}$; or
\item $s, t \in (n, \iny]$ and $\mathcal{B}\brac{\bv, \bv} \ge \ga \norm{\bv}_{W^{1,2}\pr{\Om}^N}^2$ for every $\bv \in \bF_0\pr{\Om}$.
\end{itemize}
Let $\bu \in \bF\pr{\Omega}$ and $\zeta \in C^\infty(\rn)$ with $D\zeta \in C_c^\infty(\rn)$
be such that $\bu \zeta  \in \bF_0\pr{\Om}$, $ \partial^i\zeta \,\bu\in L^2(\Omega)^N$, $i=1,...,n$, 
and $\disp \mathcal{B}\brac{\bu, \bu \zeta^2} \le \int \bff \, \cdot \bu\, \zeta^2$ for some $\bff \in L^{\ell}\pr{\Omega}^N$,  $\ell \in \pb{ \frac n 2, \iny}$. 
Then
\begin{equation}\label{eq4.1}\int \abs{D \bu}^2 \zeta^2 \le C \int \abs{\bu}^2 \abs{D \zeta}^2 + c \abs{\int  \bff\cdot \bu \,\zeta^2 },
\end{equation}
where $C = C\pr{n, s, t, \ga, \La, \norm{\bb}_{L^s\pr{\Om}}, \norm{\bd}_{L^t\pr{\Om}}}$, $c = c\pr{\ga}$.
\end{lem}

\begin{rem}
Let us make a few comments before the proof. 
First, as in the comments to \rm{\ref{A7}}, we remark that the condition $D\zeta \in C_c^\infty(\rn)$ implies that $\zeta$ is a constant outside some large ball (call it $C_\zeta$) and hence, $C_\zeta-\zeta \in C_c^\infty(\rn)$.
Then, by \rm{\ref{A4}}, $\bu\zeta^2=C_\zeta \bu\zeta-(C_\zeta-\zeta) \bu\zeta \in \bF_0\pr{\Omega}$. 
We shall use this in the proof. 
Also, the conditions $\bu \zeta \in \bF_0\pr{\Om}$, $ \partial^i\zeta \,\bu\in L^2(\Omega)^N$, $i=1,...,n$, and $D\zeta \in C_c^\infty(\rn)$, along with \eqref{eq2.5}, ensure that the first and the second integral in \eqref{eq4.1} are finite. The last one is finite for otherwise both the assumptions and the conclusion of the Lemma are meaningless.  

Second, if we do assume \rm{\ref{A6}}, then the condition $\mathcal{B}\brac{\bv, \bv} \ge \ga \norm{D \bv}_{L^2\pr{\Om}^N}^2$ for every $\bv \in \bF_0\pr{\Om}$ follows from \eqref{eq2.5}. 
Moreover, the actual requirements on $\bb$ and $\bd$ that are necessary to carry out the arguments, and appear in the constant $C$, are $\bb \in L^s\pr{\Omega\cap U}^{n \times N \times N}$, $\bd \in L^t\pr{\Omega\cap U}^{n \times N \times N}$ for any $U$ containing the support of $D\zeta$. 
Since the latter is compact, one could always reduce the case $s,t>n$ to the case $s=t=n$ and hence to work in the first regimen. However, such a reduction would bring up the dependence of the constants on the size of the support of $D\zeta$, and this is typically not desirable. 
\end{rem}

\begin{proof}
Let $\bu$, $\zeta$ be as in the statement.
A computation shows that
\begin{align*}
\mathcal{B}\brac{ \bu \zeta, \bu \zeta}
&= \mathcal{B}\brac{\bu, \bu \zeta^2} 
+ \int \bA^{\al \be}\brac{ \pr{- D_\be \bu \cdot \bu \, D_\al \zeta + \bu \, D_\be \zeta \cdot D_\al \bu} \zeta + \bu \cdot \bu \, D_\be \zeta \, D_\al \zeta} \\
&+ \int \pr{- \bb^\al \bu \zeta \cdot \bu \, D_\al \zeta
+ \bd^\be \bu D_\be  \zeta \cdot \bu \zeta}.
\end{align*}
By the assumption, $\disp \mathcal{B}\brac{\bu, \bu \zeta^2} \le \int \abs{\bff} \abs{\bu} \zeta^2$. \\
By \eqref{eq2.13},
\begin{align*}
&\int \bA^{\al \be}\brac{ \pr{- D_\be \bu \cdot \bu \, D_\al \zeta + \bu \, D_\be \zeta \cdot D_\al \bu} \zeta + \bu \cdot \bu \, D_\be \zeta \, D_\al \zeta} \\
&\le 2 \La \int \abs{D \bu} \abs{D \zeta} \abs{\bu} \eta + \La \int \abs{D \zeta}^2 \abs{\bu}^2 
\le \pr{ \frac{8 \La^2}{\ga}  + \La} \int \abs{\bu}^2 \abs{D \zeta}^2 
+ \frac{\ga}{8 } \int \abs{D \bu}^2 \zeta^2. 
\end{align*}
If $s \in \pr{n, \iny}$, then since $\bu \zeta \in \bF_0\pr{\Om}$,
\begin{align*}
\abs{ \int \bb^\al \bu \zeta \cdot \bu \, D_\al \zeta } 
&\le \int \abs{\bb} \abs{\bu \zeta}^{\frac n s} \abs{\bu \zeta}^{1 - \frac n s}\abs{ \bu D \zeta}
\le \norm{\bb}_{L^s\pr{\Om}} \norm{\bu \zeta}_{L^{2^*}\pr{\Om}}^{\frac n s}  \norm{\bu \zeta}_{L^{2}\pr{\Om}}^{1 - \frac n s} \norm{\bu D\zeta}_{L^{2}\pr{\Om}}  \\
&\le c_{n}^{\frac n s} \norm{\bb}_{L^s\pr{\Om}} \norm{D\pr{\bu \zeta}}_{L^{2}\pr{\Om}}^{\frac n s}  \norm{\bu \zeta}_{L^{2}\pr{\Om}}^{1 - \frac n s} \norm{\bu D\zeta}_{L^{2}\pr{\Om}}  \\
&\le \frac{\ga}{4 } \norm{D\pr{\bu \zeta}}_{L^{2}\pr{\Om}}^2
+ \frac{\ga}{2} \norm{\bu \zeta}_{L^{2}\pr{\Om}}^2
+ \frac{C_{n,s}}{ \ga} \norm{\bb}_{L^s\pr{\Om}}^2  \int \abs{\bu}^2 \abs{ D \zeta }^2.
\end{align*}
Similarly, if $s = \iny$, then
\begin{align*}
\abs{ \int \bb^\al \bu \zeta \cdot \bu \, D_\al \zeta } 
&\le \norm{\bb}_{L^\iny\pr{\Om}} \norm{\bu \zeta}_{L^{2}\pr{\Om}} \norm{\bu D\zeta}_{L^{2}\pr{\Om}}
\le \frac{\ga}{2} \norm{\bu \zeta}_{L^{2}\pr{\Om}}^2
+ \frac{1}{2 \ga}  \norm{\bb}_{L^\iny\pr{\Om}}^2  \int \abs{\bu}^2 \abs{ D \zeta }^2.
\end{align*}
Finally, if $s = n$, then
\begin{align*}
\abs{ \int \bb^\al \bu \zeta \cdot \bu \, D_\al \zeta } 
&\le \int \abs{\bb} \abs{\bu \zeta} \abs{ \bu D \zeta}
\le \norm{\bb}_{L^n\pr{\Om}} \norm{\bu \zeta}_{L^{2^*}\pr{\Om}} \norm{\bu D\zeta}_{L^{2}\pr{\Om}}  \\
&\le \frac{\ga}{4 } \int \abs{D\pr{\bu \zeta}}^2
+ \frac{c_n^2}{ \ga} \norm{\bb}_{L^n\pr{\Om}}^2  \int \abs{\bu}^2 \abs{ D \zeta }^2.
\end{align*}
Analogous inequalities hold for $\bd$. 

It follows from the inequalities above and the coercivity assumption on $\mathcal{B}$ that
\begin{align*}
& \frac{\ga}4 \int \abs{D \bu}^2 \zeta^2 
- \frac \ga 2 \int \abs{\bu}^2 \abs{D \zeta}^2
\le \frac{\ga}{2} \int \abs{D\pr{\bu \zeta}}^2 \\
&\le  \pr{ \frac{8 \La^2}{\ga}  + \La + \frac{C_{n,s} \norm{\bb}_{L^s\pr{\Om}}^2}{ \ga} + \frac{C_{n, t} \norm{\bd}_{L^t\pr{\Om}}^2}{ \ga} } \int \abs{\bu}^2 \abs{D \zeta}^2 
+ \frac{\ga}{8 } \int \abs{D \bu}^2 \zeta^2
+ \int \abs{\bff} \abs{\bu} \zeta^2,
\end{align*}
which leads to the claimed inequality after rearrangements.
\end{proof}

\section{Local boundedness in the equation setting}
\label{s5}

For general elliptic systems, homogeneous or not, (IB), (BB), (H), or even the fact of local boundedness of solutions may fail.
For counterexamples, we refer to \cite{MNP82} for dimension $n \ge 5$ and \cite{Fre08} for lower dimensions. In this and the next section we discuss the cases when local boundedness is valid, restricting ourselves to the context of equations rather than systems, i.e., to $N=1$. 
We insist that such a restriction is taken in Sections~\ref{s5} and \ref{s6} only and that this restriction is not necessary in order for (IB), (BB), (H) to hold. 
Nonetheless, it is perhaps the most commonly used application. 
Much of the material in Sections~\ref{s5} and \ref{s6}, or at least analogous arguments, have appeared in classical literature (e.g., \cite{GT01}, \cite{HL11}, \cite{Sta65}). 
However, we have to carefully track the constants, the exact nature of dependence on $\bb$, $\bd$, $\bV$, the impact of coercivity, and the resulting scale-invariance, since this is crucial for building the fundamental solutions.
Therefore, for completeness, we present the full arguments. 

The following lemma gives a scale-invariant (independent of the choice of $R$) version of local boundedness.
To prove the lemma, we will use de Giorgi's approach, as explained in \cite{HL11}, \cite{Sta65}. 
The novelty of our argument is that rather than assuming ellipticity of the homogeneous operator, we assume coercivity of the bilinear form associated to the full operator.
This allows us to prove a scale-invariant version of local boundedness under a certain sign assumption on the lower order terms.
In other words, we avoid picking up dependencies on the size of the domain over which we are working.
Recall that $\Om_R = B_R \cap \Om$.

We continue to work in the abstract framework that was first introduced in Section~\ref{s2}, but we will have to impose some further properties on our function spaces in order to show that local boundedness and interior H\"older continuity are in fact reasonable assumptions.
\begin{enumerate}[label=B\arabic*)]
\item
For any $R > 0$, $k \ge 0$, if $u \in \bF\pr{\Om_R}$ satisfies $u = 0$ along $\del \Om \cap B_R$ (as usual, in the sense of \eqref{eq2.10}--\eqref{eq2.11}), then $\zeta\pr{u - k}_+ \in \bF_0\pr{\Omega_R}$, $\partial^i\zeta\pr{u - k}_+ \in L^2\pr{\Omega_R}$, $i=1,...,n$, for any non-negative $\zeta\in C_c^\infty(B_{R})$, where $\pr{u-k}_+ := \max\set{u-k, 0}$.
\label{B1}
\item 
For any ball $B_R \su \rn$, $R>0$, if $u \in \bF\pr{B_R}$ is non-negative, and $k, \om > 0$, then $\pr{u + k}^{-\om} \in \bF\pr{B_R}$.
\label{B2}
\end{enumerate}

\begin{lem}
Let $\Om \su \R^n$ be open and connected and take $N = 1$.
Assume that $\bF\pr{\Om}$, $\bF_0\pr{\Om}$, $\mathcal{L}$, and $\mathcal{B}$ satisfy \rm{\ref{A1} -- \ref{A5}} and \rm{\ref{B1}}.
Suppose $b \in L^s\pr{\Om_R}^n$, $d \in L^t\pr{\Om_R}^n$ for some $s, t \in \brac{ n, \iny}$, and (instead of assuming \rm{\ref{A6}}) assume that either
\begin{itemize}
\item $s, t = n$ and $\mathcal{B}\brac{v, v} \ge \ga \norm{D v}_{L^2\pr{\Om_R}}^2$ for every $v \in \bF_0\pr{\Om_R}$; or
\item $s, t\in (n, \infty]$ and $\mathcal{B}\brac{v, v} \ge \ga \norm{v}_{W^{1,2}\pr{\Om_R}}^2$ for every $v \in \bF_0\pr{\Om_R}$.
\end{itemize}
Assume also that 
\begin{equation}
V - D_\al b^\al \ge 0 \textrm{ in } \Om_R \textrm{ in the sense of distributions}.
\label{eq5.1}
\end{equation}
Let $u \in \bF\pr{\Om_{2R}}$ satisfy $u = 0$ along $\del \Om \cap B_{2R}$.
Let $f \in L^{\ell}\pr{\Om_{R}}$ for some $\ell \in \pb{ \frac n 2, \iny}$ and assume that $\mathcal{L} u \le f$ in $\Om_{R}$ weakly in the sense that for any $\vp \in \bF_{0}\pr{B_{R}}$ such that $\vp \ge 0$ in $\Om_{R}$, we have 
\begin{align}
\mathcal{B}\brac{u, \vp} \le \int f \vp .
\label{eq5.2}
\end{align}
Then $u^+ \in L_{loc}^\iny\pr{\Om_R}$ and for any $r < R$, $q > 0$,
\begin{align}
\sup_{\Om_{r}} u^+ 
&\le \frac{C}{\pr{R- r}^{\frac n {q}}} \norm{u^+}_{L^{q}\pr{\Om_R}} 
+ c_{q} R^{2 - \frac n {\ell}}\norm{f}_{L^{\ell}\pr{\Om_R}} ,
\label{eq5.3}
\end{align}
where $C = C\pr{n, q, s, t, \ell, \ga,  \La, \norm{b}_{L^s\pr{\Om_R}}, \norm{d}_{L^t\pr{\Om_R}}}$ and $c_{q}$ depends only on $q$.
\label{l5.1}
\end{lem}

\begin{rem} Let us remark that \eqref{eq5.3} is of course vacuous if $\norm{u^+}_{L^{q}\pr{\Om_R}}$ is not finite. 
In practice, however, this is not a concern because in any ball of radius strictly smaller than $R$, the norm is finite and hence we can apply \eqref{eq5.3} in such a ball. 
Indeed, $\norm{u^+}_{L^{q}\pr{\Om_R}} < \iny$ for any $u \in \bF\pr{\Om_{2R}}$ by \eqref{eq2.8}. 
Therefore, by applying \eqref{eq5.3} with $q=2$, we conclude that $\norm{u^+}_{L^{\infty}\pr{\Om_r}}$ is finite for any $r<R$. 
Hence, $\norm{u^+}_{L^{q}\pr{\Om_r}}$ is finite for any $r<R$. 
Below, we will first prove \eqref{eq5.3} with $q=2$ and then assume that $\norm{u^+}_{L^{q}\pr{\Om_R}}$ is finite.
(Again, one can always take a slightly smaller ball if necessary).
\end{rem}

\begin{rem}
If $\Om_R = B_R$, then $\del \Om \cap \Om_R$ is empty so that the boundary condition on $u$ is vacuously satisfied.
Therefore, this version of local boundedness is applicable for all of our settings, i.e. when we are concerned with the boundary and when we are not.
\end{rem}

\begin{rem}
As previously pointed out, the estimate \eqref{eq5.3} is scale-invariant since it doesn't depend on $R$.
In our applications, we will assume that $b \in L^s\pr{\Om}^n$ and $d \in L^t\pr{\Om}^n$.
Since $\norm{b}_{L^s\pr{\Om_R}} \le \norm{b}_{L^s\pr{\Om}}$ and $\norm{d}_{L^t\pr{\Om_R}} \le \norm{d}_{L^t\pr{\Om}}$ for every $R$, then this lemma shows that we may establish local bounds with constants that are independent of the subdomain, $\Om_R$.
\end{rem}

\begin{proof}
We will first prove the case of $q = 2$ and $r=\frac 12 R$.
Fix $\zeta\in C_c^\infty(B_R)$, a cutoff function for which $0 \le \zeta \le 1$.
For some $k \ge 0$, define $v = \pr{u - k}_+$.
By \rm{\ref{B1}}, $v \zeta, v \zeta^2 \in \bF_0\pr{\Om_R}$.
Lemma~7.6 from \cite{GT01} implies that $Dv=Du$ for $u>k$ and $Dv=0$ for $u\leq k$ (since \eqref{eq2.8} implies that $v$ is weakly differentiable on $\Omega$).

Since $V - D_\al b^\al \ge 0$ in the sense of distributions and $\supp (v\zeta^2)$ is a subset of $\{u\geq k\}$, then 
\begin{align*}
\mathcal{B}\brac{v, v \zeta^2}
&= \int \pr{A^{\al \be} D_\be v + b^\al v} D_\al \pr{v \zeta^2} +  \pr{d^\be D_\be v + V v } v \zeta^2 \\
&=  \mathcal{B}\brac{u, v \zeta^2}  - k \int \pr{V - D_\al b^\al}  v \zeta^2
\le  \int f v \zeta^2,
\end{align*}
where we used \eqref{eq5.2} with $\vp := v \zeta^2 \in \bF_{0 }\pr{\Om_R}$, $\vp \ge 0$ to get the last inequality.

Since $v\zeta \in \bF_0\pr{\Om_R}$, $v\partial_i\zeta \in L^2(\Omega_R)$, and $D\zeta$ is compactly supported, then Lemma~\ref{l4.1} is applicable with $\bu = v$.
It follows that
\begin{align*}
\int \abs{D v}^2 \zeta^2 
&\le  \brac{ \pr{\frac{8 \La}{\ga}}^2  + \frac{8 \La}{\ga} + 4 + 8\frac{C_{n,s} \norm{b}_{L^s\pr{\Om_R}}^2 + C_{n, t} \norm{d}_{L^t\pr{\Om_R}}^2}{ \ga^2} } \int \abs{v}^2 \abs{D \zeta}^2 
+ \frac{8}{\ga} \int \abs{f} \abs{v} \zeta^2.
\end{align*}
By H\"older and Sobolev inequalities with $2^* = \frac{2n}{n-2}$,
\begin{align}
\int \abs{f} v \zeta^2
&\le \pr{\int \abs{f}^{\ell}}^{\frac 1 {\ell}} \pr{\int \abs{v \zeta}^{2^*}}^{\frac 1{ 2^*}} 
\abs{\set{v \zeta \ne 0}}^{1 - \frac{1}{\ell} - \frac{1}{2^*}} \nonumber \\
&\le \frac{\ga}{32} \int \abs{D\pr{v \zeta}}^{2}
+  \frac{8 c_n^2 }{\ga} \norm{f}_{L^\ell\pr{\Om_R}}^2 \abs{\set{v \zeta \ne 0}}^{1 + \frac 2 n - \frac{2}{\ell}}.
\label{eq5.4}
\end{align}
Therefore,
\begin{align*}
\int \abs{D\pr{v \zeta}}^2 
&\le  4\brac{ \pr{\frac{8 \La}{\ga}}^2  + \frac{8 \La}{\ga} + 5 + 8\frac{C_{n,s} \norm{b}_{L^s\pr{\Om_R}}^2 + C_{n, t} \norm{d}_{L^t\pr{\Om_R}}^2}{ \ga^2} } \int \abs{v}^2 \abs{D \zeta}^2 \\
&+ \pr{\frac{16 c_n }{\ga} \norm{f}_{L^\ell\pr{\Om_R}}}^2 \abs{\set{v \zeta \ne 0}}^{1 + \frac 2 n - \frac{2}{\ell}}.
\end{align*}
Since the H\"older and Sobolev inequalities imply that
\begin{align*}
\int \pr{v \zeta}^2 
\le \pr{\int \pr{v \zeta}^{2^*}}^{2/2^*} \abs{\set{ v \zeta \ne 0}}^{1 - \frac{2}{2^*}}
\le c_n^2 \int \abs{D\pr{v \zeta}}^{2} \abs{\set{ v \zeta \ne 0}}^{\frac{2}{n}},
\end{align*}
then
\begin{align}
\int \pr{v \zeta}^2 
&\le \frac{C_1}{4} \int v^2 \abs{D\zeta}^{2} \abs{\set{ v \zeta \ne 0}}^{\eps_1}  
+ C_1 F^2 \abs{\set{v \zeta \ne 0}}^{1 + \eps_2},
\label{eq5.5}
\end{align}
where $\eps_1 = \frac 2 n$, $\eps_2 = \frac 4 n - \frac 2 \ell > 0$, $F = \norm{f}_{L^\ell\pr{\Om_R}}$, and
\begin{align}
C_1 &= 16 c_n^2 \brac{ \pr{\frac{8 \La}{\ga}}^2  + \frac{8 \La}{\ga} + 5 + 8\frac{C_{n,s} \norm{b}_{L^s\pr{\Om_R}}^2 + C_{n, t} \norm{d}_{L^t\pr{\Om_R}}^2}{ \ga^2} }
+ \pr{\frac{16 c_n^2}{\ga}}^2.
\label{eq5.6}
\end{align}

For fixed $0 < r \le \rho \le R$, let $\zeta \in C^\iny_c\pr{B_\rho}$ be such that $\zeta \equiv 1$ in $B_r$ and $\abs{D\zeta} \le \frac{2}{\rho - r}$ in $B_R$.
We let $A\pr{k, r} = \set{ x \in \Om_r : u \ge k} = \supp v \cap \Om_r$.
Then, for any $0 < r < \rho \le R$ and $k \ge 0$, it follows from \eqref{eq5.5} that 
\begin{align}
\int_{A\pr{k,r}} \pr{u - k}^2
&\le C_1 \brac{\frac{\abs{A\pr{k, \rho}}^{\eps_1}}{\pr{\rho -r}^2} \int_{A\pr{k, \rho}}  \pr{u -k}^2 
+ F^2 \abs{A\pr{k, \rho}}^{1+\eps_2} }.
\label{eq5.7}
\end{align}

Considering $r=R/2$, the goal is to show that there exists a $k \ge 0$ such that 
$$\int_{A\pr{k, R/2}} \pr{u- k}^2 = 0.$$
Take $h > k\ge 0$ and $0 < r < R$.
Since $A\pr{k, r} \supset A\pr{h, r}$, then
$$\int_{A\pr{h, r}} \pr{u- h}^2 \le \int_{A\pr{k, r}} \pr{u- k}^2$$
and
$$\abs{A\pr{h, r}} = \abs{B_r \cap \set{u - k > h - k}} \le \frac{1}{\pr{h-k}^2}\int_{A\pr{k,r}} \pr{u -k}^2.$$
Using these inequalities in \eqref{eq5.7} above, we have that for $h > k\ge 0$ and $\frac{1}{2}R \le r < \rho \le R$
\begin{align*}
\int_{A\pr{h,r}} \pr{u - h}^2
&\le C_1 \brac{ \frac{\abs{A\pr{h, \rho}}^{\eps_1}}{\pr{\rho -r}^2} \int_{A\pr{h, \rho}}  \pr{u -h}^2 
+  F^2 \abs{A\pr{h, \rho}}^{1+\eps_2} } \\
&\le C_1 \set{\frac{1}{\pr{\rho -r}^2 \pr{h-k}^{2\eps_1}} \brac{\int_{A\pr{k,\rho}} \pr{u -k}^2}^{1+ \eps_1} 
+ \frac{F^2 }{\pr{h-k}^{2\pr{1 + \eps_2}}} \brac{ \int_{A\pr{k,\rho}} \pr{u -k}^2}^{1+\eps_2} }
\end{align*}
or
\begin{align}
& \norm{\pr{u - h}^+}_{L^2\pr{\Om_r}}
\le C_2 \brac{ \frac{1}{\pr{\rho -r} \pr{h - k}^{\eps_1}}\norm{\pr{u - k}^+}_{L^2\pr{\Om_\rho}}^{1+\eps_1} 
+  \frac{F}{\pr{h-k}^{1+\eps_2}} \norm{\pr{u - k}^+}_{L^2\pr{\Om_\rho}}^{1+\eps_2}},
\label{eq5.8}
\end{align}
where $C_2$ depends on $C_1$.

Set $\vp\pr{k, r} = \norm{\pr{u - k}^+}_{L^2\pr{\Om_r}}$.
For $i = 0, 1, 2, \ldots$, define
\begin{align*}
k_i = K \pr{1 - \frac{1}{2^i}}, \qquad &
r_i = \frac R 2 + \frac{R}{2^{i+1}},
\end{align*}
so that
\begin{align*}
k_i - k_{i -1} = \frac{K}{2^i}, \qquad &
r_{i - 1}  - r_i = \frac{R}{2^{i+1}},
\end{align*}
where $K > 0$ is to be determined.
Then it follows from \eqref{eq5.8} with $\rho = r_{i-1}$, $r = r_i$, $h = k_i$, and $k = k_{i-1}$ that
\begin{equation}
\vp\pr{k_i, r_i}
\le C_2 \brac{ 2 \frac{2^{\pr{1 + \eps_1}i}}{R K^{\eps_1}} \vp\pr{k_{i -1}, r_{i -1}}^{1+\eps_1}
+ F \pr{\frac{2^i}{ K }}^{1+\eps_2 } \vp\pr{k_{i -1}, r_{i -1}}^{1+\eps_2} }, \quad i\geq 1.
\label{eq5.9}
\end{equation}
{\bf Claim:} There exists $\mu > 1$ and $K$ sufficiently large (depending, in particular, on $\mu$) such that for any $i = 0 , 1, \ldots$
\begin{align}
\vp\pr{k_i, r_i } \le \frac{\vp\pr{k_0, r_0}}{\mu^i}.
\label{eq5.10}
\end{align}
It is clear that the claim holds for $i = 0$.
Assume that the claim holds for $i - 1$.
Then
\begin{align*}
\vp\pr{k_{i-1}, r_{i-1} }^{1 + \eps} 
\le \brac{\frac{\vp\pr{k_0, r_0}}{\mu^{i-1}}}^{1 + \eps}
= \brac{\frac{\vp\pr{k_0, r_0}^\eps}{\mu^{i \eps - \pr{1+\eps}}}}  \frac{\vp\pr{k_0, r_0}}{\mu^{i}}.
\end{align*}
Substituting this expression into \eqref{eq5.9}, we have
\begin{align*}
\vp\pr{k_i, r_i} 
&\le C_2 \brac{2 \mu^{\pr{1+\eps_1}} \pr{\frac{2^{\pr{1 + \eps_1}}}{ \mu^{ \eps_1}}}^i \brac{\frac{\vp\pr{k_0, r_0}}{ R^{\frac n {2}}K }}^{\eps_1}  
+  \mu^{\pr{1+\eps_2}} \pr{\frac{2^{\pr{1 + \eps_2}}}{ \mu^{ \eps_2}}}^i \frac{R^{\frac n 2 \eps_2}F}{K} \brac{\frac{\vp\pr{k_0, r_0}}{ R^{\frac n {2}} K}}^{\eps_2}  }\frac{\vp\pr{k_0, r_0}}{\mu^{i}}.
\end{align*}
If we choose $\mu> 1$ so that $\mu^{\eps_i} \ge 2^{2+\eps_i}$ for each $i$, then for the claim to hold we need
\begin{align*}
C_2 \set{2 \mu^{\pr{1+\eps_1}} \brac{\frac{R^{-\frac n {2}} \vp\pr{k_0, r_0}}{ K }}^{\eps_1}  
+  \mu^{\pr{1+\eps_2}} \frac{R^{2 - \frac n \ell} F}{K} \brac{\frac{R^{-\frac n {2}} \vp\pr{k_0, r_0}}{  K}}^{\eps_2} } \le 1.
\end{align*}
Thus, we choose $K = C_0 R^{- n/2} \vp\pr{k_0, r_0} + R^{2 - \frac n \ell}F $ for some $C_0 >> 1$ that depends on $C_2$, $\mu$ and each $\eps_i$.

Taking $i \to \iny$ in \eqref{eq5.10} shows that $\vp\pr{ K, \frac{R}{2}} = 0$.
In other words, since $\vp\pr{k_0, r_0} = \vp\pr{0, R} = \norm{u^+}_{L^{2}\pr{\Om_R}}$,
\begin{align*}
\sup_{\Om_{R/2}}u^+ 
\le K
\le C_0 R^{-\frac n 2}\norm{u^+}_{L^{2}\pr{\Om_R}} + R^{2 - \frac n \ell}\norm{f}_{L^\ell\pr{\Om_R}}.
\end{align*}
For any $q \in \brac{ 2, \iny}$, an application of the H\"older inequality gives
\begin{equation}\label{eq5.11}
\sup_{\Om_{R/2}}u^+ 
\le C_0 R^{-\frac n {q}} \norm{u^+}_{L^{q}\pr{\Om_R}} +  R^{2 - \frac n \ell}\norm{f}_{L^\ell\pr{\Om_R}}.
\end{equation}

To obtain an estimate in $\Om_{\te R}$, we apply \eqref{eq5.11} to $\Om_{\pr{1-\te}R}\pr{y}$, where $y \in \Om_{\te R}$.
That is, for any $y \in \Om_{\te R}$,
\begin{align*}
 u^+\pr{y} 
&\le \sup_{\Om_{\frac{\pr{1-\te}R}{2}}\pr{y}} u^+ 
\le C_0 \brac{\pr{1 - \te}R}^{-\frac n {q}} \norm{u^+}_{L^{q}\pr{\Om_R}} + R^{2 - \frac n \ell}\norm{f}_{L^\ell\pr{\Om_R}} .
\end{align*}
Now for $\te \in \pr{0,1}$, $R > 0$, and $q \in \pr{0,2}$, we have
\begin{align*}
\norm{u^+}_{L^\iny\pr{\Om_{\te R}}} 
&\le C_0 \brac{\pr{1 - \te}R}^{-\frac n 2} \norm{u^+}_{L^{2}\pr{\Om_R}} 
+  R^{2 - \frac n \ell}\norm{f}_{L^\ell\pr{\Om_R}}  \\
&\le C_0 \brac{\pr{1 - \te}R}^{-\frac n 2}  \norm{u^+}_{L^\iny\pr{\Om_{ R}}}^{1 - \frac {q} 2} \brac{ \int_{\Om_R} \pr{u^+}^{q} }^{\frac 1 2} 
+  R^{2 - \frac n \ell}\norm{f}_{L^\ell\pr{\Om_R}}  \\
&\le \frac{1}{2}\norm{u^+}_{L^\iny\pr{\Om_{R}}} 
+ C_{0,q} \brac{\pr{1 - \te}R}^{-\frac n {q}} \brac{ \int_{\Om_R} \pr{u^+}^{q} }^{\frac 1 {q}} 
+  R^{2 - \frac n \ell}\norm{f}_{L^\ell\pr{\Om_R}} ,
\end{align*}
where $C_{0,q}$ depends on $q$ and $C_0$.
Assuming that $\norm{u^+}_{L^\iny\pr{\Om_{R}}}<\infty$ (recall the remark before the proof), set $h\pr{t} = \norm{u^+}_{L^\iny\pr{\Om_{t}}}$ for $t \in (0,R]$.
Then, for $\te \in \pr{0,1}$, $R > 0$, and $q \in \pr{0,2}$, we have
\begin{align*}
h\pr{r}
&\le \frac{1}{2}h\pr{R} 
+ \frac{C_{0,q}}{\pr{R- r}^{\frac n {q}}} \norm{u^+}_{L^{q}\pr{\Om_R}} 
+  R^{2 - \frac n \ell}\norm{f}_{L^\ell\pr{\Om_R}} .
\end{align*}
It follows from Lemma 4.3 in \cite{HL11} that for any $r < R$,
\begin{align*}
 \norm{u^+}_{L^\iny\pr{\Om_{r}}}
&\le c_{q} \brac{\frac{C_{0,q}}{\pr{R- r}^{\frac n {q}}} \norm{u^+}_{L^{q}\pr{\Om_R}} 
+  R^{2 - \frac n \ell}\norm{f}_{L^\ell\pr{\Om_R}}  }.
\end{align*}
\end{proof}

\begin{rem}
If $u$ is a supersolution, then the conclusions of the previous lemmas apply to $u^-$ in place of $u^+$.
\end{rem}

Now we prove a slightly different version of Moser boundedness.
We show that without the assumptions of coercivity and non-degeneracy, solutions are still locally bounded, but there is a dependence on the size of the domain and on the negative part of the zeroth order potential.

\begin{lem}
Let $\Om \su \R^n$ be open and connected and take $N = 1$.
Assume that $\bF\pr{\Om}$, $\bF_0\pr{\Om}$, $\mathcal{L}$, and $\mathcal{B}$ satisfy \rm{\ref{A1} -- \ref{A5}} and \rm{\ref{B1}}.
Suppose $V = V_+ - V_-$ where $V_\pm \ge 0$ a.e. and $V_- \in L^p\pr{\Om_R}$ for some $p \in \pb{ \frac n 2, \iny}$.
Assume that $b \in L^s\pr{\Om_R}^n$, $d \in L^t\pr{\Om_R}^n$ for some $s, t \in \pb{ n, \iny}$.
Let $u \in \bF\pr{\Om_{2R}}$ satisfy $u = 0$ along $\del \Om \cap B_{2R}$.
Let $f \in L^{\ell}\pr{\Om_{R}}$ for some $\ell \in \pb{ \frac n 2, \iny}$ and assume that $\mathcal{L} u \le f$ in $\Om_{R}$ weakly in the sense that for any $\vp \in \bF_{0}\pr{B_{R}}$ such that $\vp \ge 0$ in $\Om_{R}$, we have \eqref{eq5.2}.
Then $u^+ \in L_{loc}^\iny\pr{\Om_R}$ and for any $r < R$, $q > 0$, \eqref{eq5.3} holds with $C = C\pr{n, q, p, s, t, \ell, \la, \La, R^{2 - \frac n p}\norm{V_-}_{L^p\pr{\Om_R}}, R^{1 - \frac n s}\norm{b}_{L^s\pr{\Om_R}}, R^{1 - \frac n t}\norm{d}_{L^t\pr{\Om_R}}}$, where $c_{q}$ depends only on $q$.
\label{l5.6}
\end{lem}

\begin{proof}
We will first prove the case of $q = 2$, $R = 1$, and $r=\frac 12$.
Fix $\zeta\in C_c^\infty(B_1)$, a cutoff function for which $0 \le \zeta \le 1$.
For some $k \ge 0$, define $v = \pr{u - k}_+$.
By \rm{\ref{B1}}, $v \zeta, v \zeta^2 \in \bF_0\pr{\Om_1}$, and $Dv=Du$ for $u>k$, $Dv=0$ for $u\leq k$, by Lemma~7.6 from \cite{GT01} (since \eqref{eq2.8} implies that $v$ is weakly differentiable on $\Omega$).

Since $\supp (v\zeta^2)$ is a subset of $\{u\geq k\}$, then a computation gives
\begin{align*}
\int A^{\al \be} D_\be v D_\al v \zeta^2 
&= \mathcal{B}\brac{u, v \zeta^2} 
- 2 \int A^{\al \be} D_\be v D_\al \zeta v \zeta \\
&- \int \brac{b^\al v D_\al \pr{v \zeta^2} +  \pr{d^\be D_\be v + V v } v \zeta^2}
- k \int \brac{b^\al D_\al \pr{v \zeta^2} +  V  v \zeta^2} \\
&\le \int \pr{f + k V_-} v \zeta^2
- k \int b^\al D_\al \pr{v \zeta} \zeta 
- k \int b^\al D_\al \zeta v \zeta
- 2 \int A^{\al \be} D_\be v D_\al \zeta v \zeta  \\
&- \int \pr{b^\al - d^\al} D_\al \zeta v^2 \zeta
- \int \pr{b^\al + d^\al} D_\al \pr{v \zeta} v \zeta
+ \int V_- v^2 \zeta^2,
\end{align*}
where we used \eqref{eq5.2} with $\vp := v \zeta^2 \in \bF_{0 }\pr{\Om_1}$, $\vp \ge 0$ to get the first term in the last inequality.
An application of the H\"older, Sobolev, and Young inequalities shows that
\begin{align*}
\int b^\al D_\al \zeta v^2 \zeta 
&\le \frac{\la}{16} \int \abs{D\pr{v\zeta}}^2
+ \frac{64 c_n^2 }{\la} \norm{b}_{L^s\pr{\Om_1}}^2 \int v^2 \abs{D\zeta}^2 \abs{\set{v \zeta \ne 0}}^{\frac 2 n - \frac 2 s}.
\end{align*}
Similarly,
\begin{align*}
\int b^\al D_\al\pr{v \zeta} v \zeta 
&\le c_n \norm{b}_{L^s\pr{\Om_1}} \int \abs{D\pr{v \zeta} }^{2} \abs{\set{v \zeta \ne 0}}^{\frac 1 n - \frac 1 s} \\
 \int V_- v^2 \zeta^2
&\le c_n^2  \norm{V_-}_{L^p\pr{\Om_1}}\int \abs{D\pr{v \zeta} }^{2} \abs{\set{v \zeta \ne 0}}^{\frac 2 n - \frac 1 p}.
\end{align*}

The ellipticity condition, \eqref{eq2.12}, in combination with boundedness \eqref{eq2.13} and the computations above, shows that
\begin{align*}
& \int \abs{D\pr{v\zeta}}^2 
\le \frac 8 \la\int \pr{f + k V_-} v \zeta^2
- \frac {8k} \la\int b^\al D_\al \pr{v \zeta} \zeta 
- \frac {8k} \la\int b^\al D_\al \zeta v \zeta \\
&+ \frac {16}{ \la^2} \pr{\La^2 + \frac {\la^2} 4
+ 32 c_n^2 \norm{b}_{L^s\pr{\Om_1}}^2 \abs{\set{v \zeta \ne 0}}^{\frac 2 n - \frac 2 s}
+ 32 c_n^2 \norm{d}_{L^t\pr{\Om_1}}^2 \abs{\set{v \zeta \ne 0}}^{\frac 2 n - \frac 2 t}} \int v^2 \abs{D\zeta}^2 \\
&+\frac {8c_n} \la\pr{  \norm{b}_{L^s\pr{\Om_1}} \abs{\set{v \zeta \ne 0}}^{\frac 1 n - \frac 1 s}
+ \norm{d}_{L^t\pr{\Om_1}} \abs{\set{v \zeta \ne 0}}^{\frac 1 n - \frac 1 t}
+ c_n \norm{V_-}_{L^p\pr{\Om_1}} \abs{\set{v \zeta \ne 0}}^{\frac 2 n - \frac 1 p}}\int \abs{D\pr{v \zeta} }^{2} .
\end{align*}

As in \eqref{eq5.4},
\begin{align*}
\int f v \zeta^2
&\le \frac{\la}{32} \int \abs{D\pr{v\zeta}}^2
+  \frac{8 c_n^2 }{\la} \norm{f}_{L^\ell\pr{\Om_1}}^2 \abs{\set{v \zeta \ne 0}}^{1 + \frac 2 n - \frac{2}{\ell}} \\
\int V_- v \zeta^2 
&\le \frac{\la}{32 k} \int \abs{D\pr{v\zeta}}^2
+  \frac{8 k c_n^2 }{\la} \norm{V_-}_{L^p\pr{\Om_1}}^2 \abs{\set{v \zeta \ne 0}}^{1 + \frac 2 n - \frac{2}{p}}.
\end{align*}
Similarly,
\begin{align*}
\int b^\al D_\al \pr{v \zeta} \zeta 
&\le \frac{\la}{32 k} \int \abs{D\pr{v\zeta}}^2
+  \frac{8 k}{\la} \norm{b}_{L^s\pr{\Om_1}}^2 \abs{\set{v \zeta \ne 0}}^{1 - \frac{2}{s}} \\
\int b^\al D_\al \zeta v \zeta
&\le \frac{\la}{32 k} \int v^2 \abs{D\zeta}^2
+  \frac{8 k}{\la} \norm{b}_{L^s\pr{\Om_1}}^2 \abs{\set{v \zeta \ne 0}}^{1 - \frac{2}{s}}.
\end{align*}
It follows that
\begin{align*}
& \int \abs{D\pr{v\zeta}}^2 
\le \pr{ \frac{16 c_n }{\la} \norm{f}_{L^\ell\pr{\Om_1}}}^2 \abs{\set{v \zeta \ne 0}}^{1 + \frac 2 n - \frac{2}{\ell}} \\
&+  k^2 \brac{\pr{\frac{16 c_n }{\la} \norm{V_-}_{L^p\pr{\Om_1}}}^2 \abs{\set{v \zeta \ne 0}}^{1 + \frac 2 n - \frac{2}{p}} 
+ 2\pr{\frac{16}{\la} \norm{b}_{L^s\pr{\Om_1}}}^2 \abs{\set{v \zeta \ne 0}}^{1 - \frac{2}{s}}} \\
&+ \frac {32}{ \la^2} \pr{\La^2 + \frac {17\la^2}{64}
+ 32 c_n^2 \norm{b}_{L^s\pr{\Om_1}}^2 \abs{\set{v \zeta \ne 0}}^{\frac 2 n - \frac 2 s}
+ 32 c_n^2 \norm{d}_{L^t\pr{\Om_1}}^2 \abs{\set{v \zeta \ne 0}}^{\frac 2 n - \frac 2 t}} \int v^2 \abs{D\zeta}^2 \\
&+\frac {16 c_n} \la\pr{  \norm{b}_{L^s\pr{\Om_1}} \abs{\set{v \zeta \ne 0}}^{\frac 1 n - \frac 1 s}
+ \norm{d}_{L^t\pr{\Om_1}} \abs{\set{v \zeta \ne 0}}^{\frac 1 n - \frac 1 t}
+ c_n \norm{V_-}_{L^p\pr{\Om_1}} \abs{\set{v \zeta \ne 0}}^{\frac 2 n - \frac 1 p}}\int \abs{D\pr{v \zeta} }^{2} .
\end{align*}
If $\abs{\set{v \zeta \ne 0}}$ is chosen so that
\begin{align}\label{eq5.12}
& \abs{\set{v \zeta \ne 0}} \le \min\set{ \pr{\frac \la {64 c_n  \norm{b}_{L^s\pr{\Om_1}} }}^{\frac{ns}{s-n}} , \pr{\frac \la {64 c_n \norm{d}_{L^t\pr{\Om_1}} }}^{\frac{nt}{t-n}}, \pr{\frac \la {64 c_n^2 \norm{V_-}_{L^p\pr{\Om_1}} }}^{\frac{np}{2p-n}}}
\end{align}
then
\begin{align*}
\int \abs{D\pr{v\zeta}}^2 
&\le \pr{\frac {64 \La^2}{ \la^2} +18} \int v^2 \abs{D\zeta}^2 
+ \pr{ \frac{32 c_n }{\la} \norm{f}_{L^\ell\pr{\Om_1}}}^2 \abs{\set{v \zeta \ne 0}}^{1 + \frac 2 n - \frac{2}{\ell}} \\
&+  \pr{\frac{32 k c_n }{\la} \norm{V_-}_{L^p\pr{\Om_1}}}^2 \abs{\set{v \zeta \ne 0}}^{1 + \frac 2 n - \frac{2}{p}} 
+ 2\pr{\frac{32 k}{\la} \norm{b}_{L^s\pr{\Om_1}}}^2 \abs{\set{v \zeta \ne 0}}^{1 - \frac{2}{s}}.
\end{align*}
Since the H\"older and Sobolev inequalities imply that
\begin{align*}
\int \pr{v \zeta}^2 
\le \pr{\int \pr{v \zeta}^{2^*}}^{2/2^*} \abs{\set{ v \zeta \ne 0}}^{1 - \frac{2}{2^*}}
\le c_n^2 \int \abs{D\pr{v \zeta}}^{2} \abs{\set{ v \zeta \ne 0}}^{\frac{2}{n}},
\end{align*}
then
\begin{align}
\int \pr{v \zeta}^2 
&\le \frac{C_1}{4} \int v^2 \abs{D\zeta}^{2} \abs{\set{ v \zeta \ne 0}}^{\eps}  
+ C_1 \pr{F + k}^2 \abs{\set{v \zeta \ne 0}}^{1 + \eps}.
\label{eq5.13}
\end{align}
where $\eps = \min\set{\frac 2 n, \frac 4 n - \frac 2 \ell, \frac 4 n - \frac 2 p, \frac 2 n - \frac 2 s} > 0$, $F = \norm{f}_{L^\ell\pr{\Om_1}}$, and 
\begin{align*}
C_1 &=  \pr{\frac {16 \La}{ \la}}^2 + 72 + \pr{ \frac{32 c_n }{\la} }^2 +  \pr{\frac{32 c_n }{\la} \norm{V_-}_{L^p\pr{\Om_1}}}^2 + 2\pr{\frac{32}{\la} \norm{b}_{L^s\pr{\Om_1}}}^2.
\end{align*}

For fixed $0 < r \le \rho \le 1$, let $\zeta \in C^\iny_c\pr{B_\rho}$ be such that $\zeta \equiv 1$ in $B_r$ and $\abs{D\zeta} \le \frac{2}{\rho - r}$ in $B_1$.
We let $A\pr{k, r} = \set{ x \in \Om_r : u \ge k} = \supp v \cap \Om_r$.
Then, for any $0 < r < \rho \le 1$ and $k \ge 0$, if \eqref{eq5.12} holds, then \eqref{eq5.13} implies that 
\begin{align}
\int_{A\pr{k,r}} \pr{u - k}^2
&\le C_1 \brac{\frac{\abs{A\pr{k, \rho}}^{\eps}}{\pr{\rho -r}^2} \int_{A\pr{k, \rho}}  \pr{u -k}^2 
+ \pr{F + k}^2 \abs{A\pr{k, \rho}}^{1+\eps} }.
\label{eq5.14}
\end{align}

Since the H\"older inequality implies that
\begin{align*}
\abs{A\pr{k,r}} \le \frac 1 k \int_{A\pr{k,r}} u^+ \le \frac 1 k \pr{\int_{\Om_R} \abs{u^+}^2}^{\frac 1 2} \abs{A\pr{k, r}}^{\frac 1 2},
\end{align*}
then $\disp \abs{\set{v \zeta \ne 0}} \le \frac{1}{k^2} \norm{u^+}_{L^2\pr{\Om_R}}^2$.
To ensure that \eqref{eq5.12} holds, we take
\begin{align}\label{eq5.15}
k &\ge k_0 := \mathcal{C} \norm{u^+}_{L^2\pr{\Om_R}}, 
\end{align}
where
\begin{align*}
\mathcal{C} 
&:=\pr{\frac {64 c_n } \la  \norm{b}_{L^s\pr{\Om_1}}}^{\frac{ns}{2s-2n}} + \pr{\frac {64 c_n }\la \norm{d}_{L^t\pr{\Om_1}}}^{\frac{nt}{2t-2n}} + \pr{\frac {64 c_n^2 } \la \norm{V_-}_{L^p\pr{\Om_1}} }^{\frac{np}{4p-2n}}.
\end{align*}

The goal is to show that there exists a $k \ge k_0$ such that 
$$\int_{A\pr{k, 1/2}} \pr{u- k}^2 = 0.$$
With $h > k\ge k_0$ and $0 < r < 1$, it follows from the arguments in the previous proof that
\begin{align}
& \norm{\pr{u - h}^+}_{L^2\pr{\Om_r}}
\le C_2 \brac{ \frac{1}{\pr{\rho -r} \pr{h - k}^{\eps}}
+  \frac{F+h}{\pr{h-k}^{1+\eps}} }\norm{\pr{u - k}^+}_{L^2\pr{\Om_\rho}}^{1+\eps},
\label{eq5.16}
\end{align}
where $C_2$ depends on $C_1$.

Set $\vp\pr{k, r} = \norm{\pr{u - k}^+}_{L^2\pr{\Om_r}}$.
For $i = 0, 1, 2, \ldots$, define
\begin{align*}
k_i = k_0 + K \pr{1 - \frac{1}{2^i}}, \qquad &
r_i = \frac 1 2 + \frac{1}{2^{i+1}},
\end{align*}
where $K > 0$ is to be determined.
Then it follows from \eqref{eq5.16} with $\rho = r_{i-1}$, $r = r_i$, $h = k_i$, and $k = k_{i-1}$ that for $ i\geq 1$
\begin{equation}
\vp\pr{k_i, r_i}
\le C_2 \brac{ 3 \frac{2^{\pr{1 + \eps}i}}{K^{\eps}} 
+ \pr{F + k_0} \pr{\frac{2^i}{ K }}^{1+\eps } }\vp\pr{k_{i -1}, r_{i -1}}^{1+\eps} .
\label{eq5.17}
\end{equation}
{\bf Claim:} There exists $\mu > 1$ and $K$ sufficiently large (depending, in particular, on $\mu$) such that for any $i = 0 , 1, \ldots$ \eqref{eq5.10} holds.\\
Clearly, the claim holds for $i = 0$.
If the claim holds for $i - 1$, then
\begin{align*}
\vp\pr{k_i, r_i}
\le C_2 \mu^{1+\eps} \brac{3 + \frac{F + k_0}{K} }\pr{\frac{2^{1+\eps }}{\mu^\eps}}^i \pr{\frac{\vp\pr{k_0, r_0}}{K}}^\eps \frac{\vp\pr{k_0, r_0}}{\mu^{i}}.
\end{align*}
If we choose $\mu> 1$ so that $\mu^{\eps} \ge 2^{1+\eps}$, then for the claim to hold we need
\begin{align*}
C_2 \mu^{1+\eps} \brac{3+ \frac{F + k_0}{K} } \pr{\frac{\vp\pr{k_0, r_0}}{K}}^\eps \le 1.
\end{align*}
Setting $K = C_0 \vp\pr{k_0, r_0} + F + k_0$ for some $C_0 >> 1$ that depends on $C_2$, $\mu$ and $\eps$, gives the claim.

Taking $i \to \iny$ in \eqref{eq5.10} shows that $\vp\pr{k_0 + K, \frac{1}{2}} = 0$.
Since $\vp\pr{k_0, r_0} = \vp\pr{k_0, 1} \le \norm{u^+}_{L^{2}\pr{\Om_1}}$, then
\begin{align*}
\sup_{\Om_{1/2}}u^+ 
\le K + k_0
\le C_0 \norm{u^+}_{L^{2}\pr{\Om_R}} + F + 2 k_0 
= C_3 \norm{u^+}_{L^{2}\pr{\Om_1}} + \norm{f}_{L^\ell\pr{\Om_1}},
\end{align*}
where $C_3 = C_0 + 2 \mathcal{C}$.

The estimate for $R \ne 1$ follows from a standard scaling argument.
Assume that $\mathcal{L} u = f$ weakly on $\Om_R$.
Let $u_R\pr{x} = u\pr{Rx}$, $V_R\pr{x} = R^2 V\pr{Rx}$, $b_R\pr{x} = R b\pr{Rx}$, $d_R\pr{x} = R d\pr{Rx}$, $f_R\pr{x} = R^2 f\pr{Rx}$, and define $\mathcal{L}_R$ to be the scaled version of $\mathcal{L}$.
Then $\mathcal{L}_R u_R = f_R$ on $B_1$.
Since $\mathcal{L}_R$ has the same ellipticity constant as $\mathcal{L}$, then by the previous estimate,
\begin{align*}
\sup_{\Om_{R/2}}u^+ 
= \sup_{\Om_{1/2}}u_R^+ 
\le C_{3,R} \norm{u_R^+}_{L^{2}\pr{\Om_1}} + \norm{f_R}_{L^\ell\pr{\Om_1}} 
\le C_{3,R} R^{-n/2} \norm{u^+}_{L^{2}\pr{\Om_r}} + R^{2 - \frac n \ell} \norm{f}_{L^\ell\pr{\Om_R}},
\end{align*}
where 
\begin{align*}
C_{3,R} &= c \brac{ \pr{\frac {16 \La}{ \la}}^2  + 72 + \pr{ \frac{32 c_n }{\la} }^2 +  \pr{\frac{32 c_n }{\la} R^{2 - \frac n p} \norm{V_{-}}_{L^p\pr{\Om_R}}}^2 + 2\pr{\frac{32}{\la} R^{1 - \frac n s} \norm{b}_{L^s\pr{\Om_R}}}^2} \\
&+ 2 \brac{ \pr{\frac {64 c_n^2 } \la  R^{2 - \frac n p} \norm{V_{-}}_{L^p\pr{\Om_R}}}^{\frac{np}{4p-2n}} + \pr{\frac {64 c_n } \la R^{1 - \frac n s} \norm{b}_{L^s\pr{\Om_R}}}^{\frac{ns}{2s-2n}} + \pr{\frac {64 c_n }\la R^{1 - \frac n t} \norm{d}_{L^t\pr{\Om_R}}}^{\frac{nt}{2t-2n}}  }
\end{align*}
grows with $R$.

The rest of the proof, which includes $q \ne 2$ and $r = \te R$ for $\te \ne \frac 1 2$, follows that of the previous lemma.
\end{proof}

\section{Interior H\"older continuity in the equation setting}
\label{s6}

Within this section, we prove H\"older continuity of solutions to general second-order elliptic equations with lower order terms.
Towards proving H\"older continuity of solutions, we first show that a lower bound holds for all non-negative supersolutions to our PDE.
The combination of this lower bound with the upper bounds in Section~\ref{s5} and the arguments presented in Corollary 4.18 from \cite{HL11} leads to the proof of H\"older continuity.

To prove the lower bound, we use some of the ideas presented in \cite{HL11}, but since lower order terms were not considered there, we have added the details. 
Again, the general approach that we follow is based on the ideas of de Giorgi.  
Similar estimates are presented in \cite{GT01} using Moser's approach.
We actually avoid the use of Moser's iteration, and as a consequence, we prove a lower bound for $u$ in terms of $\norm{u}_{q_0}$ for only one $q_0$ instead of a full range of values as was done in \cite{HL11} and \cite{GT01}.
For us, the lower bound is a step towards H\"older continuity, so a single $q_0$ is sufficient.

Since our proofs are different from those in \cite{HL11} and \cite{GT01}, we have included the details here.
We also present the structure of the associated constants.

To start, we prove the following result that uses the John-Nirenberg lemma.

\begin{lem}
Take $N = 1$.
Assume that $\bF\pr{B_R}$, $\bF_0\pr{B_R}$, $\mathcal{L}$, and $\mathcal{B}$ satisfy \rm{\ref{A1} -- \ref{A5}} and \rm{\ref{B2}}.
Suppose $V = V_+ - V_-$ where $V_\pm \ge 0$ a.e. and $V_+ \in L^p\pr{B_R}$ for some $p \in \pb{ \frac n 2, \iny}$.
Assume that there exists $s, t \in \pb{n, \iny}$ so that $b \in L^s\pr{B_R}^n$, $d \in L^t\pr{B_R}^n$.
Assume that $f \in L^{\ell}\pr{B_R}$ for some $\ell \in \pb{ \frac n 2, \iny}$, $g^\al \in L^m\pr{B_R}$ for some $m \in\pb{ n, \iny}$.
Let $u \in \bF\pr{B_{R}}$ be a non-negative supersolution in the sense that for any $\vp \in \bF_0\pr{B_{R}}$ such that $\vp \ge 0$ in $B_{R}$, we have
\begin{align}
\mathcal{B}\brac{u, \vp} \ge - \int f \vp  + \int g^\al D_\al \vp .
\label{eq6.1}
\end{align}
Then there exists $q_0\pr{n, p, s, t, \la, \La, R^{2 - \frac n p} \norm{V_+}_{L^p\pr{B_R}}, R^{1 - \frac n s} \norm{b}_{L^s\pr{B_R}}, R^{1 - \frac n t} \norm{d}_{L^t\pr{B_R}}} > 0$ so that for any $k \ge \abs{B_R}^{\frac {2} {n} - \frac 1 \ell}\norm {f}_{L^{\ell}\pr{B_R}} +\abs{B_R}^{\frac {1} {n} - \frac 1 m} \norm {g}_{L^{m}\pr{B_R}}$,  and any $B_r\pr{y} \su B_{3R/4}$,
\begin{align}
\int_{B_r\pr{y}} \pr{u + k}^{- q_0} \int_{B_r\pr{y}} \pr{u + k}^{q_0}
&\le C_n r^{2n}.
\label{eq6.2}
\end{align}
\label{l6.1}
\end{lem}

\begin{rem}
This lemma is analogous to the first step of the proof of Theorem 4.15 from \cite{HL11}, except that here we have lower order terms.
\end{rem}

\begin{proof} 
Let $\zeta \in C^\iny_c\pr{B_R}$ be a cutoff function, $0 \le \zeta \le 1$.
By \rm{\ref{B2}} with $\om = 1$, for any $k > 0$, $\bar u := \pr{u + k}^{-1} \in \bF\pr{B_R}$.
It follows from \rm{\ref{A4}} that $\vp := \bar u \zeta^2 \in \bF_0\pr{B_R}$. 
Since $u$ is a supersolution, we have
\begin{align*}
0
&\le \int \pr{A^{\al \be} D_\be u + b^\al u} D_\al \vp + d^\be D_\be u \vp  + V u \vp  + \int f \bar u^{-1} \zeta^2 - \int g^\al D_\al \vp \\
&=  - \int A^{\al \be} \, D_\be w \, D_\al w \, \zeta^2
+ 2 \int A^{\al \be} \, D_\be w \, D_\al \zeta \, \zeta
- \int \pr{1 - \frac{k}{\bar u}} b^\al D_\al w \, \zeta^2 
+ 2 \int \pr{1 - \frac{k}{\bar u}} b^\al  D_\al \zeta \, \zeta \\
&+ \int d^\be D_\be w \, \zeta^2  
+ \int V \pr{1 - \frac{k}{\bar u}} \zeta^2 
+ \int \frac{f}{ \bar u} \zeta^2 
+ \int \frac{g^\al}{\bar u} D_\al w \zeta^2
- 2\int \frac{g^\al}{\bar u} \zeta D_\al \zeta,
\end{align*}
where we have set $w = \log \bar u$.
With $\disp \tilde f := \frac{ f}{\bar u}$, $\disp \tilde g := \frac{\abs{g}}{\bar u}$, we rearrange and bound to get
\begin{align*}
\la \int \abs{D w}^2 \zeta^2 
&\le \int A^{\al \be} \, D_\be w \, D_\al w \, \zeta^2 \\
&\le 2 \La \int \abs{ D w} \, \abs{D \zeta} \, \zeta
+ \int \pr{\abs{b} + \abs{d} + \abs{\tilde g}} \abs{ D w} \, \zeta^2
+ 2 \int \pr{\abs{b} + \abs{\tilde g}} \abs{ D \zeta} \, \zeta
+ \int \pr{\abs{V_+} + \tilde f} \zeta^2\\
&\le \frac{\la}{2} \int \abs{Dw}^2 \zeta^2 
+ C_1 \int \abs{D\zeta}^2,
\end{align*}
where
\begin{align*}
C_1 
&= \frac{8 \La^2}{\la} 
+ \frac{2 c_n^2}{\la} \pr{ 
\abs{B_R}^{\frac {1} {n} - \frac 1 s}\norm{ b}_{L^s\pr{B_R}} 
+ \abs{B_R}^{\frac {1} {n} - \frac 1 t}\norm{ d}_{L^t\pr{B_R}} 
+ \abs{B_R}^{\frac {1} {n} - \frac 1 m}\norm{\tilde g}_{L^m\pr{B_R}} }^2 \\
&+ 2 c_n \pr{\abs{B_R}^{\frac {1} {n} - \frac 1 s}\norm{ b}_{L^s\pr{B_R}} 
+ \abs{B_R}^{\frac {1} {n} - \frac 1 m} \norm{\tilde g}_{L^m\pr{B_R}}}
+ c_n^2 \pr{\abs{B_R}^{\frac 2 n - \frac 1 p} \norm{V_+}_{L^p\pr{B_R}}
+ \abs{B_R}^{\frac 2 n - \frac 1 \ell} \norm{ \tilde f}_{L^\ell\pr{B_R}}} .
\end{align*}
If $ \abs{B_R}^{\frac 2 n - \frac 1 \ell}  \norm {f}_{L^{\ell}\pr{B_R}} + \abs{B_R}^{\frac {1} {n} - \frac 1 m}\norm {g}_{L^{m}\pr{B_R}} > 0$, then we choose $k = \abs{B_R}^{\frac {2} {n} - \frac 1 \ell}\norm {f}_{L^\ell\pr{B_R}} +\abs{B_R}^{\frac {1} {n} - \frac 1 m} \norm {g}_{L^{m}\pr{B_R}}$.
Otherwise, we choose $k > 0$ to be arbitrary and eventually take $k \to 0^+$.
Then
\begin{align}
\int \abs{D w}^2 \zeta^2 
&\le C_2 \int \abs{D \zeta}^2,
\label{eq6.3}
\end{align}
where 
\begin{align}
C_2 
&= \pr{\frac{4 \La}{\la} }^2
+ \pr{\frac{2 c_n}{\la}}^2 \pr{ \abs{B_R}^{\frac {1} {n} - \frac 1 s}\norm{ b}_{L^s\pr{B_R}} + \abs{B_R}^{\frac {1} {n} - \frac 1 t}\norm{ d}_{L^t\pr{B_R}} + 1 }^2 \nonumber \\
&+ \frac{4 c_n}{\la} \pr{\abs{B_R}^{\frac {1} {n} - \frac 1 s} \norm{ b}_{L^s\pr{B_R}} + 1}
+ \frac{2c_n^2}{\la} \pr{\abs{B_R}^{\frac 2 n - \frac 1 p} \norm{V_+}_{L^p\pr{B_R}} + 1} .
\label{eq6.4}
\end{align}
Let $B_r\pr{y} \su B_{3R/4}$.
Choose $\zeta$ so that $\zeta \equiv 1$ in $B_r\pr{y}$, $\supp \zeta \Subset B_R$, and $\abs{D \zeta} \le \frac{C}{r}$.
It follows from the H\"older inequality, Poincar\'e inequality, then \eqref{eq6.3}, that for any $B_r\pr{y} \su B_{3R/4}$,
\begin{align*}
\int_{B_r\pr{y}} \abs{w - w_{y,r}}
&\le \abs{B_r}^{\frac 1 2} \pr{\int_{B_r\pr{y}} \abs{w - w_{y,r}}^2}^{\frac 1 2} 
\le c_n r^{\frac{n+2}{2}} \pr{ \int_{B_r} \abs{Dw}^2}^{\frac 1 2} \\
&\le c_n r^{\frac {n+2} 2} \pr{ C_2 \int \abs{D \zeta}^2 }^{\frac 1 2}
\le C_3 r^{n},
\end{align*}
where $\disp w_{y,r} = \fint_{B_r\pr{y}} w$ and $C_3 = c_n \sqrt{C_2}$.
Therefore, $w$ is a BMO function.
By the John-Nirenberg lemma, there exists $q_1, C_4 > 0$, depending only on $n$, so that for any $B_r\pr{y} \su B_{3R/4}$
\begin{align*}
\int_{B_r\pr{y}} e^{\frac{q_1}{C_3} \abs{w - w_{y,r}}} \le C_4 r^n.
\end{align*}
Therefore, with $q_0= \frac{q_1}{C_3} = \frac{q_1}{c_n \sqrt{C_2}}$,
\begin{align*}
\int_{B_r\pr{y}} \bar u^{- q_0} \int_{B_r\pr{y}} \bar u^{q_0}
&= \int_{B_r\pr{y}} e^{- q_0 \, {\log \bar u}} \int_{B_r\pr{y}} e^{ q_0 \, {\log \bar u }} 
= \int_{B_r\pr{y}} e^{- q_0 \pr{w - w_{y,r}}} \int_{B_r\pr{y}} e^{ q_0 \pr{w - w_{y,r}}} \\
&= \int_{B_r\pr{y}} e^{q_0 \abs{w - w_{y,r}}} \int_{B_r\pr{y}} e^{- q_0 \abs{w - w_{y,r}}} 
\le C_4 r^{2n}. 
\end{align*}
\end{proof}

\begin{rem}\label{r6.3}
We sometimes use the notation $q_0\pr{R}$ to refer to the exponent $q_0$ associated to the ball of radius $R$.
\end{rem}

With the previous estimate, we can prove a lower bound for solutions.

\begin{lem}
Take $N = 1$.
Assume that $\bF\pr{B_R}$, $\bF_0\pr{B_R}$, $\mathcal{L}$, and $\mathcal{B}$ satisfy \rm{\ref{A1} -- \ref{A5}} and \rm{\ref{B1} -- \ref{B2}}.
Assume that there exists $p \in \pb{\frac n 2, \iny}$, $s, t \in \pb{n, \iny}$ so that $V_+ \in L^p\pr{B_R}$, $b \in L^s\pr{B_R}$, $d \in L^t\pr{B_R}$.
Assume that $f \in L^{\ell}\pr{B_R}$ for some $\ell \in \pb{\frac n 2, \iny}$, $g^\al \in L^m\pr{B_R}$ for some $m \in \pb{n, \iny}$.
Suppose $u \in \bF\pr{B_R}$ is a nonnegative supersolution in the sense that for any $\vp \in \bF_0\pr{B_R}$ such that $\vp \ge 0$ in $B_R$, \eqref{eq6.1} holds.
Then for $q_0 = q_0\pr{R}$ (see Remark~\ref{r6.3}), we have
\begin{align*}
\pr{\fint_{B_{3R/4}} u^{q_0} }^{\frac{1}{q_0}}
\le C_0 \pr{\inf_{B_{R/2}} u + \abs{B_R}^{\frac {2} {n} - \frac 1 \ell}\norm {f}_{L^{\ell}\pr{B_R}} +\abs{B_R}^{\frac {1} {n} - \frac 1 m} \norm {g}_{L^{m}\pr{B_R}}}, 
\end{align*}
where 
$C_0 = C_0\pr{n, q_0, p, s, t, \ell, m, \la, \La, R^{2-\frac n p} \norm{V_+}_{L^p\pr{B_R}}, R^{1-\frac n s} \norm{b}_{L^s\pr{B_R}}, R^{1-\frac n t} \norm{d}_{L^s\pr{B_R}}}$.
\label{l6.4}
\end{lem}

\begin{proof}
If $\abs{B_R}^{\frac {2} {n} - \frac 1 \ell}\norm {f}_{L^{\ell}\pr{B_R}} +\abs{B_R}^{\frac {1} {n} - \frac 1 m} \norm {g}_{L^{m}\pr{B_R}} > 0$, let $k = \abs{B_R}^{\frac {2} {n} - \frac 1 \ell}\norm {f}_{L^{\ell}\pr{B_R}} +\abs{B_R}^{\frac {1} {n} - \frac 1 m} \norm {g}_{L^{m}\pr{B_R}}$.
Otherwise, if $f, g \equiv 0$, let $k > 0$ and eventually take to $k \to 0^+$.
Set $\bar u = u + k$.
Let $\xi \in C^\iny_c\pr{B_R}$, $ \xi \ge 0$, and set $\vp = \bar u ^{-\pr{1+\frac 1 2 q_0}} \xi \ge 0$, where $q_0 = q_0\pr{R}$ is the constant given to us in Lemma~\ref{l6.1}.
By \rm{\ref{B2}} with $\om = 1 + \frac{q_0}{2}$ and an application of \rm{\ref{A4}}, $\vp \in \bF_0\pr{B_R}$, so we may use it as a test function.

Set $w = \bar u^{- \frac {q_0} 2 }$ so that $\disp Dw = - \tfrac {q_0} 2  \bar u^{-\pr{1+ \frac {q_0} 2 } } D \bar u$.
By \rm{\ref{B2}}, $w \in \bF\pr{B_R}$ as well.
Then
\begin{align*}
&\int A^{\al \be} D_\be u \, D_\al \vp  + b^\al \, u \, D_\al \vp + d^\be \, D_\be u \, \vp + V \, u \, \vp \\
&= - \frac{2}{q_0} \int A^{\al \be} D_\be w \, D_\al \xi
-4\frac{4 + 2 q_0}{q_0^2} \int A^{\al \be} D_\be \pr{\bar u ^{-\frac 1 4 q_0}} \, D_\al \pr{\bar u ^{-\frac 1 4 q_0}} \, \xi 
+ \int \pr{1 -  \frac{k}{\bar u} }b^\al \, w \, D_\al \xi \\
&+ \frac{2 + q_0}{q_0} \int \pr{1 -  \frac{k}{\bar u} } b^\al \, D_\al w \xi 
- \frac{2}{q_0} \int d^\be \, D_\be w \, \xi 
+ \int \pr{1 -  \frac{k}{\bar u} } V \, w \, \xi  .
\end{align*}
It follows from \eqref{eq6.1}, with $\disp \tilde f = \frac{f}{\bar u}$, $\disp \tilde g^\al = \frac{g^\al}{\bar u}$ that
\begin{align*}
\int A^{\al \be} D_\be w \, D_\al \xi  
&\le -4 \frac{2 + q_0}{q_0}  \int A^{\al \be} D_\be\pr{ \bar u^{-\tfrac{q_0}{4}}} \, D_\al\pr{ \bar u^{-\tfrac{q_0}{4}}} \xi
+\frac{q_0}{2} \int \pr{1 -  \frac{k}{\bar u} } b^\al \, w \, D_\al \xi \\
&+ \pr{1 + \frac{q_0}{2}} \int \pr{1 -  \frac{k}{\bar u} } b^\al \, D_\al w \xi 
-  \int d^\be \, D_\be w \, \xi 
+ \frac{q_0}{2} \int \pr{1 -  \frac{k}{\bar u} } V \, w \, \xi  \\
&+ \frac{q_0}{2}  \int \tilde f \, w \, \xi 
- \pr{1 + \frac {q_0} 2}  \int \tilde g^\al D_\al w \, \xi
- \frac{q_0}{2} \int \tilde g^\al w \, D_\al \xi .
\end{align*}
Therefore, with $\tilde b^\al = \frac{q_0}{2} \brac{ \tilde g^\al + \pr{\frac{k}{\bar u}  - 1 } b^\al }$, $\tilde d^\be = d^\be + \pr{1 + \frac{q_0}{2}}\brac{\pr{\frac{k}{\bar u} - 1 } b^\be + \tilde g^\be}$, and $\tilde V = - \frac{q_0}{2} \pr{V_+ + \tilde f}$, we have that
\begin{align*}
\int A^{\al \be} D_\be w \, D_\al \xi  
+ \tilde b^\al w \, D_\al \xi
+ \tilde d^\be D_\be w \, \xi 
-\tilde V \, w \, \xi 
&\le -4 \pr{1 + \frac 2 {q_0}}  \int A^{\al \be} D_\be\pr{ \bar u^{-\tfrac{q_0}{4}}} \, D_\al\pr{ \bar u^{-\tfrac{q_0}{4}}} \xi
\le 0.
\end{align*}
Since $\xi \in C^\iny_c\pr{B_R}$ is arbitrary and nonnegative, then it follows from \rm{\ref{A2}} that $\widetilde{ \mathcal{L}} w \le 0$ in $B_R$ in the weak sense.
We may apply Lemma~\ref{l5.6} to $w$.
Thus,
\begin{align*}
\sup_{B_{R/2}} w \le C R^{-\frac n 2} \norm{w}_{L^2\pr{B_{3R/4}}}, 
\end{align*}
where $C = C\pr{n, q_0, p, s, t, \ell, m, \la, \La, R^{2 - \frac n p}\norm{V_+}_{L^p\pr{\Om_R}}, R^{1 - \frac n s}\norm{b}_{L^s\pr{\Om_R}}, R^{1 - \frac n t}\norm{d}_{L^t\pr{\Om_R}}}$.
Since $w = \bar u^{-\frac 1 2 q_0}$ and $\bar u = u + k$, then
\begin{align*}
\inf_{B_{R/2}} u + k
&= \inf_{B_{R/2}} \bar u
= \pr{\sup_{B_{R/2}} w }^{-\frac{2}{q_0}}
\ge  \pr{ C R^{- \frac n 2} \norm{w}_{L^{2}\pr{B_{3R/4}}}}^{-\frac{2}{q_0}} 
\ge C^{-\frac 2{q_0}} R^{\frac n {q_0}}\pr{\int_{B_{3R/4}} \bar u^{-q_0} }^{-\frac{1}{q_0}}.
\end{align*}
By Lemma~\ref{l6.1},
\begin{align*}
\pr{\int_{B_{3R/4}} \bar u^{-q_0}}^{-\frac{1}{q_0}} \ge \brac{C_n R^{n} \pr{\fint_{B_{3R/4}} \bar u^{q_0} }^{-1}}^{-\frac{1}{q_0}}
\end{align*}
and therefore,
\begin{align*}
\inf_{B_{R/2}} u + k
\ge \pr{C^2 C_n}^{-\frac 1 {q_0}} \pr{\fint_{B_{3R/4}} \bar u^{q_0} }^{\frac{1}{q_0}}
\ge \pr{C^2 C_n}^{-\frac 1 {q_0}} \pr{\fint_{B_{3R/4}} u^{q_0} }^{\frac{1}{q_0}},
\end{align*}
since $\bar u \ge u \ge 0$.
\end{proof}

By combining our upper and lower bounds, we arrive at the following Harnack inequality.

\begin{lem}
Take $N = 1$.
Assume that $\bF\pr{B_{2R}}$, $\bF_0\pr{B_{2R}}$, $\mathcal{L}$, and $\mathcal{B}$ satisfy \rm{\ref{A1} -- \ref{A5}} and \rm{\ref{B1} -- \ref{B2}}.
Assume that there exists $p \in \pb{\frac n 2, \iny}$, $s, t \in \pb{ n, \iny}$ so that $V \in L^p\pr{B_R}$, $b \in L^s\pr{B_R}^n$, and $d \in L^t\pr{B_R}^n$.
Let $f \in L^\ell\pr{B_{R}}$ for some $\ell \in \pb{\frac n 2, \iny}$.
Let $u \in \bF\pr{B_{2R}}$ be a non-negative solution in the sense that $\disp \mathcal{B}\brac{u, \vp} = \int f \vp$ for any $\vp \in \bF_0\pr{B_{R}}$.
Then
\begin{align*}
\sup_{B_{R/4}} u
&\le C\pr{R} \inf_{B_{R/2}} u 
+ c\pr{R} R^{2 - \frac n {\ell}} \norm{f}_{L^{\ell}\pr{B_R}} ,
\end{align*}
where $C\pr{R} = C C_0 \abs{B_{3/4}}^{\frac 1 {q_0}}$ and $c\pr{R} = C C_0 \abs{B_{3/4}}^{\frac 1 {q_0}} \abs{B_1}^{\frac {2} {n} - \frac 1 \ell} + c_{q_0}$, with $q_0 = q_0\pr{R}$,  $C$ and $c_{q_0}$ as in Lemma~\ref{l5.6}, and $C_0$ as in Lemma~\ref{l6.4}.
\label{l6.5}
\end{lem}

The proof is an application of Lemmas \ref{l5.6} and \ref{l6.4}.

\begin{proof}
By Lemma~\ref{l5.6} with $q_0 =  q_0\pr{R}$, 
\begin{align*}
\sup_{B_{R/4}} u 
&\le C R^{- \frac n {q_0}} \norm{u}_{L^{q_0}\pr{B_{3R/4}}} 
+ c_{q_0} R^{2 - \frac n {\ell}}\norm{f}_{L^{\ell}\pr{B_R}},
\end{align*}
where $C = C\pr{n, q_0, p, s, t, \ell, \la, \La, R^{2 - \frac n p}\norm{V_-}_{L^p\pr{\Om_R}}, R^{1 - \frac n s}\norm{b}_{L^s\pr{\Om_R}}, R^{1 - \frac n t}\norm{d}_{L^t\pr{\Om_R}}}$.
By Lemma~\ref{l6.4},
\begin{align*}
\pr{\fint_{B_{3R/4}} u^{q_0} }^{\frac{1}{q_0}}
\le C_0 \pr{\inf_{B_{R/2}} u + \abs{B_R}^{\frac {2} {n} - \frac 1 \ell}\norm {f}_{L^{\ell}\pr{B_R}} }, 
\end{align*}
where 
$C_0 = C_0\pr{n, q_0, p, s, t, \ell, \la, \La, R^{2-\frac n p} \norm{V_+}_{L^p\pr{B_R}}, R^{1-\frac n s} \norm{b}_{L^s\pr{B_R}}, R^{1-\frac n t} \norm{d}_{L^s\pr{B_R}}}$.
Thus,
\begin{align*}
\sup_{B_{R/4}} u 
&\le C C_0 \abs{B_{3/4}}^{\frac 1 {q_0}} \inf_{B_{R/2}} u
+\pr{C C_0 \abs{B_{3/4}}^{\frac 1 {q_0}} \abs{B_1}^{\frac {2} {n} - \frac 1 \ell} + c_{q_0}} R^{2 - \frac n \ell}\norm{f}_{L^{\ell}\pr{B_R}}
\end{align*}
\end{proof}

Now we have all of the tools we need to prove interior H\"older continuity of solutions.

\begin{lem}
Take $N = 1$.
Assume that $\bF\pr{B_{2R_0}}$, $\bF_0\pr{B_{2R_0}}$, $\mathcal{L}$, and $\mathcal{B}$ satisfy \rm{\ref{A1} -- \ref{A5}} and \rm{\ref{B1} -- \ref{B2}}.
Assume that there exists $p \in \pb{\frac n 2, \iny}$, $s, t \in \pb{ n, \iny}$ so that $V \in L^p\pr{B_{R_0}}$, $b \in L^s\pr{B_{R_0}}^n$, and $d \in L^t\pr{B_{R_0}}^n$.
Let $u \in \bF\pr{B_{2R_0}}$ be a solution in the sense that $\disp \mathcal{B}\brac{u, \vp} = 0$ for any $\vp \in \bF_0\pr{B_{R_0}}$.
Let $C_0 = C_0\pr{R_0}$ be as given in Lemma~\ref{l6.4}.
Then there exists $\eta\pr{n, p, s, C_0} \in \pr{0,1}$, such that for any $R \le R_0$, if $x, y \in B_{R/2}$
\begin{align*}
\abs{u\pr{x} - u\pr{y}}
&\le C \pr{\frac{\abs{x-y}}{R}}^\eta  \pr{\fint_{B_{R}} \abs{u}^{2^*} }^{\frac 1 {2^*}},
\end{align*}
where $C\pr{n, p, s, t, \la, \La, \eta, C_0\pr{R_0}, R_0^{2 - \frac n p} \norm{V}_{L^p\pr{B_{R_0}}}, R_0^{1 - \frac n s} \norm{b}_{L^s\pr{B_{R_0}}}, R_0^{1 - \frac n t} \norm{d}_{L^t\pr{B_{R_0}}}}$.
\label{l6.6}
\end{lem}

\begin{proof}
Assume first that $R = 2$.  
For $r \in \pr{0,1}$, let $\disp m\pr{r} = \inf_{B_r} u$, $\disp M\pr{r} = \sup_{B_r} u$.
By our previous results, $- \iny < m\pr{r} \le M\pr{r} < \iny$.
Set $\disp M_0 = \sup_{B_1} \abs{u}$.
Let $q_0 = q_0\pr{1}$ as given in Lemma~\ref{l6.1}.
The Minkowski inequality shows that
\begin{align}
M\pr{r} - m\pr{r} 
&= \pr{ \fint_{B_{ 3r/ 4}} \abs{M\pr{r} - m\pr{r} }^{q_0} }^{\frac 1 {q_0}} \nonumber \\
&\le \pr{ \fint_{B_{3 r/ 4}} \abs{M\pr{r} - u }^{q_0} }^{\frac 1 {q_0}}
+ \pr{ \fint_{B_{3 r/ 4}} \abs{u - m\pr{r} }^{q_0} }^{\frac 1 {q_0}} .
\label{eq6.5}
\end{align}
Let $\vp \in \bF_0\pr{B_r}$ be such that $\vp \ge 0$ in $B_r$.
Since $M\pr{r} - u \ge 0$ and 
\begin{align*}
\mathcal{B}\brac{M\pr{r} - u, \vp}
&= \int \brac{A^{\al \be} D_\be \pr{M\pr{r} - u} + b^\al \pr{M\pr{r} - u} } D_\al \vp + \brac{d^\be D_\be\pr{M\pr{r} - u} + V \pr{M\pr{r} - u}} \vp \\
&= -\mathcal{B}\brac{u, \vp}
+ M\pr{r} \int \pr{V - D_\al b^\al} \vp
= M\pr{r} \int \pr{V - D_\al b^\al} \vp,
\end{align*}
then by Lemma~\ref{l6.4} with $f :=  -M\pr{r} V \in L^p\pr{B_r}$ and $g^\al := M\pr{r} b^\al \in L^s\pr{B_r}$, 
\begin{align}
\pr{\fint_{B_{3r/4}} \abs{M\pr{r} - u}^{q_0} }^{\frac{1}{q_0}}
\le C_0 \brac{\inf_{B_{r/2}} \brac{M\pr{r} - u} + M_0 \pr{\abs{B_r}^{\frac {2} {n} - \frac 1 p}\norm {V }_{L^{p}\pr{B_r}} +\abs{B_r}^{\frac {1} {n} - \frac 1 s} \norm {b}_{L^{s}\pr{B_r}}}}.
 \label{eq6.6}
 \end{align}
Similarly, since $u - m\pr{r} \ge 0$ and 
\begin{align*}
\mathcal{B}\brac{u - m\pr{r}, \vp}
&= \int \brac{A^{\al \be} D_\be \pr{u - m\pr{r}} + b^\al \pr{u - m\pr{r}} } D_\al \vp + \brac{d^\be D_\be\pr{u - m\pr{r}} + V \pr{u - m\pr{r}}} \vp \\
&= \mathcal{B}\brac{u, \vp}
- m\pr{r} \int \pr{V - D_\al b^\al} \vp  
= - m\pr{r} \int \pr{V - D_\al b^\al} \vp,
\end{align*}
then
\begin{align}
\pr{\fint_{B_{3r/4}} \abs{u - m\pr{r}}^{q_0} }^{\frac{1}{q_0}}
&\le C_0 \brac{ \inf_{B_{r/2}}\brac{u - m\pr{r}} 
+ M_0 \pr{\abs{B_r}^{\frac {2} {n} - \frac 1 p} \norm {V }_{L^{p}\pr{B_r}} +\abs{B_r}^{\frac {1} {n} - \frac 1 s} \norm {b}_{L^{s}\pr{B_r}}}}.
 \label{eq6.7}
\end{align}
Combining \eqref{eq6.5}, \eqref{eq6.6} and \eqref{eq6.7}, we see that
\begin{align*}
\frac{1}{C_0} \brac{ M\pr{r} - m\pr{r} }
&\le M\pr{r} - M\pr{\frac r 2}
+ m\pr{\frac r 2} - m\pr{r} \\
&+ 2 M_0 \pr{ \abs{B_1}^{\frac {2} {n} - \frac 1 p} r^{2 - \frac n p}\norm {V }_{L^{p}\pr{B_r}} 
+\abs{B_1}^{\frac {1} {n} - \frac 1 s}r^{1 - \frac n s} \norm {b}_{L^{s}\pr{B_r}}}.
\end{align*}
Set $\disp \om\pr{r} = \osc_{B_r} u = M\pr{r} - m\pr{r}$,  $\de = \min\set{2 - \frac n p, 1 - \frac n s}$, $c = 2\max\set{\abs{B_1}^{\frac {2} {n} - \frac 1 p}, \abs{B_1}^{\frac {1} {n} - \frac 1 s}}$.
Since $C_0 = C_0\pr{r}$ is monotonically increasing,
\begin{align*}
\om\pr{\frac r 2} 
&\le \pr{1 - \frac 1 {C_0\pr{1}} }  \om\pr{r}
 + c r^{\de} M_0 \pr{\norm{V }_{L^p\pr{B_1}} + \norm{b}_{L^s\pr{B_1}}}.
\end{align*}
Choose $\mu \in \pr{0,1}$, so that $\eta := \pr{1 - \mu}\frac{\log \pr{1 - C_0\pr{1}^{-1}}}{\log \pr{\frac 1 2}} < \mu \de$.
For any such $\eta$, it follows from Lemma 4.19 in \cite{HL11} that for any $\rho \in [0, 1)$, 
\begin{align*}
\om\pr{\rho} 
&\le \frac{2^\eta}{1 - C_0\pr{1}^{-1}} \rho^\eta \om\pr{1}
 + \frac{ c \, C_0\pr{1}  }{2^{\de\pr{1 - \mu}}}  \pr{\norm{V }_{L^p\pr{B_1}} + \norm{b}_{L^s\pr{B_r}}}\rho^{\eta} M_0.
\end{align*}
By Lemma~\ref{l5.6},
\begin{align*}
&\om\pr{1} \le C \pr{\int_{B_2} \abs{u}^{2^*} }^{\frac 1 {2^*}} \\
&M_0 = \sup_{B_1} \abs{u} \le C \pr{\int_{B_2} \abs{u}^{2^*} }^{\frac 1 {2^*}}.
\end{align*}
Thus,
\begin{align*}
\om\pr{\rho} 
&\le C \rho^\eta  \pr{\int_{B_2} \abs{u}^{2^*} }^{\frac 1 {2^*}},
\end{align*}
where $C\pr{n, p, s, t, \la, \La, \eta, C_0\pr{1}, \norm{V}_{L^p\pr{B_2}}, \norm{b}_{L^s\pr{B_2}}, \norm{d}_{L^t\pr{B_2}}}$.
The usual scaling argument gives the general result.
\end{proof}

\section{Examples}
\label{s7}

Within this section, we show that a number of cases satisfy the assumptions from our general set-up:
\begin{enumerate}
\item[Case 1.] {\em Homogeneous operators}: When $\bb, \bd, \bV \equiv \bf 0$, take $\bF\pr{\Om} = Y^{1,2}\pr{\Om}^N$.  This case was studied by Hofmann and Kim in \cite{HK07} and fits into our framework.
\item[Case 2.] {\em Lower order coefficients in $L^p$, Sobolev space}: 
When $\bb, \bd, \bV$ are in some $L^p$ spaces and satisfy a non-degeneracy condition, $\bF\pr{\Om} = W^{1,2}\pr{\Om}^N$.
\item[Case 3:] {\em Reverse H\"older potentials}: When $\bV \in B_p$ for some $p \in \brp{ \frac n 2, \iny}$ (to be defined below), $\bb, \bd \equiv \bf 0$, we define $\bF\pr{\Om} = W^{1,2}_{ V}\pr{\Om}^N$, a weighted Sobolev space, with the weight function depending on the potential function $\bV$.
\end{enumerate}

The goal of this section is to show that each of the three cases listed above fits into the framework described in the Section~\ref{s2}.
More specifically, we first show that $\bF\pr{\Om}$ and $\bF_0\pr{\Om}$ satisfy assumptions {\rm\ref{A1}--\ref{A4}. 
Then we show that \rm{\ref{A5}--\ref{A7}} hold for $\bF\pr{\Om}$, $\bF_0\pr{\Om}$, $\mathcal{L}$ and $\mathcal{B}$; we prove boundedness as in \eqref{eq2.21}, coercivity as in \eqref{eq2.22}, and the Caccioppoli inequality \eqref{eq2.23}.} 
At this point, if we assume that (IB), (BB), and (H) also hold, then we have the full set of results on fundamental and Green matrices. 
Going further, we consider the case of real equations (as opposed to real systems), and we justify the assumptions of (IB), (BB), and (H) in each of the cases described above. 
To this end, due to Sections \ref{s5} and \ref{s6}, we will only have to show that 
\rm{\ref{B1}--\ref{B2}} hold. 
We remind the reader that for systems, the assumptions (IB), (BB), and (H) may actually fail.

\subsection{Homogeneous operators}

We start with the case when $\bV, \bb, \bd \equiv \bf 0$, $\mathcal{L} = L$ and 
$$\mathcal{B}\brac{\bu, \bv} = B\brac{\bu, \bv} := \int \bA^{\al \be} D_\be \bu \cdot D_\al \bv = \int A^{\al \be}_{ij} D_\be u^j \, D_\al v^i.$$
By ellipticity \eqref{eq2.12} and boundedness \eqref{eq2.13} of the matrix $\bA$, $B\brac{\cdot, \cdot}$ is comparable to the inner product given by \eqref{eq2.3}. 
Therefore, it is natural to take the Banach space to be $\bF \pr{\Om}= Y^{1,2}\pr{\Om}^N$, while the associated Hilbert space is $\bF_0 \pr{\Om}= Y^{1,2}_0\pr{\Om}^N$, for all $\Omega$ open and connected.

The restriction property \eqref{eq2.4} is obviously true in this setting.
It is also clear that $C^\iny_c\pr{\Om}^N$ functions belong to  $Y^{1,2}\pr{\Om}^N$, and, by the discussion in the beginning of Section~\ref{s2}, $Y_0^{1,2}\pr{\Om}^N$ is a Hilbert space equipped with the scalar product \eqref{eq2.3}. {\rm\ref{A3}} is trivially satisfied. 

By Lemma~\ref{lD.2}, $C^\iny\pr{U}^N\cap Y^{1,2}\pr{U}^N$ is dense in $Y^{1,2}\pr{U}^N$ for any bounded $U$.
This implies \eqref{eq2.7} since we may assume that $U$ in \eqref{eq2.7} is bounded because the support of $\xi$ is bounded.
With $\xi \in C_c^\infty\pr{U}$, it is immediate that $\bu \xi \in L^{2^*}\pr{\Om \cap U}^N$ and
\eqs{\frac{\partial}{\partial x_i}\pr{\bu \xi} = \xi \frac{\partial \bu}{\partial x_i} + \bu \frac{\partial \xi}{\partial x_i} \in L^2\pr{\Om \cap U}^N,}
where we have used that $\bu\in L^{2^*}\pr{\Om \cap U}^N \hookrightarrow L^2\pr{\Om \cap U}^N$ since $U$ is bounded.
It follows that $\norm{\bu \xi}_{Y^{1,2}\pr{\Om\cap U}} \le C_\xi \norm{\bu}_{Y^{1,2}\pr{\Om}}$.
Now if $\set{\bu_n} \su C^\infty (\Om\cap U)^N$ approximates $\bu$ in the $Y^{1,2}(\Om\cap U)^N$-norm, then for $\xi \in C^\iny_c\pr{U}$, we observe that $\set{\bu_n \xi} \su C^\infty (\Om\cap U)^N$ approximates $\bu \xi$ since 
\begin{align}\label{eq7.1}
\|\bu_n\xi-\bu\xi\|_{Y^{1,2}(\Om\cap U)}
&\le \|D(\bu_n-\bu)\|_{L^2(\Om\cap U)}\|\xi\|_{L^\infty(\Om\cap U)}
+\|\bu_n-\bu\|_{L^2(\Om\cap U)}\|D\xi\|_{L^\infty(\Om\cap U)}\\
&+\|\bu_n-\bu\|_{L^{2^*}(\Om\cap U)}\|\xi\|_{L^\infty(\Om\cap U)}
\nonumber.
\end{align}
Applying the H\"older inequality to the second term, the latter is majorized by $\|\bu_n-\bu\|_{Y^{1,2}(U\cap \Omega)}$, as desired.
A similar argument implies that when $\xi \in C^\iny_c\pr{\Om \cap U}$, $\set{\bu_n \xi} \su C_c^\infty (\Om\cap U)^N$ approximates $\bu \xi$.

Turning to A5)--A7), \eqref{eq2.21} and \eqref{eq2.22} follow directly from \eqref{eq2.13} and \eqref{eq2.12}
with $\Ga = \La$ and $\ga = \la$.
The Caccioppoli inequality is well-known in this context, however one can also refer to Lemma~\ref{l4.1}. 
Indeed, since all of the lower order coefficients vanish, then Lemma~\ref{l4.1} applies to give the Caccioppoli inequality \eqref{eq2.23} with $C = C\pr{n, \la, \La}$.
All in all, \rm{\ref{A1} -- \ref{A7}} are verified in this setting.

Reducing to the case of equations, i.e., $N=1$, conditions (IB) and (BB) hold with $C = C\pr{n, q, \ell, \la, \La}$ due to Lemma 4.1 from \cite{HL11}, or one could also use Lemma~\ref{l5.1} by showing that \rm{\ref{B1}} holds.

If one wants to show \rm{\ref{B1}}, it is enough to observe that its proof can be reduced to the case of $\bF(\Omega)=W^{1,2}(\Omega)$. 
This is because $Y_0^{1,2}(\Omega_R)=W_0^{1,2}(\Omega_R)$ by Lemma~\ref{lA.7}. Indeed, for any $u\in  Y^{1,2}(\Omega_R) \hookrightarrow W^{1,2}(\Omega_R)$ (see Lemma~\ref{lA.7}), if $u\zeta \in Y_0^{1,2}(\Omega_R)$ for all $\zeta\in C_c^\infty (B_R)$, then $u\zeta \in W_0^{1,2}(\Omega_R)$. 
If \rm{\ref{B1}} holds for $\bF(\Omega)=W^{1,2}(\Omega)$, we have for all  $\zeta$ smooth compactly supported non-negative $\zeta (u-k)_+ \in  W_0^{1,2}(\Omega_R)=Y_0^{1,2}(\Omega_R)$ by Lemma~\ref{lA.7}, as desired. Clearly, the property $\partial^i\zeta (u-k)_+ \in L^2(\Omega_R)$ is also inherited. We will postpone the proof of \rm{\ref{B1}} for $\bF(\Omega)=W^{1,2}(\Omega)$ to Case 2. 

In this context, (H) also can be found in the literature, specifically, Corollary 4.18 from \cite{HL11} applies since the spaces $W^{1,2}(B_R)$ and $Y^{1,2}(B_R)$ coincide for any $B_R\subset \Omega$  (see Corollary~\ref{cA.11}). 
The latter fact also allows us to reduce the proof of \rm{\ref{B2}} to the case of $\bF(\Omega)=W^{1,2}(\Omega)$ (discussed below) should we prefer to use Lemma~\ref{l6.6}. 

\subsection{Lower order coefficients in $L^p$, Sobolev space}

Assume that there exist exponents $p \in \pb{ \frac n 2, \iny}$, $s, t \in \pb{n, \iny}$ so that $\bV  \in L^p\pr{\Om}^{N\times N}$, $\bb \in L^s\pr{\Om}^{n \times N \times N}$, and $\bd \in L^t\pr{\Om}^{n \times N \times N}$.
Set $\bF\pr{\Om} = W^{1,2}\pr{\Om}^N$ and $\bF_0\pr{\Om} = W_0^{1,2}\pr{\Om}^N$. 

To establish the assumptions \rm{\ref{A1}} through \rm{\ref{A4}}, we rely on a number of facts regarding Sobolev spaces which are contained in Appendix \ref{AppD}, with further details in \cite{Eva98}, for example.

The property \eqref{eq2.4} is straightforward and therefore \rm{\ref{A1}} holds. Clearly, 
$C^\iny_c\pr{\Om}^N$ is contained in $W^{1,2}\pr{\Om}^N$ and the completion, $W^{1,2}_0\pr{\Om}^N$, is a Hilbert space with respect to $\|\cdot\|_{W^{1,2}_0\pr{\Om}^N}=\|\cdot\|_{W^{1,2}\pr{\Om}^N}$.
{\rm\ref{A3}} follows from Lemma~\ref{lA.1}. 
For $\bu \in W^{1,2}\pr{\Om}^N$ and $\xi \in C^\infty_c\pr{U}$, boundedness of $\xi$ and $D\xi$ implies that $\bu \xi \in W^{1,2}\pr{\Om \cap U}^N$, and, as in the previous case, $\norm{\bu \xi}_{W^{1,2}\pr{\Om\cap U}} \le C_\xi \norm{\bu}_{W^{1,2}\pr{\Om}}$.
By Lemma~\ref{lD.2}, $C^\iny\pr{U}^N\cap W^{1,2}\pr{U}^N$ is dense in $W^{1,2}\pr{U}^N$, so that \eqref{eq2.7}, and hence \rm{\ref{A4}}, holds by the same argument as in Case 1, similar to \eqref{eq7.1}. 

Boundedness of the matrix $\bA$, \eqref{eq2.13}, implies that for any $\bu, \bv \in W_0^{1,2}\pr{\Om}^N$,
\begin{align*}
\mathcal{B}\brac{\bu, \bv} 
&\le \La \int \abs{D \bu} \abs{ D \bv} 
+ \int \abs{\bb} \abs{\bu} \abs{ D \bv} 
+ \int \abs{\bd} \abs{D \bu} \abs{\bv}
+ \int \abs{\bV} \abs{\bu}\abs{\bv}.
\end{align*}
By the H\"older inequality
\begin{align*}
 \int \abs{D \bu} \abs{ D \bv}  
&\le  \pr{\int \abs{D \bu}^2}^{\frac 1 2} \pr{\int \abs{ D \bv} ^2}^{\frac 1 2} 
\end{align*}
By H\"older, homogeneous Sobolev and Young's inequalities, since $s \in \pb{n, \iny}$,
\begin{align*}
\int \abs{\bb} \abs{\bu} \abs{ D \bv} 
&= \int \abs{\bb} \abs{\bu}^{\frac{s-n}{s}} \abs{\bu}^{\frac n s} \abs{ D \bv} 
\le \pr{\int \abs{\bb}^s}^{\frac 1 s} \pr{\int \abs{\bu}^2 }^{\frac{s-n}{2s}} \pr{\int \abs{\bu}^{2^*} }^{\frac{n-2}{2s}}  \pr{\int \abs{ D \bv}^2}^{\frac 1 2}\\
&\le c_{n}^{\frac n {2s}} \norm{\bb}_{L^s\pr{\Om}} \pr{\int \abs{\bu}^2 }^{\frac{s-n}{2s}} \pr{\int \abs{D\bu}^{2}}^{\frac {n}{2s}} \pr{\int \abs{ D \bv}^2}^{\frac 1 2} \\
&\le c_{n}^{\frac n {2s}} \norm{\bb}_{L^s\pr{\Om}} \brac{ \pr{1 - \frac n s} \int \abs{\bu}^2+  \frac n s \int \abs{D\bu}^{2}}^{\frac {1}{2}} \pr{\int \abs{ D \bv}^2}^{\frac 1 2},
\end{align*}
where we as usual interpret $\frac 1 s$ to be $0$ in the case where $s = \iny$.
Similarly,
\begin{align*}
\int \abs{\bd} \abs{D \bu} \abs{\bv}
&\le  c_{n}^{\frac n {2t}} \norm{\bd}_{L^t\pr{\Om}} \pr{\int \abs{ D \bu}^2}^{\frac 1 2} \brac{\pr{1 - \frac{n}{t}} \int \abs{\bv}^2 + \frac n t \int \abs{D\bv}^{2}}^{\frac {1}{2}},
\end{align*}
and
\begin{align*}
\int \abs{\bV} \abs{\bu} \abs{\bv}
&\le c_{n}^{\frac n {2p}} \norm{\bV}_{L^p\pr{\Om}} \brac{\pr{1 - \frac{n}{2p}} \int \abs{\bu}^{2} + \frac n {2p} \int \abs{D\bu}^{2} }^{\frac{1}{2}}\brac{\pr{1 - \frac{n}{2p}} \int \abs{\bv}^{2}+ \frac{n}{2p} \int \abs{D\bv}^{2} }^{\frac{1}{2}} .
\end{align*}
Combining these inequalities, we see that
\begin{align*}
\mathcal{B}\brac{\bu, \bv} &\le \pr{\La 
+ c_{n}^{\frac n {2s}} \norm{\bb}_{L^s\pr{\Om}} 
+ c_{n}^{\frac n {2t}} \norm{\bd}_{L^t\pr{\Om}}
+c_{n}^{\frac n {2p}} \norm{\bV}_{L^p\pr{\Om}}} \norm{\bu}_{W^{1,2}\pr{\Om}^N}  \norm{\bv}_{W^{1,2}\pr{\Om}^N}.
\end{align*}
Therefore, we may take $\disp \Ga = \La + c_{n}^{\frac n {2s}} \norm{\bb}_{L^s\pr{\Om}} + c_{n}^{\frac n {2t}} \norm{\bd}_{L^t\pr{\Om}} + c_{n}^{\frac n {2p}} \norm{\bV}_{L^p\pr{\Om}}$ so that \eqref{eq2.21}, and therefore \rm{\ref{A5}}, holds.
Clearly, the estimate from below on $ \mathcal{B}\brac{\bu, \bu}$ may or may not be satisfied without further assumptions on the lower order terms.
Thus, we have to assume that for some $\ga > 0$, depending on $\la, \bV, \bb, \bd$,
\begin{align*}
& \ga  \pr{\norm{\bu}_{L^2\pr{\R^n}^N}^2 + \norm{D\bu}_{L^2\pr{\R^n}^N}^2 } \le \mathcal{B}\brac{\bu, \bu} .
\end{align*}
In other words, we assume that \eqref{eq2.22} holds.
This is valid, for instance, if $\bV$ is positive definite and the first order terms are small with respect to the zeroth and second order terms. 
To be specific, we say that $\bV$ is positive definite if there exists $\eps > 0$ so that for any $\xi \in \R^N$, $V_{ij}\pr{x} \xi^i \xi^j \ge \eps \abs{\xi}^2$ for every $x \in \Om$.
In this case,
\begin{align*}
\mathcal{B}\brac{\bu, \bu} 
&\ge \la \int \abs{D \bu}^2
+\int \bb^\al \, \bu \cdot D_\al \bu 
+ \int \bd^\be D_\be \bu \cdot \bu
+ \eps \int \abs{\bu}^2.
\end{align*}
If $\bb$ and $\bd$ are small in the sense that for some $\de_1, \de_2 > 0$
\begin{align*}
\abs{\int \bb^\al \, \bu \cdot D_\al \bu 
+ \int \bd^\be D_\be \bu \cdot \bu}
&\le \frac{\la}{1+ \de_1} \int \abs{D \bu}^2
+ \frac{\eps}{1 + \de_2} \int \abs{\bu}^2,
\end{align*}
then it follows that $\mathcal{B}\brac{\bu, \bu} \ge \ga \norm{\bu}_{W^{1,2}\pr{\Om}^N}^2$, where $\disp \ga = \min\set{\frac{\la \de_1}{1 + \de_1}, \frac{\eps \de_2}{1 + \de_2}}$.
There are other conditions that we could impose to ensure that the lower bounds holds for some $\ga > 0$.
When $N=1$, the lower bound holds also in the presence of more involved non-degeneracy assumptions on the zeroth and first order terms that we discuss below.

By Lemma~\ref{l4.1}, the Caccioppoli inequality, \eqref{eq2.23}, holds with $C = C\pr{n, s, t, \ga, \La, \norm{\bb}_{L^s\pr{\Om}}, \norm{\bd}_{L^t\pr{\Om}}}$.

Moving towards (IB), (BB), (H), when $N = 1$, 
\begin{align}
\mathcal{L} u 
&= -D_\al\pr{A^{\al \be} D_\be u + b^\al u} + d^\be D_\be u + V u.
\label{eq7.2}
\end{align}
where $\la \abs{\xi}^2 \le A^{\al \be}\pr{x}\xi_\al \xi_\be \le \La \abs{\xi}^2$ for all $x \in \Om$, $\xi \in \R^n$, $V \in L^p\pr{\Om}$, $b^\al \in L^s\pr{\Om}$, and $d^\be \in L^t\pr{\Om}$.
Moreover,
\begin{align}
\mathcal{B}\brac{u, v} 
&= \int A^{\al \be} D_\be u \; D_\al v + b^\al \, u \; D_\al v + d^\be D_\be u \; v + V \, u \, v.
\label{eq7.3}
\end{align}
Since $u \in L^2\pr{\Om} \cap L^{2^*}\pr{\Om}$ and $Du \in L^2\pr{\Om}$, then by an application of the H\"older inequality $D\abs{u}^2 = 2 u \, Du \in L^p\pr{\Om}$ for any $\disp p \in \brac{1, \frac n {n-1}}$.  
It follows that $D_\al b^\al$ and $D_\be d^\be$ can be paired with $\abs{u}^2$ in the sense of distributions.
That is,
\begin{align}
\mathcal{B}\brac{u, u} 
&= \int A^{\al \be} D_\be u \; D_\al u + \frac 1 2 b^\al \, D_\al \abs{u}^2 + \frac 1 2 d^\be D_\be \abs{u}^2 + V \, \abs{u}^2 \nonumber \\
&= \int A^{\al \be} D_\be u \; D_\al u +  \pr{V - \frac 1 2 D_\al  b^\al - \frac 1 2  D_\be  d^\be} \abs{u}^2,
\label{eq7.4}
\end{align}
where the integrals above are interpreted as pairings in dual spaces.

Note that to ensure coercivity of the bilinear form, it suffices, for example,  to assume that there exists $\de > 0$ so that $V - \frac 1 2 D_\al  b^\al - \frac 1 2  D_\be  d^\be \ge \de$ in the sense of distributions.
That is, for any $\vp \in C^\iny_c\pr{\Om}$ such that $\vp \ge 0$, 
\begin{align*}
\int \pr{V - \frac 1 2 D_\al  b^\al - \frac 1 2  D_\be  d^\be - \de} \vp \ge 0.
\end{align*}
In this case, we see from \eqref{eq7.4} that the bound from below, \eqref{eq2.22}, holds with $\ga = \min\set{\la, \de}$.
If we further assume that $V - D_\al  b^\al \ge 0$ and $V - D_\be  d^\be \ge 0$ in $\Om$ in the sense of distributions, then Lemma~\ref{l5.1} implies that (IB) and (BB) hold for this setting with $C = C\pr{n, q, s, t, \ell, \ga, \La, \norm{b}_{L^s\pr{\Om}}, \norm{d}_{L^t\pr{\Om}}}$ in \eqref{eq3.47} and \eqref{eq3.90} as long as {\rm\ref{B1}} holds.
If \rm{\ref{B2}} also holds, then it follows from Lemma~\ref{l6.6} that assumption (H) is also valid.

Therefore, we need to show that assumptions \rm{\ref{B1}} and \rm{\ref{B2}} are valid for $\bF(\Omega)=W^{1,2}(\Omega)$. 
These facts are commonly used in the classical arguments for De Giorgi-Nash-Moser estimates, but the proofs are often omitted. 
One can find details, e.g., in \cite{HKM93}. 
Since $\Om_R$ is bounded, then Lemma~\ref{lD.2} implies that $W^{1,2}(\Om)$ could also be defined as a completion of $C^\infty(\Om_R)$ in the $W^{1,2}(\Om)$-norm, thereby coinciding with the Sobolev space $H^{1,2}(\Om; dx)$ of \cite{HKM93}. 
Then, given that $u \in W^{1,2}(\Om_R)$, Theorem~1.20 of \cite{HKM93} implies that $(u-k)_+ \in W^{1,2}(\Omega)$, and therefore $(u-k)_+ \zeta \in W^{1,2}(\Omega)$, $(u-k)_+ \partial^i\zeta \in L^2(\Omega)$, $i=1,...,n$ (by a direct computation). 
Also, since we assume that $u\zeta \in W^{1,2}_0(\Om)$, then $(u\zeta)_+\in W^{1,2}_0(\Omega)$ by Lemma~1.23 of \cite{HKM93}. 
Finally, if $\zeta$ and $k$ are non-negative, $0\leq (u-k)_+\zeta \leq (u\zeta)_+$ and hence, $(u-k)_+\zeta \in W^{1,2}_0(\Omega)$ by Lemma~1.25(ii) of \cite{HKM93}, as desired. 

To show that \rm{\ref{B2}} holds, we use a modification of the arguments given in Theorem~1.18 of \cite{HKM93}.
We work with $f(t)=(t+k)^{-\omega}$, $t\geq 0$, which belongs to $C^1(\overline{\RR_+})$ and has a bounded derivative on $\R^+$ (not on the entire $\RR$).
The exact same argument applies upon observing that a non-negative function $u \in W^{1,2}\pr{B_R}$ can be approximated by non-negative $u_i \in C^\iny\pr{B_R}$ due to Corollary~\ref{cD.3}.

\subsection{Reverse H\"older potentials}

Recall that $B_p$, $1 < p < \iny$, denotes the reverse H\"older class of all (real-valued) nonnegative locally $L^p$ integrable functions that satisfy the reverse H\"older inequality.
That is, $w \in B_p$ if there exists a constant $C$ so that for any ball $B \su \R^n$,
\begin{align}
\pr{\fint_B w\pr{x}^p dx}^{1/p} \le C \fint_B w\pr{x} dx.
\label{eq7.5}
\end{align}
If $w \in B_p$, then it follows from an application of the H\"older inequality that $w \in B_q$ for any $q < p$.
Moreover, if $w \in B_p$, then there exists $\eps > 0$, depending on $w$ and $n$, so that $w \in B_{p+\eps}$ as well \cite{Geh73}. 

For an $N \times N$ matrix function ${\bf M}\pr{x}$, define lower and upper bounds on $\bf M$ in the following way
\begin{align*}
M_{\ell}\pr{x} &= \inf\set{ M_{ij}\pr{x} \xi_j \xi_i : \xi \in \R^N, \abs{\xi} = 1  } \\
M_{u}\pr{x} &= \sup \set{ \abs{M_{ij}\pr{x} \xi_j \zeta_i } : \xi, \zeta \in \R^N, \abs{\xi} = 1 = \abs{\zeta} }.
\end{align*}
For the zeroth order term $\bV$, we assume that there exist constants $c_1, c_2 > 0$ and a non-trivial $V \in B_{p}$ for some $p \in \brp{ \frac{n}{2}, \iny}$ (and therefore $p \in \pr{ \frac n 2, \iny}$ without loss of generality) so that 
\begin{equation}
c_1 V \le V_\ell \le V_u \le c_2 V.
\label{eq7.6}
\end{equation}
Even if $\Om$ is a proper subset of $\R^n$, we still assume that $\bV$ is associated to some $V \in B_p$ which is defined on all of $\R^n$.
Since $V \in B_p$, then $V$ is a Muckenhoupt $A_\iny$ weight, and it follows that $V\pr{x} dx$ is a doubling measure.  
As $V$ is assumed to be non-trivial, it follows from the doubling property that $V$ cannot vanish on any open set.
We set $\bb, \bd = \bf 0$.

One might wonder whether an appropriate matrix $B_p$ class could be suitable in this context. 
We did not pursue this topic, in part, because the theory of matrix reverse H\"older classes seems to be largely undeveloped.
For the case of $p=2$, some (very limited) discussion can be found in \cite{Ros14}.
Developing the theory of matrix $B_p$ for $p \ne 2$ was not in the scope of the present work.

Let $m\pr{x,V}$ denote the Fefferman-Phong maximal function associated to $V \in B_p$. 
This function was introduced by Shen in \cite{She94}, motivated by the work of Fefferman and Phong in \cite{Fef83}.
For the definition and additional details we refer the reader to Appendix \ref{AppB}.
For any open set $\Om \su \R^n$, we define the space $W_V^{1,2}\pr{\Om}$ as the family of all weakly differentiable functions $u \in L^{2}\pr{\Om, m\pr{x, V}^2 dx}$, whose weak derivatives are functions in $L^2\pr{\Om, dx}$.
We endow the space $W_V^{1,2}\pr{\Om}$ with the norm
\begin{align}
\norm{u}^2_{W_V^{1,2}\pr{\Om}} :=  \norm{u}^2_{L^{2}\pr{\Om, m\pr{x, V}^2 dx}} + \norm{D u}^2_{L^2\pr{\Om, dx}}
= \norm{u \, m\pr{\cdot, V}}^2_{L^{2}\pr{\Om}} + \norm{D u}^2_{L^2\pr{\Om}}.
\label{eq7.7}
\end{align}
Since $m\pr{\cdot, V}$ is non-degenerate in the sense that it is bounded away from zero on any bounded set (see for example Lemma~\ref{lB.3} and Remark \ref{rB.5}), it follows that $\norm{\cdot}^2_{W_V^{1,2}\pr{\Om}}$ is indeed a norm and this norm makes the space complete (details in Appendix \ref{AppC}).
The space $W^{1,2}_{0, V}\pr{\Om}$ is defined as the closure of $C^\iny_c\pr{\Om}$ in $W_V^{1,2}\pr{\Om}$.
For further properties of $W^{1,2}_V\pr{\Om}$, we refer the reader to Appendix \ref{AppC}. 
Perhaps the most important fact that we want to highlight here is that on bounded sets $W^{1,2}_V\pr{\Om}$ and $W^{1,2}\pr{\Om}$ coincide (with the norm comparison depending on the set though) -- see Remark~\ref{rC.1}. 

\begin{rem}\label{r7.1}
Using $V\pr{x}$ in place of $m\pr{x, V}^2$, we define the space $\hat{W}_V^{1,2}\pr{\Om}$ as the family of all weakly differentiable functions $u \in L^{2}\pr{\Om, V\pr{x} dx}$ whose weak derivatives are functions in $L^2\pr{\Om, dx}$.
The space $\hat{W}_V^{1,2}\pr{\Om}$ is endowed with the norm
\begin{align}
\norm{u}^2_{\hat{W}_V^{1,2}\pr{\Om}} :=  \norm{u}^2_{L^{2}\pr{\Om, V\pr{x} dx}} + \norm{D u}^2_{L^2\pr{\Om, dx}}
= \norm{u \, V^{1/2}}^2_{L^{2}\pr{\Om}} + \norm{D u}^2_{L^2\pr{\Om}}.
\label{eq7.8}
\end{align}
Since $V$ is also non-degenerate, this is indeed a norm and not a semi-norm.
We define $\hat W^{1,2}_{0, V}\pr{\Om}$ as the closure of $C^\iny_c\pr{\Om}$ in $\hat{W}_V^{1,2}\pr{\Om}$.

The space $\hat{W}_V^{1,2}\pr{\Om}$ serves as an alternative (but not equivalent) Hilbert space to $W_V^{1,2}\pr{\Om}$ for the case of reverse H\"older zeroth order terms. 
The spaces $\hat W^{1,2}_{0, V}\pr{\Om}$ and $W^{1,2}_{0, V}\pr{\Om}$ are  the same -- see Appendix~\ref{AppC}. 
In practice, we find it easier to work with  $W_V^{1,2}\pr{\Om}$ compared to  $\hat{W}_V^{1,2}\pr{\Om}$ due to the aforementioned fact that $W_V^{1,2}\pr{\Om}$ coincides with the usual Sobolev spaces $W^{1,2}\pr{\Om}$ whenever $\Omega$ is bounded.

\end{rem}


For $\bV$ specified above, we set $\bF\pr{\Om} = W^{1,2}_{V}\pr{\Om}^N$ and $\bF_0\pr{\Om} = W^{1,2}_{0,V}\pr{\Om}^N$. 
Define the inner product on $W^{1,2}_{V}\pr{\Om}^N$ by
\begin{align*}
\innp{\bu, \bv}_{W^{1,2}_{V}\pr{\Om}^N} := \int_{\Om} D_\al u^i D_\al v^i + u^i v^i m\pr{\cdot, V}^2.
\end{align*}

As above, \rm{\ref{A1}} and \rm{\ref{A2}} follow directly from the definition. 
{\rm\ref{A3}} is shown using the exact same argument as that for Lemma~\ref{lA.1}.
For $\bu \in W^{1,2}_V\pr{\Om}^N$ and $\xi \in C^\infty_c\pr{U}$, it follows from the boundedness of $\xi$ and $D\xi$, along with Remark \ref{rB.5}, that $\bu \xi \in W^{1,2}_V\pr{\Om \cap U}^N$ with $\norm{ \bu \xi}_{W_V^{1,2}\pr{\Om \cap U}} \le C_\xi \norm{\bu}_{W^{1,2}_V\pr{\Om}}$.  
Using the density of smooth functions in $W^{1,2}_V\pr{U}^N$ for any bounded domain $U$, i.e., Lemma~\ref{lD.2}, the remainder of \rm{\ref{A4}} follows from the arguments in Case 1 and Case 2, with appropriate modifications to \eqref{eq7.1}.

The next goal is to show that boundedness and coercivity given by \eqref{eq2.21} and \eqref{eq2.22} hold. 
At this point we recall Lemma~\ref{lC.6}. 
Having that at hand, 
for any $\bu, \bv \in W^{1,2}_{0,V}\pr{\Om}$ we have
\begin{align*}
\mathcal{B}\brac{\bu, \bv} 
&= \int \bA^{\al \be} D_\be \bu \cdot D_\al \bv + \bV \, \bu \cdot \bv 
\le \La \int \abs{D \bu} \abs{ D \bv} 
+ c_2 \int V \abs{\bu} \abs{\bv} \\
&\le \La \pr{\int \abs{D \bu}^2}^{\frac 1 2} \pr{\int \abs{ D \bv}^2}^{\frac 1 2}
+ c_2 \pr{\int V \abs{\bu}^2}^{\frac 1 2} \pr{ \int V \abs{\bv}^2}^{\frac 1 2}  \\
&\le \pr{\La + c_2\, C_{V,n} } \norm{\bu}_{W^{1,2}_V\pr{\Om}^N} \norm{\bv}_{W^{1,2}_V\pr{\Om}^N},
\end{align*}
where the last line follows from Lemma~\ref{lC.6}.
Therefore, boundedness holds with $\Ga = \La + C\,C_{V,n}$.
Since
\begin{align*}
\mathcal{B}\brac{\bu, \bu} 
&= \int \bA^{\al \be} D_\be \bu \cdot D_\al \bu + \bV \, \bu \cdot \bu 
\ge \la \int \abs{D \bu}^2
+ c_1 \int V \abs{\bu}^2 \\
&\ge \frac{\la}{2} \int \abs{D \bu}^2
+ \min\set{\frac \la 2, c_1} \brac{ \int \abs{D \bu}^2 + \int V \abs{\bu}^2 },
\end{align*}
then by another application of Lemma~\ref{lC.6}, we see that coercivity holds with $\ga = \min\set{\frac \la 2, \frac \la {2C}, \frac{c_1}{C}}$.

By Lemma~\ref{l4.1}, \eqref{eq2.23} holds with $C = C\pr{n, \ga, \La}$.

When $N = 1$, $\mathcal{L}$ and $\mathcal{B}$ are given by \eqref{eq7.2} and \eqref{eq7.3}, respectively, with $b, d =0$ and $V \in B_p$ for some $p \in \pr{\frac n 2, \iny}$, without loss of generality.
By the non-negativity of $V$, Lemma~\ref{l5.1} implies that (IB) and (BB) hold for this setting with $C = C\pr{n, q,  \ell, \ga, \La}$ in \eqref{eq3.47} and \eqref{eq3.90} whenever \rm{\ref{B1}} holds.
Since $V \in L^p_{loc}$, Lemma~\ref{l6.6} shows that assumption (H) holds under the additional assumption of \rm{\ref{B2}}. 
In turn, {\rm\ref{B1}} and {\rm\ref{B2}} in the setting of $\bF(\Om) =W^{1,2}_V(\Om)^N$ follow directly from the same statements for $\bF(\Om) =W^{1,2}(\Om)^N$, i.e., Case 2, and Remark~\ref{rB.5} since $\Om_R$ and $B_R$ are bounded and the statements {\rm\ref{B1}} and {\rm\ref{B2}} are qualitative (they assure membership in the corresponding function spaces, without particular norm control).

\begin{rem}\label{r7.2} We point out that for the case of equations ($N=1$) with the potentials in $B_p$ class for some $p \in \brp{ \frac n 2, \iny}$, a stronger version of the Harnack inequality than Lemma~\ref{l6.5} is possible, without the dependence of constants on the size of the ball \cite{CFG86}. 
In the present paper, we do not need this stronger estimate, and we aim to keep the discussion uniform across several cases.  
\end{rem}

\subsection{Conclusions} \label{s7.4} 
From the above arguments, we conclude that $\bG\pr{x,y}$ exists and satisfies the estimates of Theorem~\ref{t3.6}, where in the vector case ($N>1$) we must assume that (IB) and (H) hold for solutions.
 The estimates of Theorem~\ref{t3.6} imply immediately that
 \eqs{\bG\pr{\cdot,y}\in Y^{1,2}\pr{\R^n\setminus B_r\pr{y}}^{N\times N} \quad \text{for any $r>0$.}}
 With these estimates, however, it does not follow that $\bG\pr{\cdot, y} \in \bF\pr{\R^n \setminus B_r\pr{y}}$ for the general space $\bF$.
 
 Nevertheless, in many reasonable cases it is true.
 In Case 1, it follows clearly since $\bF\pr{\R^n \setminus B_r\pr{y}} = Y^{1,2}\pr{\R^n \setminus B_r\pr{y}}^N$.
 In Case 2, it is true locally -- i.e., we have that $\bG\pr{\cdot, y}\in W^{1,2}_{\loc}\pr{\R^n\setminus \{y\}}^{N\times N}$ because of the relationship between the spaces (see Lemma~\ref{lA.7}).
 Furthermore, for $|U|<\infty$, the space $Y^{1,2}(U)$ embeds continuously into $W^{1,2}\pr{U}$, so we have
 \eq{\label{eq7.9}\norm{\bG\pr{\cdot, y}}_{W^{1,2}\pr{U\setminus B_r\pr{y}}} \le C_U \norm{\bG\pr{\cdot, y}}_{Y^{1,2}\pr{U\setminus B_r\pr{y}}} \le C_U C r^{1-\frac{n}{2}}, \quad \forall r>0,}
 where $C$ is the constant of Theorem~\ref{t3.6}.
 In Case 3, observe that for $|U|<\infty$,
 \eqs{W^{1,2}_V\pr{U}^{N\times N} \hookrightarrow W^{1,2}\pr{U}^{N\times N} \hookrightarrow Y^{1,2}\pr{U}^{N\times N}.}
 (see Remark~\ref{rC.1}).
 Thus a similar estimate to \eqref{eq7.9} holds in Case 3.
 
By the same reasoning, similar conclusions hold for $\bG\pr{x, \cdot}$, $\bGr\pr{\cdot, y}$, and $\bGr\pr{x, \cdot}$.


\begin{appendix}

\section{Function spaces $Y^{1,2}$ and $W^{1,2}$}
\label{AppA}

Let $\Om$ be an open, connected subset of $\R^n$, $n\ge 3$.  
Let us recall the definitions. 
Define the space $Y^{1,2}\pr{\Om}$ as the family of all weakly differentiable functions $u \in L^{2^*}\pr{\Om}$, with $2^* = \frac{2n}{n-2}$, whose weak derivatives are functions in $L^2\pr{\Om}$, endowed with the norm
\eqs{
\norm{u}_{Y^{1,2}\pr{\Om}} = \norm{u}_{L^{2^*}\pr{\Om}} + \norm{D u}_{L^2\pr{\Om}}.
}
Define $Y_0^{1,2}$ to be the closure of $C_c^\infty\pr{\Om}$ in the $Y^{1,2}\pr{\Om}$-norm.  
Define $W^{1,2}\pr{\Om}$ to be the space of all weakly differentiable functions $u \in L^{2}\pr{\Om}$, whose weak derivatives are functions in $L^2\pr{\Om}$, endowed with the norm
\eqs{
\norm{u}_{W^{1,2}\pr{\Om}} = \norm{u}_{L^{2}\pr{\Om}} + \norm{D u}_{L^2\pr{\Om}}.
}
Let $W_0^{1,2}\pr{\Om}$ be the closure of $C_c^\infty\pr{\Om}$ in the $W^{1,2}$-norm. 

This section will explore various connections between $W$ and $Y$-spaces. 
We remark that for any open connected set $\Omega$ in $\rn$, by completeness of $W^{1,2}\pr{\Om}$ and $Y^{1,2}\pr{\Om}$,
\begin{equation}\label{eqA.1} 
W_0^{1,2}\pr{\Om} \hookrightarrow W^{1,2}\pr{\Om} \mbox{ and } Y_0^{1,2}\pr{\Om} \hookrightarrow Y^{1,2}\pr{\Om}.
\end{equation}

\begin{lem} \label{lA.1}
For any open set $\Omega\subset \rn$
$$ W_0^{1,2}(\Omega)\hookrightarrow Y_0^{1,2}(\Omega). $$
\end{lem}
\begin{pf}
Let $u \in W^{1,2}_0\pr{\Om}$.
Then there exists $u_i \in C^\iny_c\pr{\Om}$ such that $\disp \lim_{i \to \iny} \norm{u_i - u}_{W^{1,2}\pr{\Om}} = 0$.
By the Sobolev inequality applied to $u_i-u_k$ we have 
$$\norm{u_i-u_k}_{L^{2^*}\pr{\Om}}\leq c_n  \norm{Du_i-Du_k}_{L^{2}\pr{\Om}} \leq c_n \norm{u_i-u_k}_{W^{1,2}\pr{\Om}},  $$ 
and therefore, $\{u_i\}_{i=1}^\infty$ is Cauchy in $Y^{1,2}(\Omega)$. Hence, there is a limit in  $Y_0^{1,2}(\Omega)$ and since this limit is, in particular, in $L^{2^*}\pr{\Om}$, it must coincide with $u$ a.e.
\end{pf}

Before stating the next result, we recall a standard smoothing procedure. 
\begin{defn}
For any $U \su \R^n$ open, and any $\eps > 0$, define $U_\eps = \set{ x \in U : \dist\pr{x, \del U} > \eps}$.
\end{defn}

\begin{defn}
Define the function $\phi \in C^\iny_c\pr{\R^n}$ by
$$\phi\pr{x} = \left\{\begin{array}{ll} C\exp\pr{\frac 1 {\abs{x}^2 - 1}} & \text{ if } \abs{x} < 1 \\
0 & \text{otherwise,}  \end{array}\right.$$
where the constant $C > 0$ is chosen so that $\disp \int_{\R^n} \phi\pr{x} dx = 1$. 
We refer to $\phi$ as the standard mollifier.

For every $\eps > 0$, set $$\phi_\eps\pr{x} = \frac{C}{\eps^n} \phi\pr{\frac{x}{\eps}}.$$
We remark that for every $\eps > 0$, $\phi_\eps \in C^\iny_c\pr{\R^n}$, $\supp \phi_\eps \su B_\eps\pr{0}$ and $\disp \int_{\R^n} \phi_\eps\pr{x} dx = 1$.
\end{defn}

\begin{defn}
For any function $f$ that is locally integrable in $U$, we may define
$$f^\eps := \phi_\eps * f \; \text{ in } U_\eps.$$
That is, for every $x \in U_\eps$,
$$f^\eps\pr{x} = \int_{B_\eps\pr{0}} \phi_\eps\pr{y} f\pr{x - y} dy = \int_U \phi_\eps\pr{x-y} f\pr{y} dy.$$
\end{defn}

The proofs of the first four statements below may be found in the appendix of \cite{Eva98}, and the last one is a part of the proof of Theorem 1 in \cite{Eva98}, \S~5.3.1.

\begin{lem}[Properties of mollifiers] \label{lA.5}
Let $U$ be an arbitrary open set in $\rn$ and let $f\in L^1_{loc}(U)$. Then 
\begin{enumerate}
\item $f^\eps \in C^\iny\pr{U_\eps}$.
\item $f^\eps \to f$ a.e. as $\eps \to 0$.
\item If $f \in C\pr{U}$, then $f^\eps \to f$ uniformly on compact subsets of $U$.
\item\label{i4} If $1 \le q < \iny$ and $f \in L^q_{\loc}\pr{U}$, then $f^\eps \to f$ in $L^q_{\loc}\pr{U}$.
\item\label{i5} If, in addition, $f$ is weakly differentiable on $U$ and $Df\in L^1_{loc}(U),$ then 
$$ Df^\eps=\phi_\eps * Df \mbox{ in } U_\eps.$$
\end{enumerate}
\end{lem}

\begin{lem}\label{lA.6} If $\Om=\R^n$, then we have the following relations:
\eq{\label{eqA.2}
W_0^{1,2}(\R^n)=W^{1,2}(\R^n) \hookrightarrow Y_0^{1,2}(\R^n)=Y^{1,2}(\R^n).
}
where the inclusion is strict.
\end{lem}

\begin{pf} 
To show that $W_0^{1,2}(\R^n)=W^{1,2}(\R^n)$, we take any $u\in W^{1,2}(\rn)$, multiply it by a smooth cut-off function $\zeta_R$, for $R>0$, that is supported in $B_{2R}$ and equal to 1 on $B_R$, and convolve the product with a standard mollifier $\phi_\eps$, $\eps>0$. 
One can show that $u_{R,\eps} := \phi_\eps *(u\zeta_R) \in C_c^\infty(\RR^n)$ converges to $u$ in $L^2(\RR^n),$  and that the derivatives converge to $Du$ in $L^2(\RR^n),$ as $\eps\to 0$, $R\to \infty$. Indeed, one can see directly from the properties of the Lebesgue integral that $u_R$ belongs to $W^{1,2}(\rn)$ and converges to $u$ in the $W^{1,2}(\rn)$-norm since $u\in W^{1,2}(\rn)$. 
Now, since each $u_R$ is compactly supported, the fact that $u_{R,\eps}$ converge to $u_R$ as $\eps\to 0$ in $L^2$ is due to \eqref{i4} in Lemma~\ref{lA.5}.
The fact that each $Du_{R,\eps}$ exists and converges to $Du_R$ in $L^2(\rn)$ follows from a combination of \eqref{i5} and \eqref{i4} in Lemma~\ref{lA.5}. 
The same argument shows that $Y_0^{1,2}(\R^n)=Y^{1,2}(\R^n)$.

We only have to show that the inclusion is strict. To this end, consider
\eq{\label{eqA.3}
f(x):=\frac{1}{(1+|x|)^{n/m+1/2}}
}
with $\frac{2n}{n-1}<m<2^*$. 
A direct computation shows that 
\eq{\label{eqA.4} \norm{Df}_{L^{2}(\rn)} <\infty, \quad 
\norm{f}_{L^{2^*}(\rn)} <\infty, \quad \text{and} \quad 
\norm{f}_{L^2(\rn)}=\infty, 
} so that 
\eq{
f\in Y^{1,2}(\rn) \setminus W^{1,2}(\rn).
}
Therefore, $W^{1,2}(\rn) \subsetneq Y^{1,2}(\rn)$. 
\end{pf}

\begin{lem}\label{lA.7}
If $|\Om|<\infty$, then we have the relations
\begin{equation}\label{eqA.6}
W_0^{1,2}\pr{\Om}=Y_0^{1,2}\pr{\Om} \hookrightarrow Y^{1,2}\pr{\Om} \hookrightarrow 
W^{1,2}\pr{\Om},
\end{equation}
where the last inclusion may be an equality for certain domains (see, e.g., the next Lemma), and the norm of the embeddings $Y^{1,2}\pr{\Om} \hookrightarrow W^{1,2}\pr{\Om}$ and $Y^{1,2}_0\pr{\Om} \hookrightarrow W^{1,2}_0\pr{\Om}$ depends on $|\Omega|$.
\end{lem}

\begin{pf}
One side of the first equality in \eqref{eqA.6} is due to Lemma~\ref{lA.1}. 
On the other hand, for $u\in Y_0^{1,2}\pr{\Om}$ (or more generally, $u\in Y^{1,2}\pr{\Om}$), since $|\Om|<\infty$, we have by H\"older inequality
\eq{\label{eqA.7}
\norm{u}_{L^2\pr{\Om}} \le C_\Omega \norm{u}_{L^{2^*}\pr{\Om}}.
}
Therefore, $Y^{1,2}\pr{\Om} \hookrightarrow W^{1,2}\pr{\Om}$ and we can prove that $Y^{1,2}_0\pr{\Om} \hookrightarrow W^{1,2}_0\pr{\Om}$ roughly the same way as Lemma~\ref{lA.1}, using \eqref{eqA.7} to make sure that the sequence which is Cauchy in $Y^{1,2}_0\pr{\Om}$ is also Cauchy in $W^{1,2}_0\pr{\Om}$. Together with \eqref{eqA.1}, this finishes the proof of the lemma.
\end{pf}

Now, the opposite inclusion, $W^{1,2}\pr{\Om} \hookrightarrow Y^{1,2}\pr{\Om}$, is a question of validity of the Sobolev embedding $W^{1,2}\pr{\Om} \hookrightarrow L^{2^*}(\Om)$.  
It may fail, but it holds, e.g., for  Lipschitz domains. Following  \cite{Ste70}, we adopt the following definitions. 

We say that $\Om$ is a Lipschitz graph domain (or special Lipschitz domain) if there exists a Lipschitz function $\phi:\R^{n-1}\to \RR$ such that
\eqs{
\Om=\{(x',x_n): x_n>\phi(x')\}.
}
We say that $\Om$ is a Lipschitz domain (or a minimally smooth domain, following Stein's terminology) if there exists an $\eps>0$, $N \in \NN$, $M > 0$, and a sequence of open sets $U_1,\dots,U_m, \dots$ along with the corresponding Lipschitz functions $\phi_1,\dots,\phi_m, \dots$ defined on $\R^{n-1}$ and having a Lipschitz constant bounded by $M$, such that 
\begin{enumerate}
\item If $x\in \po$ then $B(x, \eps)\subset U_i$ for some $i$.
\item No point of $\rn$ is contained in more than $N$ of the $U_i$'s.
\item For each $i$ we have, up to rotation, that
\eqs{
U_i \cap \Om = U_i \cap \{(x',x_n): x_n>\phi_i(x')\}.
}
\end{enumerate}
If $\Omega$ satisfies the definition above and is bounded, we refer to it as a bounded Lipschitz domain.

\begin{defn} 
We say that $\Om$ is a Sobolev extension domain if there exists a linear mapping $\mathcal{E}: W^{1,2}\pr{\Om}\to W^{1,2}(\R^n)$ and a constant $C_\mathcal{E}>0$ such that for all $u\in W^{1,2}\pr{\Om}$,
\aln{
\mathcal{E} u|_{\Om} &= u \label{eqA.8}\\
\norm{\mathcal{E} u}_{W^{1,2}(\R^n)} &\le C_\mathcal{E} \norm{u}_{W^{1,2}\pr{\Om}}. \label{eqA.9}
}
\end{defn}

\begin{thm}[\cite{Ste70},VI, \S 3.3] Lipschitz domains are Sobolev extension domains.  
The constant of the corresponding extension operator, $C_\mathcal{E}$, depends on the number of graphs and their Lipschitz constants.
\end{thm}

\begin{lem}\label{lA.10}
If $\Om$ is a Sobolev extension domain, then we have the inclusion (which may be equality)
\eq{
W^{1,2}\pr{\Om} \hookrightarrow Y^{1,2}\pr{\Om},
}
with the constant in the accompanying estimate for norms depending on $C_\mathcal{E}$.
\end{lem}
\begin{pf} If $\Om$ is a Sobolev extension domain, then it follows from \eqref{eqA.8}, \eqref{eqA.9}, and Lemma~\ref{lA.6} that for all $u\in W^{1,2}\pr{\Om}$ we have ${\mathcal E} u\in W^{1,2}\pr{\rn} \hookrightarrow Y^{1,2}(\rn)$ and 
\eq{\label{eqA.11}
\norm{u}_{L^{2^*}\pr{\Om}} 
\le \norm{\mathcal{E} u}_{L^{2^*}(\R^n)} 
\le C_n (\norm{\mathcal{E} u}_{L^2(\R^n)} + \norm{D \pr{\mathcal{E}u}}_{L^2(\R^n)}) 
\le C_n C_\mathcal{E} (\norm{u}_{L^2\pr{\Om}} + \norm{D u}_{L^2\pr{\Om}}).
}
\end{pf}

\begin{cor}\label{cA.11} 
If $\Om$ is a bounded Lipschitz domain, then we have the following 
relations:
\eq{
W_0^{1,2}\pr{\Om}=Y_0^{1,2}\pr{\Om} \hookrightarrow Y^{1,2}\pr{\Om} =W^{1,2}\pr{\Om}.
}
\end{cor}

\begin{lem}
If $\Om$ is a Lipschitz graph domain, then we have the following 
relations:
\eqs{
W_0^{1,2}\pr{\Om} \hookrightarrow Y_0^{1,2}\pr{\Om} \stackrel{not\, 
comparable}{\longleftrightarrow} W^{1,2}\pr{\Om} \hookrightarrow Y^{1,2}\pr{\Om},
}
where the inclusions cannot be made equalities.
\end{lem}
\begin{pf}
The inclusions are given by Lemmas~\ref{lA.1} and \ref{lA.10}.

Without loss of generality, assume $0\in \partial \Om$.  
Let $\Gamma\subset \Om$ be a cone with its vertex at $0$ and its axis in the $x_n$-direction.
Define
\eqs{
\gamma:=\{x\in \Gamma: \dist(x,\partial \Gamma)>1\}.
}
Let $\zeta \in C^\iny_c\pr{\Ga}$ be a smooth cutoff function such that $\zeta \equiv 1$ in $\gamma$, 
$\zeta \equiv 0$ in $\Om \setminus \Gamma$, and $|D \zeta| \le C$.  Note that 
$\zeta \equiv 0$ on $\partial \Om$.
Let $f(x)$ be as in the counterexample given by \eqref{eqA.3} with $\frac{2n}{n-1}<m<2^*$.  
Consider
\eqs{
g(x):=\zeta(x) f(x).
}
Then, a computation similar to that which gives \eqref{eqA.4} also gives
\eqs{
g\in L^{2^*}\pr{\Om} \setminus L^2\pr{\Om}.
}
It remains only to show that $Dg\in L^2\pr{\Gamma \setminus \gamma}$.  
Since the cones $\gamma$ and $\Gamma$ have equal aperture, we have for sufficiently large $s$,
\eqs{|\pr{\Gamma\setminus \gamma} \cap \partial B_s\pr{0}| \le C s^{n-2}.}

Consequently, a direct computation shows
\eqs{\norm{f}_{L^2\pr{\Gamma \setminus \gamma}} <\infty.
}

Notice that for $t>1$, $\pr{\Gamma \setminus \gamma} \cap \{x_n=t\}$ forms a $\pr{n-1}$-dimensional annulus of width 1.  Thus, we have
$$|\pr{\Gamma \setminus \gamma} \cap \partial B_s\pr{0}| \le Cs^{n-2}, \quad \forall 
s> 1,$$
and
$$\norm{f}_{L^2\pr{\Gamma \setminus \gamma}}  \le \int_{B_1} |f|^2 + C \int_{1}^\infty |f\pr{s}|^2 s^{n-2} \, ds \le C + C \int_1^\infty s^{\pr{1-2/m}n-3} \, ds < \infty,$$
where in the last step we have used that $\pr{1-2/m}n<2$.

Therefore,
\eqs{
\int_{\Om} |D g|^2 \le 2 \int_{\Gamma \setminus \gamma} |D \zeta|^2 |f|^2 + 2 
\int_{\Gamma} |\zeta^2| |D f|^2 \le C \left[\norm{f}_{L^2\pr{\Gamma \setminus \gamma}} + 
\int_{\Om} |D f|^2 \right]<\infty
} 
so that $g\in Y^{1,2}(\Omega)$. 
As in the proof of Lemma~\ref{lA.6}, multiplying $g$ by smooth cut-offs $\zeta_R$, we obtain a sequence of $C_c^\infty(\Omega)$ functions that approximate $g$ in the $Y^{1,2}(\Omega)$-norm. 
Thus, $g\in Y_0^{1,2}\pr{\Om}$ or more precisely, 
\eqs{
g\in Y_0^{1,2}\pr{\Om} \setminus W^{1,2}\pr{\Om}.
}
Therefore, $Y_0^{1,2}\pr{\Om} \not\subseteq W^{1,2}\pr{\Om}$. The fact that the opposite inclusion fails is obvious as elements of $W^{1,2}(\Omega)$ do not need to have trace zero on $\po$ (in the sense of approximation by smooth compactly supported functions). 
\end{pf}

\section{The auxiliary function $m\pr{x, V}$}
\label{AppB}

Within this section, we will quote a number of results from \cite{She95}.
Other versions of these lemmas and definitions appeared in \cite{She94} and \cite{She96}, and are related to the ideas of Fefferman and Phong \cite{Fef83}.
We omit the proofs in our exposition.

Recall that $V \in B_p$, $1 < p < \iny$, if there exists a constant $C$ so that for any ball $B \su \R^n$,
\begin{align}
\pr{\fint_B V\pr{x}^p dx}^{1/p} \le C \fint_B V\pr{x} dx.
\label{eqB.1}
\end{align}
If $V \in B_p$, then $V$ is a Muckenhoupt $A_\iny$ weight function \cite{Ste93}.
Therefore, $V\pr{x} dx$ is a doubling measure.
That is, there exists a constant $C_0$ such that
\begin{align*}
\int_{B\pr{x, 2r}} V\pr{y} dy \le C_0 \int_{B\pr{x, r}} V\pr{y} dy.
\end{align*}
This fact is very useful in establishing the following results.
We now define 
\begin{align}
\psi\pr{x, r; V} = \frac{1}{r^{n-2}} \int_{B\pr{x,r}} V\pr{y} dy.
\label{eqB.2}
\end{align}
We will at times use the shorter notation $\psi\pr{x,r}$ when it is understood that this function is associated to $V$.

We assume that $V \in B_p$ for some $p \in\brp{ \frac{n}{2}, \iny}$.
In fact, it follows from the self-improvement result for reverse H\"older classes that $V \in B_p$ for some $p \in\pr{\frac{n}{2}, \iny}$ \cite{Geh73}.
Therefore, we will assume throughout that the inequality is strict.

\begin{lem}[Lemma 1.2, \cite{She95}]
If $V \in B_p$, then there exists a constant $C > 0$ so that for any $0 < r < R  < \iny$,
\begin{align*}
\psi\pr{x, r; V} \le C \pr{\frac{r}{R}}^{2 - \frac{n}{p}} \psi\pr{x, R; V}.
\end{align*}
\label{lB.1}
\end{lem}

The proof of Lemma~\ref{lB.1} uses the reverse H\"older inequality \eqref{eqB.1} as well the H\"older inequality.

As $V \ge 0$, then for every $x \in \R^n$, either there exists $r > 0$ so that $\psi\pr{x, r; V} > 0$ or $V \equiv 0$ a.e. in $\R^n$.
For now, we assume that $V \not\equiv 0$.
Since $p > \frac{n}{2}$, the power $2 - \frac n p > 0$ and 
\begin{align}
& \lim_{r \to 0^+} \psi\pr{x, r; V} = 0, \label{eqB.3} \\
& \lim_{r \to \iny} \psi\pr{x, r; V} = \iny \label{eqB.4}.
\end{align}
This leads to the following definition.

\begin{defn}
For $x \in \R^n$, the function $m\pr{x, V}$ is defined by
\begin{align}
\frac{1}{m\pr{x, V}} = \sup_{r > 0} \set{ r : \psi\pr{x,r; V} \le 1}.
\label{eqB.5}
\end{align}
\end{defn}

It follows from \eqref{eqB.3} and \eqref{eqB.4} that $0 < m\pr{x,V} < \iny$ and for every $x \in \R^n$
\begin{equation}
\psi\pr{x, \frac{1}{m\pr{x, V}}; V} = 1.
\label{eqB.6}
\end{equation}
Furthermore, from Lemma~\ref{lB.1}, if $\psi\pr{x, r; V} \sim 1$, then $r \sim \frac{1}{m\pr{x, V}}$.
If $\disp r = \frac{1}{m\pr{x, V}}$ then $\disp \fint_{B\pr{x,r}} V\pr{y} dy = \frac{1}{\om_n r^2}$, where $\om_n$ is the measure of the unit ball in $\R^n$.

\begin{lem}[Lemma 1.4, \cite{She95}]\label{lB.3}
There exist constants $C, c, k_0 > 0$ so that for any $x, y \in \R^n$,
\begin{enumerate}
\item[(a)] $\disp m\pr{x, V} \sim m\pr{y, V}$ if $\disp \abs{x - y} \le \frac{C}{m\pr{x, V}}$,
\item[(b)] $\disp m\pr{y, V} \le C \brac{1 + \abs{x - y}m\pr{x, V}}^{k_0} m\pr{x, V}$,
\item[(c)] $\disp m\pr{y, V} \ge \frac{c \, m\pr{x, V}}{\brac{1 + \abs{x -y}m\pr{x, V}}^{k_0/\pr{k_0+1}}}.$
\end{enumerate}
\end{lem}

\begin{cor}[Corollary 1.5, \cite{She95}]
There exist constants $C, c, k_0 > 0$ so that for any $x, y \in \R^n$,
\begin{align*}
c \brac{1 + \abs{x-y} m\pr{y, V}}^{1/\pr{k_0 +1}} 
&\le 1 + \abs{x - y} m\pr{x, V} 
\le C \brac{1 + \abs{x - y} m\pr{y, V} }^{k_0 + 1}.
\end{align*}
\end{cor}

\begin{rem}\label{rB.5} 
Another important consequence of Lemma~\ref{lB.3} is that $m(x, V)$ is locally bounded from above and below. 
More specifically, for any bounded open set $U \su \rn$, there exists a constant $C=C_U>0$, depending on $U$ and on the constants in  Lemma~\ref{lB.3}, such that 
$$ \frac 1C \leq m(x, V) \leq C, \quad\mbox{ for any } x\in U.$$
Indeed, the collection $\set{B_{1/m\pr{x,V}}\pr{x}}_{x \in U}$ is an open covering of $\overline U$.
Since $\overline{U}$ is compact, then there exists a finite collection of points, $x_1, \ldots, x_M$, such that $\disp \overline{U} \su \bigcup_{i = 1}^M B_{1/m\pr{x_i,V}}\pr{x_i}$.
It follows from Lemma~\ref{lB.3} that there exists $C > 0$, depending on $V$, $n$, so that for any $x \in \overline{U}$, $C^{-1} \min\set{ m\pr{x_i, V}}_{i=1}^M \le m\pr{x, V} \le C \max\set{m\pr{x_i, V}}_{i = 1}^M$.
In other words, $m\pr{x, V}$ is bounded above and below on $\overline{U}$, and consequently on $U$.
\end{rem}

\begin{lem}[Lemma 1.8, \cite{She95}]
There exist constants $C, k_0 > 0$ so that if $\disp R \ge \frac{1}{m\pr{x, V}}$
\begin{align*}
\frac{1}{R^{n-2}}\int_{B\pr{x, R}} V\pr{y} dy \le C \brac{R m\pr{x, V}}^{k_0}.
\end{align*}
\end{lem}

The last lemma that we will quote from \cite{She95} is the Fefferman-Phong inequality.

\begin{lem}[Lemma 1.9 \cite{She95}, see also \cite{Fef83}]
If $u \in C^1_c\pr{\R^n}$, then 
\begin{align*}
\int_{\R^n} \abs{u\pr{x}}^2 m\pr{x, V}^2 dx \le C \brac{ \int_{\R^n} \abs{D u \pr{x}}^2 dx + \int_{\R^n} \abs{u\pr{x}}^2 V\pr{x} dx }.
\end{align*}
\label{lB.7}
\end{lem}

If $V \equiv 0$, then $m\pr{x, V} \equiv 0$ and the previous four results are automatically satisfied.

\section{The weighted Sobolev space $W^{1,2}_{V}$}
\label{AppC}

Recall that we define $W_V^{1,2}\pr{\Om}$ as the family of all weakly differentiable functions $u \in L^{2}\pr{\Om, m\pr{x, V}^2 dx}$ whose weak derivatives are functions in $L^2\pr{\Om, dx}$.
The norm and inner product on $W_V^{1,2}\pr{\Om}$ are given by
\begin{align*}
\norm{u}^2_{W_V^{1,2}\pr{\Om}} &:= \norm{u \, m\pr{\cdot, V}}^2_{L^{2}\pr{\Om}} + \norm{D u}^2_{L^2\pr{\Om}} \\
\innp{u, v}_{W_V^{1,2}\pr{\Om}} &:= \innp{u m\pr{\cdot, V}, v m\pr{\cdot, V}}_{L^{2}\pr{\Om}} + \innp{D u, Dv}_{L^2\pr{\Om}} .
\end{align*}
$W^{1,2}_{0, V}\pr{\Om}$ is defined as the closure of $C^\iny_c\pr{\Om}$ in $W_V^{1,2}\pr{\Om}$.
Recall also the analogously defined spaces $\hat W_V^{1,2}\pr{\Om}$ and $\hat W_{0,V}^{1,2}\pr{\Om}$ with $V(x)$ in place of $m(x, V)^2$ in the norms, see Remark~\ref{r7.1}. 
Here we prove the claim stated in Remark~\ref{r7.1} to the effect that $\hat W_{0,V}^{1,2}\pr{\Om} = W_{0,V}^{1,2}\pr{\Om}$ for any open set $\Omega\subset \rn$.  

\begin{rem}\label{rC.1} 
Assume that $V \in B_p$ for some $p \in \pr{\frac{n}{2}, \iny}$.
First of all, we observe that by Remark~\ref{rB.5}, for any bounded open set $U\subset \rn$, the spaces $W^{1,2}_V(U)$ and $W^{1,2}(U)$ coincide, albeit the norms are only comparable modulo multiplicative constants that depend on $U$.
\end{rem}

First we prove that the weighted Sobolev spaces are indeed Hilbert spaces as defined.

\begin{lem}\label{lC.2} 
Let $\Om\subset \R^n$ be open and let $\eta\in L^1_{\loc}\pr{\Om}$ be real-valued with $\eta>0$ a.e.
The inner product
\eqs{\langle u, v \rangle_{L^2\pr{\Om, \eta\pr{x} dx}} = \int_{\Om} u\pr{x} \overline{v\pr{x}} \eta\pr{x} \, dx,}
makes $L^2\pr{\Om, \eta\pr{x} dx}$ a Hilbert space.
\end{lem}

\begin{pf}
It is easy to check that $L^2\pr{\Om, \eta\pr{x} dx}$ is a vector space and $\langle \cdot, \cdot, \rangle_{L^2\pr{\Om, \eta\pr{x} dx}}$ defines an inner product that generates a norm on the space.

To prove completeness, it suffices to show that $L^2\pr{\Om, \eta\pr{x} dx}$ is unitarily equivalent to $L^2\pr{\Om}$.  
Consider the map
\eqs{\phi:L^2\pr{\Om} \rightarrow L^2\pr{\Om, \eta\pr{x} dx}: f \mapsto f \eta^{-1/2}.}
For $f\in L^2\pr{\Om}$, we have
\eqs{\norm{\phi\pr{f}}_{L^2\pr{\Om, \eta\pr{x} dx}} = \int_\Om \pr{f \eta^{-1/2}} \overline{\pr{f \eta^{-1/2}}} \eta = \norm{f}_{L^2\pr{\Om}}.}
Thus, $\phi$ is injective.
For $g\in L^2\pr{\Om, \eta\pr{x} dx}$, take $f=g\eta^{1/2}$.  
Then $f\in L^2\pr{\Om}$ and $\phi\pr{f}=g$.
Thus, $\phi$ is surjective.
Finally, we check
\eqs{\langle \phi\pr{f}, \phi\pr{g} \rangle_{L^2\pr{\Om, \eta\pr{x} dx}} 
= \int_{\Om} \pr{f \eta^{-1/2}} \overline{\pr{g \eta^{-1/2}}} \eta 
= \int_\Om f \overline{g} = \langle f, g \rangle_{L^2\pr{\Om}}.}
\end{pf}

\begin{lem}Let $\Om\subset \R^n$ be open and let $\eta\in L^{1}_{\loc}\pr{\Om}$ be real-valued with $\eta>0$ a.e.
Define the space $W^{1,2}_\eta\pr{\Om}$ as a collection of functions in $L^2(\Omega, \eta(x)dx)$ that are weakly differentiable in $\Omega$ with the weak gradient in $L^2(\Omega)$.
The inner product
\eqs{\langle u, v \rangle_{W^{1,2}_\eta\pr{\Om}} = \int_\Om D u \cdot D v + \int_\Om u \, v \, \eta^{1/2}}
makes $W^{1,2}_\eta(\Om)$ a Hilbert space.
\end{lem}
\begin{pf}
 A quick computation verifies that $\langle \cdot, \cdot \rangle_{W^{1,2}_\eta(\Om)}$ is an inner product generating the norm on the space.
 It remains only to show completeness.
 
 Let $\{u_k\}$ be a Cauchy sequence in $W^{1,2}_\eta\pr{\Om}$.
 Then $\{u_k\}$ is Cauchy in $L^2\pr{\Om, \eta\pr{x} dx}$, so by Lemma~\ref{lC.2} there exists $u \in L^2\pr{\Om, \eta\pr{x} dx}$ such that
 \eq{\label{eqC.1} u_k \rightarrow u \quad \text{ in } L^2\pr{\Om, \eta\pr{x} dx}.}
 Furthermore, for $\alpha=1,\dots,n$, $\{D_\alpha u_k\}$ is Cauchy in $L^2\pr{\Om}$, so there exists $v_\alpha \in L^2\pr{\Om}$ such that
 \eq{\label{eqC.2} D_\alpha u_k \rightarrow v_\alpha \quad \text{ in } L^2\pr{\Om}.}
 It remains to show that $v_\alpha = D_\alpha u$.
 Let $\zeta \in C_c^\infty\pr{\Om}$.
 We need to show that
 \eqs{\int_\Om v_\alpha \zeta = -\int_\Om u D \zeta.}
 First, suppose $\supp(\zeta) \subset B$, where $B$ is an open ball that is compactly contained in $\Om$.
 The Poincar\'e inequality yields
 \eqs{\norm{\pr{u_k - \fint_B u_k} - \pr{u_j - \fint_{B} u_j}}_{L^2\pr{B}}=\norm{u_k - u_j - \fint_{B} (u_k-u_j)}_{L^2\pr{B}} \le C \norm{D (u_k- u_j)}_{L^2\pr{B}}.}
Therefore, $\{u_k - c_k\}$ is Cauchy in $L^2\pr{B}$, with $c_k = \fint_B u_k$.
Thus, there exists $\tilde{u}\in L^2\pr{B}$ such that
\eq{\label{eqC.3}u_k - c_k \rightarrow \tilde{u} \quad \text{in } L^2\pr{B}.}
By H\"older's inequality, 
\begin{align}\label{eqC.4}
\norm{(u_k - c_k - \tilde{u}) \eta^{1/2}}_{L^{1}\pr{B}} 
&\le \norm{u_k-c_k-\tilde{u}}_{L^2\pr{B}} \norm{\eta^{1/2}}_{L^{2}\pr{B}} \\
&= \norm{u_k-c_k-\tilde{u}}_{L^2\pr{B}} \norm{\eta}_{L^{1}\pr{B}}^{1/2}\rightarrow 0.
\nonumber
\end{align}
Therefore, by \eqref{eqC.4},
\eq{\label{eqC.5}(u_k - c_k)\eta^{1/2} \rightarrow \tilde{u} \eta^{1/2} \quad \text{in } L^{1}\pr{B}.}
By \eqref{eqC.1}, $u_k \eta^{1/2} \rightarrow u \eta^{1/2}$ in $L^2\pr{B}$, so it follows that
\eq{\label{eqC.6}u_k \eta^{1/2} \rightarrow u \eta^{1/2} \quad \text{in } L^{1}\pr{B}.}
Combining the previous two results shows that
\eqs{c_k \eta^{1/2} \rightarrow \pr{u - \tilde{u}} \eta^{1/2} \quad \text{in } L^{1}\pr{B}.}
Since each $c_k$ is a constant, it follows that $\disp \lim_{k \to \iny}c_k = c$, where $c$ is some fixed constant.
This fact, in combination with \eqref{eqC.5}, implies that
\eqs{u_k \eta^{1/2} \rightarrow \pr{\tilde{u} + c} \eta^{1/2} \quad \text{in } L^{1}\pr{B}.}
With \eqref{eqC.6}, using that $\eta$ is almost everywhere non-vanishing, we conclude that $\tilde{u}+c = u$ a.e. in $B$.
From \eqref{eqC.3} and the fact that $\set{c_k}$ is a convergent sequence of real numbers, we have
\eq{\label{eqC.7} u_k \rightarrow \tilde{u} + c = u \quad \text{in } L^2\pr{B}.}
Therefore, by \eqref{eqC.2} and \eqref{eqC.7},
\eq{\label{eqC.8}\int_B v_\al \zeta = \lim_{k\rightarrow \infty} \int_B D_\al u_k \zeta = -\lim_{k\rightarrow \infty} \int_B u_k D_{\al} \zeta =- \int_B u D_{\al} \zeta.}
Now, for any $\zeta \in C_c^\infty\pr{\Om}$, we can cover $\supp(\zeta)$ with finitely many balls, $\{B_i\}$, with each $B_i$ compactly contained in $\Om$.
Using a partition of unity argument and the result \eqref{eqC.8}, we obtain the desired equality.
\end{pf}

\begin{cor}
 Let $\Om\subset \R^n$ be an open set.  The spaces $W^{1,2}_V\pr{\Om}$, $\hat{W}^{1,2}_V\pr{\Om}$, $W^{1,2}_{0,V}\pr{\Om}$, and $\hat{W}^{1,2}_{0,V}\pr{\Om}$ are Hilbert spaces.
\end{cor}
\begin{proof}
 This follows directly from the previous lemma and the fact that $W^{1,2}_{0,V}\pr{\Om}$ and $\hat{W}^{1,2}_{0,V}\pr{\Om}$ are defined as the closure of $C_c^\infty\pr{\Om}$ in their respective spaces.
\end{proof}

The following lemma shows an important relationship between $W^{1,2}_V\pr{\R^n}$ and $\hat{W}^{1,2}_V\pr{\R^n}$.

\begin{lem}\label{lC.5}
Assume that $V \in B_p$ for some $p > \frac{n}{2}$.
Then for any $u \in W_V^{1,2}\pr{\R^n}$,
\begin{equation}\label{eqC.9}
\int_{\R^n} V\pr{x} \, \abs{u\pr{x}}^2 dx
\le C_{V, n} \pr{\int_{\R^n} \abs{Du\pr{x}}^2 dx + \int_{\R^n} \abs{u\pr{x}}^2 m\pr{x, V}^2 dx }
= C_{V, n} \norm{u}^2_{{W}_V^{1,2}\pr{\R^n}}.
\end{equation}
Conversely, for any $u \in \hat W_V^{1,2}\pr{\R^n}$
\begin{equation}\label{eqC.10}
\int_{\R^n} \abs{u\pr{x}}^2 m\pr{x, V}^2 dx
\le C_{V, n} \pr{\int_{\R^n} \abs{Du\pr{x}}^2 dx + \int_{\R^n} \abs{u\pr{x}}^2 V(x)\, dx }
= C_{V, n} \norm{u}^2_{\hat{W}_V^{1,2}\pr{\R^n}}.
\end{equation}
In other words, $W_V^{1,2}\pr{\R^n}=\hat{W}_V^{1,2}\pr{\R^n}$. 
\end{lem}

\begin{proof} 
This is essentially Theorem~1.13 in \cite{She99}. 
We only remark that our $V\, dx$ satisfies the conditions of $d\mu$ in the aforementioned Theorem by Remark~0.10 in  \cite{She99}, and that the functions with $D u\in L^2(\rn)$ are $L^2_{loc}(\rn)$ -- this is a standard part of the proof of the Poincar\'e inequality (see, e.g., \cite{Maz11}, 1.1.2). 
\end{proof}

If $\Om \subset \R^n$ is open and connected, then a similar relationship holds for $W^{1,2}_{0, V}\pr{\Om}$ and $\hat{W}^{1,2}_{0,V}\pr{\Om}$ and we have the following result.

\begin{lem} \label{lC.6}
Assume that $V \in B_p$ for some $p > \frac{n}{2}$. Then for any open set $\Omega \subset \rn $ we have $W_{0,V}^{1,2}\pr{\Om}=\hat W_{0,V}^{1,2}\pr{\Om}$, and $\|\cdot\|_{W_{0,V}^{1,2}\pr{\Om}}\approx\|\cdot\|_{\hat W_{0,V}^{1,2}\pr{\Om}}$ with implicit constants depending on dimension and the $B_p$ constant of $V$ only. 
\end{lem}

\begin{proof}
Let $u \in W_{0,V}^{1,2}\pr{\Om}$.
By definition, there exists $u_i \in C^\iny_c\pr{\Om}$ so that $\disp \lim_{i \to \iny} \norm{u_i - u}_{W^{1,2}_V\pr{\Om}} = 0$.
Applying Lemma~\ref{lC.5} to $u_i-u_k$, we deduce that the sequence $\{u_i\}_{i=1}^\infty$ is Cauchy in $\hat W_{0,V}^{1,2}\pr{\Om}$.
Hence, it has a limit in $\hat W_{0,V}^{1,2}\pr{\Om}$ and this limit must coincide with $u$ a.e. since $V>0$ a.e. and $m(x, V)>0$ for all $x\in \rn$. 
Applying again Lemma~\ref{lC.5}, we deduce the desired control of norms. 
The same argument works in the converse direction.
\end{proof}

\section{Smoothing and Approximations}
\label{AppD}

Here we build on Lemma~\ref{lA.5} and collect some results that are related to approximation by smooth functions. 

\begin{lem}[Local approximation by smooth functions]
Let $U \su \R^n$ be open.
Let $\bF\pr{U}$ be either $Y^{1,2}\pr{U}$, $W^{1,2}\pr{U}$ or $W^{1,2}_V\pr{U}$.
Assume that $u \in \bF\pr{U}$, and set $u^\eps = \phi_\eps * u$ in $U_\eps$.
Then $u^\eps \in C^\iny\pr{U_\eps}$ for each $\eps > 0$ and $u^\eps \to u$ in $\bF_{\loc}\pr{U}$ as $\eps \to 0$.
\label{lD.1}
\end{lem}

The case of $\bF\pr{U} = W^{1,2}\pr{U}$ appears in \cite{Eva98}, the case  $\bF\pr{U} = W^{1,2}_V\pr{U}$ is the exact same statement due to the local nature of the result and Remark~\ref{rC.1}. 
The case of $\bF\pr{U} = Y^{1,2}\pr{U}$ is a slight modification of the aforementioned proof in \cite{Eva98}, and we omit it.

\begin{lem}[Global approximation by smooth functions]
Assume that $U$ is bounded.
Let $\bF\pr{U}$ be either $Y^{1,2}\pr{U}$, $W^{1,2}\pr{U}$ or $W^{1,2}_V\pr{U}$.
 If $u \in \bF\pr{U}$, then there exists a sequence $\set{u_k}_{k=1}^\iny \su C^\iny\pr{U} \cap \bF\pr{U}$ such that $\disp \lim_{k \to \iny} u_k = u$ in $\bF\pr{U}$.
\label{lD.2}
\end{lem}

When $\bF\pr{U} = W^{1,2}\pr{U}$, this is Theorem 2 from \S 5.3.2 of \cite{Eva98}, the case $\bF\pr{U} = W_V^{1,2}\pr{U}$ is the same due to boundedness of $U$ and Remark~\ref{rC.1}, and 
the case $\bF\pr{U} = Y^{1,2}\pr{U}$ is proved in an analogous way. However, we outline the proof here as some elements of it will be useful down the road.

\begin{proof}
We have that $\disp U = \bigcup_{i=1}^\iny U_i$ where $U_i = \set{x \in U : \dist\pr{x, \del U} > \frac 1 i}$.
Set $W_i = U_{i+3} - \overline{U}_{i+1}$.
Choose $W_0 \Subset U$ so that $\disp U = \bigcup_{i=0}^\iny W_i$.
Let $ \set{\zeta_i}_{0 =1}^\iny$ be a smooth partition of unity subordinate to $\set{W_i}_{i =1}^\iny$.
In other words, for each $i$, $0 \le \zeta_i \le 1$, $\zeta_i \in C^\iny_c\pr{W_i}$, and $\disp \sum_{i=1}^\iny \zeta_i = 1$ on $U$.
Let $u \in \bF\pr{U}$.
Since each $\zeta_i \in C^\iny_c\pr{U}$, then $\supp \pr{u \zeta_i} \subset W_i$ and by a straightforward argument similar to the proof of Lemma 1(iv) from \S 5.2 of \cite{Eva98}, $u \zeta_i \in \bF\pr{U}$.

For each $i = 0, 1 , \ldots$, define $X_i = U_{i+4} - \overline{U}_i \supset W_i$.
Fix $\de > 0$.
Then, for each $i$, choose $\eps_i> 0$ so small that $u^i := \phi_{\eps_i} * \pr{u \zeta_i}$ is such that $\supp{u^i} \su X_i$ and $\norm{u^i - u \zeta_i}_{\bF\pr{U}} \le \de 2^{-i-1}$.
The second property is guaranteed by the Lemma~\ref{lD.1}.

Define $\disp v := \sum_{i=1}^\iny u^i$.
For any open set $W \Subset U$, there are at most finitely many terms in the sum for $v$, so it follows that $v \in C^\iny\pr{W}$.
As $\disp u = \sum_{i=1}^\iny u \zeta_i$ then for each $W \Subset U$, we have that
\begin{align*}
\norm{v - u}_{\bF\pr{W}}
&\le  \sum_{i=0}^\iny \norm{u^i - u \zeta_i}_{\bF\pr{W}}
\le \de \sum_{i=0}^\iny 2^{-i-1} = \de.
\end{align*}
By taking the supremum over all sets $W \Subset U$, we conclude that $\norm{v - u}_{\bF\pr{U}} \le \de$, and the conclusion of the lemma follows.
\end{proof}

Since the mollification of an a.e. non-negative function is also non-negative, the following corollary is true.

\begin{cor}[Global approximation by smooth non-negative functions]
Assume that $U$ is bounded.
Let $\bF\pr{U}$ be either $Y^{1,2}\pr{U}$, $W^{1,2}\pr{U}$ or $W^{1,2}_V\pr{U}$.
If $u \in \bF\pr{U}$ is non-negative a.e., then there exists a sequence $\set{u_k}_{k=1}^\iny \su C^\iny\pr{U} \cap \bF\pr{U}$ of non-negative functions such that $\disp \lim_{k \to \iny} u_k = u$ in $\bF\pr{U}$.
\label{cD.3}
\end{cor}

Finally, if $u$ is compactly supported in $U$, then it follows from the previous lemma that $u$ may be approximated by smooth compactly supported functions.

\begin{lem}[Global approximation by smooth compactly supported functions]
Assume that $U$ is bounded.
Let $\bF\pr{U}$ be either $Y^{1,2}\pr{U}$, $W^{1,2}\pr{U}$ or $W^{1,2}_V\pr{U}$.
If $u \in \bF\pr{U}$ and $\supp u \Subset U$, then there exists a sequence $\set{u_k}_{k=1}^\iny \su C_c^\iny\pr{U} \cap \bF\pr{U}$ such that $\disp \lim_{k \to \iny} u_k = u$ in $\bF\pr{U}$.
\label{lD.4}
\end{lem}

We sketch the proof of the lemma.

\begin{proof}
Define $U_i, W_i, \zeta_i$ as in the proof of Lemma~\ref{lD.2} and conclude as before that each $u \zeta_i \in \bF\pr{U}$.
Since $\supp u \Subset U$, and $\disp U = \bigcup_{i=0}^\iny W_i$, then there exists $M \in \N$ so that $\disp \supp u \su \bigcup_{i=0}^M W_i$.
Therefore, $\disp u = \sum_{i=0}^M u \zeta_i$.
Then (for $i = 0, \ldots, M$) define $X_i$, $u^i$ as before so that $\supp{u^i} \su X_i$ and $\norm{u^i - u \zeta_i}_{\bF\pr{U}} \le \de 2^{-i-1}$ and set $\disp v := \sum_{i=1}^M u^i$.
Since each $u_i \in C^\iny_c\pr{U}$, then $v \in C^\iny_c\pr{U}$ as well.
Moreover,
\begin{align*}
\norm{v - u}_{\bF\pr{U}}
&\le  \sum_{i=0}^M \norm{u^i - u \zeta_i}_{\bF\pr{U}}
\le \de \sum_{i=0}^M 2^{-i-1} = \de
\end{align*}
and the conclusion follows.
\end{proof}

\end{appendix}

\bibliography{refs}
\bibliographystyle{alpha}

\end{document}